\documentclass{amsart}
\usepackage{amsmath,amssymb,hyperref,mathrsfs,amsfonts,graphicx}
\newtheorem{The}{Theorem}[section]
\newtheorem{Cor}[The]{Corollary}
\newtheorem{Lem}[The]{Lemma}
\newtheorem{Con}[The]{Conjecture}
\newtheorem{Pro}[The]{Proposition}
\newtheorem{Clm}[The]{Claim}
\newtheorem{demo}[The]{Demonstration }
\theoremstyle{definition}
\newtheorem{defn}[The]{Definition}
\theoremstyle{remark}
\newtheorem{Rem}[The]{Remark}

\numberwithin{equation}{section}
\theoremstyle{Question}

\newcommand{\I}{\mathbb{I}}
\title{Asymptotic trajectories of KAM torus}
\author{Jianlu Zhang$\dagger$, Chong-Qing Cheng$\ddagger$}

\address{Department of Mathematics, Nanjing University\\Nanjing, China, 210093}

\thanks{$\dagger$Email:\;jellychung1987@gmail.com}
\thanks{$\ddagger$Email:\;chengcq@nju.edu.cn}
\subjclass{Primary 37Jxx; Secondary 37Dxx}
\keywords{Arnold Diffusion, KAM torus, Aubry Mather Theory, Variational Method, Asymptotic Trajectory}
\date{\today}
\begin{document}
\maketitle
\begin{abstract}
  In this paper we construct a certain type of nearly integrable systems of two and a half degrees of freedom:
  \[
  H(p,q,t)=h(p)+\epsilon f(p,q,t),\quad (q,p)\in T^{*}\mathbb{T}^2,t\in \mathbb{S}^1=\mathbb{R}/\mathbb{Z},
  \]
  with a self-similar and weak-coupled $f(p,q,t)$ and $h(p)$ strictly convex. For a given Diophantine rotation vector $\vec{\omega}$, we can find asymptotic orbits towards the KAM torus $\mathcal{T}_{\omega}$, which persists owing to the classical KAM theory, as long as $\epsilon\ll1$ sufficiently small and $f\in C^r(T^{*}\mathbb{T}^2\times\mathbb{S}^1,\mathbb{R})$ properly smooth.

The construction bases on several new approaches developed in \cite{Ch}, where he solved the generic existence of diffusion orbits of a priori stable systems. As an expansion of Arnold Diffusion problem, our result supplies several useful viewpoints for the construction of preciser diffusion orbits.
\end{abstract}

\section{Introduction}
\subsection{Statement of Main Result}

For a nearly integrable systems
\begin{equation}\label{1}
  H(p,q)=h(p)+\epsilon f(q,p),\quad(q,p)\in T^*\mathbb{T}^n,
\end{equation}

KAM theory assures that the set of KAM tori occupies a rather large-measured part in phase space, but it's still a topologically sparse set. It indicates that for a system with a freedom not bigger than two degrees, every orbit will be confined in the `cells' formed by energy surface and KAM tori, and the oscillations of action variables do not exceed a quantity of order $\mathcal{O}(\sqrt{\epsilon})$\cite{Ar1}.
This disproved the ergodic hypothesis formulated by Maxwell and Boltzmann:\\

{\bf For a typical Hamiltonian on a typical energy surface, all but a set of zero measure of initial conditions, have trajectories covering densely this energy surface itself.}\\

However, if the number of degrees of freedom $n$ greater than two, the n-dimensional invariant tori can not divide each (2n-1)-dimensional energy surface into disconnected parts and the action variables of trajectories not laying on the tori are unrestrained. So it's reasonable to modify the ergodic hypothesis and raise:
\begin{Con}\upshape{(Quasi-ergodic Hypothesis\cite{Bi,E})}\label{Q H}
  For a typical Hamiltonian on a typical energy surface, there exists at least one dense trajectory.
\end{Con}
 The first progress towards this direction is made by V. Arnold \cite{Ar2} in 1964. In his paper, he constructed a 2.5 degrees of freedom system which has an unperturbed normally hypobolic invariant cylinder (NHIC) and a `homoclinic overlap' structure. This `homoclinic overlap' structure assures the existence of heteroclinic trajectories towards different lower-dimensional tori located in the NHIC, and along these heteroclinic trajectories the slow action variable of shadowing orbits changes of $\mathcal{O}(1)$ in a rather long time.

 We can simplify this mechanism and raise the following:
\begin{Con}\upshape{(Arnold Diffusion\cite{Ar1})}\label{A D}
Typical integrable Hamiltonian systems with n degrees of freedom $(n\geq2.5)$ is topologically instable: through an arbitrarily small neighborhood of any point there passes a phase trajectory whose slow variables drift away from the initial value by a quantity of order $\mathcal{O}(1)$.
\end{Con}

Celebrated progress has been made through the past twenty years and we can give a quite positive answer to this Conjecture \ref{A D}. For results of {\it a priori} unstable case, the readers can see \cite{CY1,CY2,D L S,Tr} and \cite{Ch,Mat1,K Z} of {\it a priori} stable case. But a definite answer whether Conjecture \ref{Q H} is right or wrong is still far from the reach of modern dynamical theory.\\

One instinctive idea towards Quasi-ergodic Hypothesis is to use the same method in solving Conjecture \ref{A D} to construct trajectories to fill the topologically open-dense complement of KAM tori. So to find asymptotic trajectories of KAM tori is the first difficulty we must overcome. The first exploration was made by R. Duady\cite{Du}:
\begin{The}
  For a fixed Diophantine rotation vector $\vec{\omega}$, there exists a nearly integrable system $H_{\epsilon}(p,q)\big{|}_{(p,q)\in T^{*}\mathbb{T}^3}$ which is $C^{\infty}$-approached to an integrable system $H_0(p)$, such that for any open neighborhood $U_n\big{|}_{n\in\mathbb{N}}$ of $\mathcal{T}_{\omega}$, there exists one trajectory $\gamma_n$ of system $H_{\epsilon}(p,q)$ entering $U_n$ from the place $\mathcal{O}(1)$ far from $\mathcal{T}_{\omega}$.
\end{The}

%In \cite{Du}, he also raised the following question:
%\begin{Que}
%Can we find a trajectory $\gamma_{\epsilon}$ of some nearly integrable system $H_{\epsilon}$, such that the KAM torus $\mathcal{T}_{\omega}$ is the $\omega$-limit (or $\alpha$-limit) set of $\mathcal{T}_{\omega}$?
%\end{Que}
From this theorem, we could deduce that KAM torus is of Lyapunov instability. But as $n\rightarrow\infty$, $\gamma_n$ is different from each other. So his construction is invalid to find asymptotic orbits of KAM torus.\\

%From now on we we only consider the $2.5-$degrees of freedom case. 
We can generalize Arnold's construction of \cite{Ar2} to a certain type of nearly integrable systems, which is known by {\it a priori} unstable ones. This condition actually assures that the existence of NHIC (Normally Hyperbolic Invariant Cylinder) with considerable length. Based on the celebrated theory developed by J. Mather, \cite{CY1} and \cite{CY2} first found the generic existence of diffusion orbits in this case from the variational view. More generalized nearly integrable systems are called {\it a priori} stable systems. We will face new difficulties in solving this case comparing to {\it a priori} {unstable} one:\\
\begin{enumerate}
  \item non-existence of long-length NHIC. Complicated resonance-relationship divides the 1-resonance lines into short segments, so we can't just use the overlap mechanism to find $\mathcal{O}(1)$ diffusion orbits.\\
      \item Coming out of 2-resonance. NHICs corresponding to different 1-resonance segments are separated by the chaotic layer caused by 2-resonance. We need to find trajectories in this layer to connect them (see figure \ref{fig1}).
\end{enumerate}
\begin{figure}
 \centering
  \includegraphics[width=7cm]{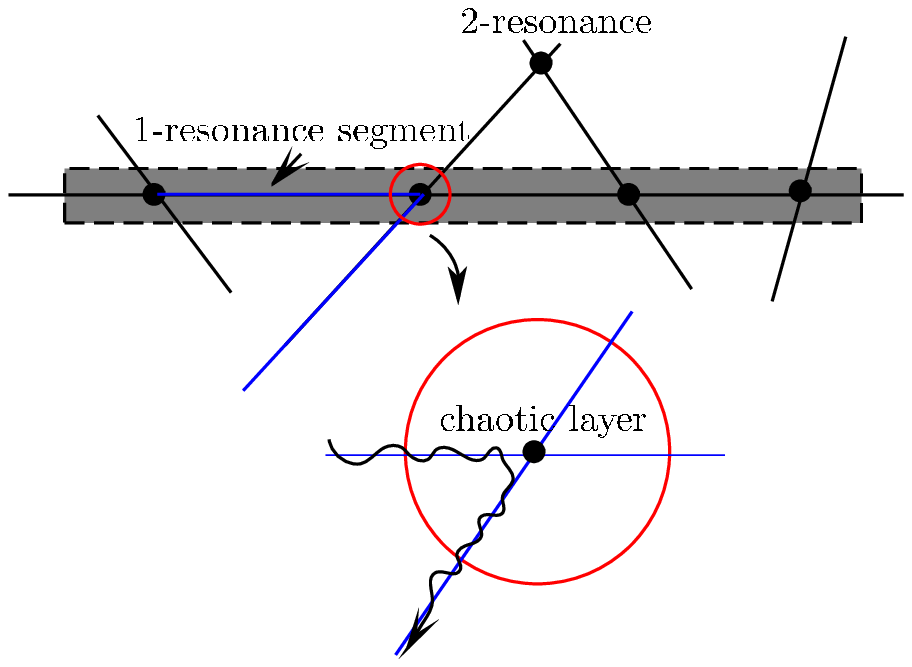}
 \label{fig1}
  \caption{}
\end{figure}
\vspace{10pt}
Since Diophatine vector ${\omega}$ is non-resonant, we have to overcome these two difficulties in finding asympytotic orbits. The first announcement of {\it a priori} stable case was given by J. Mather in 2003. He defined a conception `cusp residue' to measure the `size' of a set in topological space. Later, C-Q. Cheng verified the cusp genericity of diffusion orbits in {\it a priori} stable case. In \cite{Ch}, he proposed a plan to overcome the difficulties caused by 2-resonance:\\

{\bf Around the lowest flat $\mathbb{F}\doteq\{c\in H^1(\mathbb{T}^2,\mathbb{R})\big{|}\alpha_H(c)=\min \alpha_H\}$ there exists an incomplete intersection annulus $\mathbb{A}\subseteq H^1(\mathbb{T}^2,\mathbb{R})$ whose width we could precisely calculate and the Ma\~{n}\'{e} set $\tilde{\mathcal{N}}(c)$ ranges as a broken lamination structure, $\forall c\in\mathbb{A}$. Besides, the cohomology classes corresponding to NHICs could plug into $\mathbb{A}$.}\\

Based on this idea, we can connect different NHICs in this annulus $\mathbb{A}$ with Mather's mechanism diffusion orbits discovered in \cite{Mat2}. These orbits can be connected with the 'Arnold' mechanism diffusion orbits of NHICs and we succeed to construct $\mathcal{O}(1)$ diffusion orbits in {\it a priori} stable case.\\

From now on we only consider the case of $2.5-$degrees of freedom. As a special case of {\it a priori} stable systems, finding asymptotic trajectories of KAM torus will face another new difficulty: infinitely many changes of 1-resonance lines will be involved in. Here we can give a rough explanation on this. From \cite{Ch,Mat1} we know that the `cusp genericity' is caused by the restrictions of hyperbolic strength on different 1-resonant lines. Since $\omega$ is non-resonant, it's unavoidable to face infinitely many 1-resonant lines. These lines cause infinite times `cusp remove' to the perturbed function space $C^r(T^{*}\mathbb{T}^2\times\mathbb{S}^1, \mathbb{R})$ on the contrary. So we only have a `porous' set $\mathcal{P}^{r}\doteq\{f\in C^r(T^{*}\mathbb{T}^2\times\mathbb{S}^1,\mathbb{R})\big{|}\|f\|_{C^r}\leq1\}$ left, of which we have the chance to find asymptotic trajectories. Recall that this set $\mathcal{P}^r$ is not open in $C^r(T^{*}\mathbb{T}^2\times\mathbb{S}^1, \mathbb{R})$!

\begin{The}\label{main1}
For nearly integrable systems written by
\begin{equation}\label{2}
H(p,q,t)=h(p)+\epsilon f(p,q,t),\;(p,q,t)\in T^{*}\mathbb{T}^2\times\mathbb{S}^1,\;\epsilon\ll1,
\end{equation}
here $h(p)$ is strictly convex, $\nabla h(0)=\omega$, and $D^2h(0)$ is strictly positively definite. For a fixed Diophantine vector $\tilde{\omega}=(\omega,1)\in\mathbb{R}^3$, we could find $\epsilon_0=\epsilon_0(\omega,D^2 h(o))$ such that for $\epsilon\leq\epsilon_0$ and $f(q,p,t)\in \mathcal{P}^{r}(r\geq 8)$ with self-similar and weak-coupled structures, of which we can find asymptotic trajectories of KAM torus $\mathcal{T}_{\omega}$.
\end{The}
\begin{Rem}
Here the self-similar and weak-coupled structures are proposed to the Fourier coefficients of $f(q,p,t)$. To avoid the collapse of infinitely many times cusp-remove, we should control the speed of decline of hyperbolicity along the resonance lines tend to $\omega$, and then the Fourier coefficients are involved in. Later we will see that the Fourier coefficients corresponding to different resonant lines are independent from each other. We could benefit from this and raise a self-similar structure to simplify our treatment of infinite resonance relationships to finite ones.
\end{Rem}
For a system of a form (\ref{2}), we can ensure the persistence of $\mathcal{T}_{\omega}$ as long as $\epsilon_0$ is sufficiently small. Then the following holds:
\begin{The} \cite{Ch2}\label{change}
There exists a smooth exact symplectic transformation $\mathfrak{T}_f^{\infty}:\mathbb{D}_0\rightarrow\mathbb{D}_0$, where $\mathbb{D}_0\subset T^{*}\mathbb{T}^2\times\mathbb{S}^1$ is a small neighborhood of $\{0\}\times\mathbb{T}^2\times\mathbb{S}^1$. For $\epsilon\leq\epsilon_0$ and under this transformation, we can convert Hamiltonian (\ref{2}) to
\begin{equation}\label{3}
H(p,q,t)=h(p)+\langle p^t,f(q,t)p\rangle + \mathcal{O}(p^3),\quad (q,p,t)\in \mathbb{D}_0,
\end{equation}
here $\langle p^t,f(q,t)p\rangle$ is a quadratic polynomial with $p^t=(p_1,p_2)$, $\nabla h(0)=\omega$, and $D^2h(0)$ is strictly positively definite.
\end{The}
\begin{proof}
It's a direct cite of Lemma (6.1) in \cite{Ch2} for details.
\end{proof}

Moreover, we can use finite steps of `Birkhoff Normal Form' transformations to raise the order of polynomial and get
\begin{Lem}\label{robust}
There exists another smooth exact symplectic transformation $\mathfrak{R}_f^{\infty}:\mathbb{D}_0\rightarrow\mathbb{D}_0$ under which system (\ref{3}) can be changed into
\begin{equation}\label{4}
H(q,p,t)=h(p)+p^{\sigma}f(q,t)+\mathcal{O}(p^{\sigma+1}),\quad (q,p,t)\in \mathbb{D}_0,
\end{equation}
with a sufficiently large $\sigma\in\mathbb{Z}_{+}$ and $f(q,t)\in C^r(\mathbb{T}^2\times\mathbb{S}^1,\mathbb{R})(r\geq5=2\times2+1)$.
\end{Lem}
%\begin{proof}
%See Appendix for details about 'Birkhoff Normal Form' transformation.
%\end{proof}
\begin{Rem}
In the above theorem we omit the small number $\epsilon$ which assures the existence of KAM torus $\mathcal{T}_{\omega}$, since we can restrict diam$\mathbb{D}_0$ much smaller than $\epsilon$. We just need to find asymptotic trajectories in this domain. From now on, we will write $\langle p^t,f(q,t)\underbrace{p\rangle,p\rangle,\cdots,p\rangle}_{\sigma-1}$ as $p^{\sigma} f(q,t)$ for short without confusion.
\end{Rem}
Based on Theorem \ref{change} and Lemma \ref{robust}, we could convert Theorem \ref{main1} into the following
\begin{The}{\bf (Main Result)}\label{main2}
For the system of a form (\ref{4}), we can find proper $f(q,t)\in C^r(\mathbb{T}^2\times\mathbb{S}^1,\mathbb{R})$ with a self-similar and weak-coupled structure of which there exists at least one asymptotic trajectory to the KAM torus $\mathcal{T}_{\omega}$.
\end{The}
At last, we sketch out our plan: to find the proper $f(q,t)$, we need a list of `rigid' conditions to be satisfied. Owing to these conditions, we can give a `skeleton' of $f(q,t)$ which is not easy to be destroyed. then we use `soft' generic perturbations to construct diffusion orbits which finally tends to $\mathcal{T}_{\omega}$. Here, `soft' means the perturbations can be chosen arbitrarily small and arbitrarily smooth, which is known from \cite{CY1}, \cite{CY2} and \cite{Ch}.\\

From the proof the readers can see that the frame we made in order to get the asymptotic orbits is `firm' enough and small perturbations can't destroy it. Our method is neither the same with the way V. Kaloshin and M. Saprykina used in \cite{KS}, nor the same with the way P. Calvez and R. Douady used in \cite{CD}(in their papers they considered some close problems with ours). Our new approach benefits us with the chance to find more systems satisfying our demand.\\

We also recall that other two papers related with our result: one is \cite{KZZ} in 2010 and the other is \cite{KMV} in 2004. The latter one considered the asymptotic trajectories of resonant elliptic points, which is different from our situation and only finitely many resonant lines are involved in.\\

This paper is our first step to find preciser diffusion orbits in general systems. There's still a long way to go for the target of giving a rigorous answer to the quasi-ergodic hypothesis, and it's still open to find alternative mechanisms to construct diffusion orbits. Interestingly, T. Tao found an example of cubic defoucusing nonlinear Schr\"{o}dinger equation of which energy transports to higher frequencies in \cite{Tao}. His construction shares some similarities with Arnold Diffusion. Moreover, a self-similar resonant structure with special arithmetic properties is also applied in his construction. Aware of these, we are confident that it must be a hopeful direction to apply our diffusion mechanisms to PDE problems. Recently, M. Guardia and V. Kaloshin have made some progress in this domain\cite{GK}.
\vspace{10pt}
\subsection{Outline of the Proof}

We just need to prove Theorem \ref{main2} which is a special form of Theorem \ref{main1} via a $C^{\infty}$ symplectic transformation. As we can see, system (\ref{4}) is a Tonelli Hamiltonian, of which $\mathcal{T}_{\omega}=\{p=0\}\times\mathbb{T}^2\times\mathbb{S}^1$ is actually unperturbed. 
%Since $\{\omega\doteq\nabla h(p)\big{|}p\in \mathcal{B}(0,\ell)\subset\mathbb{R}^2\}$ is diffeomorphic to $\{p\in\mathcal{B}(0,\ell)\}$ for $\ell\ll1$, 
We could find a skeleton of infinitely many resonant lines $\{\Gamma_i^{\omega}\}_{i=1}^{\infty}$ in the frequency space, 
%which correspond to $\{\Gamma_i^{p}\}_{i=1}^{\infty}$ in the momentum space. 
along which $\Gamma_i^{\omega}$ approaches to Diophantine vector $\omega$ as $i\rightarrow\infty$ (See figure \ref{fig2}). We could divide this skeleton into {\bf 1-resonant lines, transitional segments from 1-resonance to 2-resonance and 2-resonant points} according to different resonant relationships. Each of them we need different mechanisms to deal with.\\
\begin{figure}
  \centering
  % Requires \usepackage{graphicx}
  \includegraphics[width=5cm]{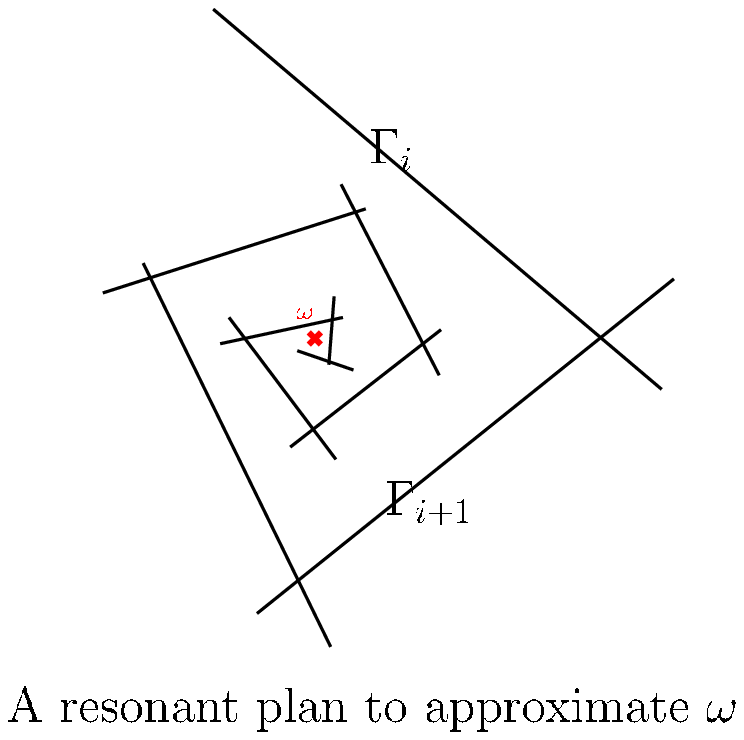}\\
  \caption{}\label{fig2}
\end{figure}

For the former two cases, to supply the NHICs with enough hyperbolicity, we need several `rigid' conditions {\bf U2,3} for the Fourier coefficients of $f(q,t)$ corresponding to the current resonant line $\Gamma_i^{\omega}$ of considerations. For the 2-resonant case $\Gamma_i^{\omega}\cap\Gamma_{i+1}^{\omega}$, a `weak-coupled' structure can be available by properly choosing the mixed Fourier coefficients according to $\Gamma_{i}^{\omega}$ and $\Gamma_{i+1}^{\omega}$. This simplifies the dynamic behaviors of 2-resonance greatly. Also several `rigid' conditions {\bf U$4\rightarrow8$} are needed in this case to generate an incomplete intersection annulus with certain width and to persist the bottom parts of crumpled NHICs.\\

Notice that there exist extra 2-resonant points inside $\Gamma_i^{\omega}$, which we call {\bf sub 2-resonant} points. Since there isn't any transition between different resonant lines at these points, our diffusion orbits just need to cross them and go on along the NHICs according to 1-resonant lines. To achieve this, rigid condition {\bf U3} is needed.\\

Recalls that infinitely many resonant relationships are considered in our case, so these `rigid' conditions have to be uniformly satisfied according to all of these resonant lines. That's why we mark the `rigid' conditions with a letter `U'. It's the cost to avoid the collapse caused by infinite times cusp remove. But the self-similar structure simplifies the complexity and gives us a universal treatment for all the resonant relationships.\\ 

Now we explain why these `rigid' conditions can be satisfied without conflict and give a sketch of the construction. 
%$\forall f(q,t)\in C^r(\mathbb{T}^2\times\mathbb{S}^1,\mathbb{R})$, These conditions and structural demands can all be satisfied as long as the $\mathbb{Z}^3$ sequence of Fourier coefficients $\{f_{(k_1,k_2,k_3)}\}_{k_i\in\mathbb{Z},i=1,2,3}$ are properly chosen, where $f_{\vec{k}}\in\mathbb{R}$, $\vec{k}=(k_1,k_2,k_3)$. We will explain this point in the following and 
We begin with such a Tonelli Hamiltonian:
\vspace{4pt}
\begin{equation}\label{5}
H(q,p,t)=h(p)+p^{\sigma}f(q,t),\quad (q,p,t)\in \mathbb{D}_0.
\end{equation}

$\bullet$ First, we choose a proper resonant plan $\{\Gamma_i^{\omega}\}_{i=1,i\in\mathbb{N}}^{\infty}$ which approximate $\omega$ steadily.\\

$\bullet$ Second, along these resonant lines, we can transform system (\ref{5}) to a Resonant Normal Form $H=h+Z+R$ with finite KAM iterations in a neighborhood of $\{\mathcal{B}(\Gamma_i^{p},\delta_i)\times\mathbb{T}^2\times\mathbb{S}\}_{i=1}^{\infty}$. Here $Z=[f]_{\omega^*}$ is the average term corresponding to the current frequency $\omega^*$. All these conditions and structural demands can all be satisfied by $Z$ as long as the $\mathbb{Z}^3$ sequence of Fourier coefficients $\{f_{(k_1,k_2,k_3)}\in\mathbb{R}\}_{k_i\in\mathbb{Z},i=1,2,3}$ are properly chosen, where $\vec{k}=(k_1,k_2,k_3)\in\Lambda_{\omega^*}$ subspace of $\mathbb{Z}^3$. We can see that different resonant lines will decide different subspaces of $\mathbb{Z}^3$ which are independent from each other. This point is very important for our case.\\

$\bullet$ Third, at the 2-resonance $\Gamma_i^{\omega}\cap\Gamma_{i+1}^{\omega}$, we need to connect different NHICs of their bottoms. A weak-coupled structure can be available by the mixed terms of Fourier coefficients according to $\Gamma_i^{\omega}\cap\Gamma_{i+1}^{\omega}$. This structure doesn't damage the hyperbolicity of NHICs and are strong enough to supply us a chaotic layer with sufficient width, which we called {\bf incomplete intersection annulus} here. In this part we mainly used the same method proposed in \cite{Ch}.\\

It's remarkable that this `weak-coupled' structure reduces the complexity of dynamical behaviors greatly at the 2-resonant domain, which can be considered as an application of Melnikov's approach. We take \cite{Tr2} for a convenient reference. \\
%Also we can demand a self-similar structure to the Fourier coefficients of $Z$ since they are independent from each other at different resonant places. So we can deal with these 'rigid' conditions and structural demands unifiedly at finite many resonance relationships.\\

$\bullet$ At last, since all the `rigid' conditions have been satisfied, we can perturb system (\ref{5}) step by step to find diffusion trajectories which extend towards to $\mathcal{T}_{\omega}$ gradually. Based on the matured methods of genericity and regularity developed by \cite{Ch,CY1,CY2} and \cite{L}, the perturbation functions are very `soft', i.e. they can be made arbitrarily small and smooth. Here we use a list of perturbations $\{f_j\}_{j=1,j\in\mathbb{N}}^{\infty}$ to modify system (\ref{5}) and get a system of a form (\ref{4}), for which we indeed get an asymptotic trajectory of $\mathcal{T}_{\omega}$.\\

The paper is organized as follows. In the rest part of this section we will give a sketch of the Mather Theory and Fathi's weak-KAM method, also several properties about global elementary weak-KAM solutions in the finitely covering space\cite{Ch}. In Section 2, we give a resonant plan $\{\Gamma_i^{\omega}\}_{i=1,i\in\mathbb{N}}^{\infty}$ to approximate $\omega$ and get its fine properties. Section 3 supplies a Stable Normal Form with finite KAM iterations which is unified for all the $\Gamma_{i}$. In this part all the `rigid' conditions could be raised naturally. In Section 4, we prove the existence of NHICs in every case separately, 1-resonance, transitional segments from 1-resonance to 2-resonance and then 2-resonance. In the later two cases, a homogenized method is involved. In the 2-resonance case, a weak-coupled structure is used and we could give a precise estimate about the lowest positions where the bottoms of NHICs could persist. Besides, we get the conclusion that Aubry sets just locate on these NHICs. In Section 5, the existence of incomplete intersection annulus of c-equivalence is established and its width can also be precisely estimated. In Section 6, we recall the approach to get two types of locally connecting orbits by modifying the Lagrangian. This part is mainly based on genericity and regularity of \cite{CY1,CY2,L} and \cite{Ch}. With these preliminary works, we get our asymptotic trajectories by a list of `soft' perturbations in Section 7. Therefore, we finish our construction of system (\ref{4}) and get our main conclusion.
\vspace{10pt}
\subsection{Brief introduction to Mather Theory and properties of weak KAM solutions}\label{Mather theory}
$\newline$

In this subsection we will give a profile of the tools we used in this paper: Mather Theory and weak KAM theorem. Recall that the earliest version of weak KAM theory which Fathi gave us in \cite{Fa} mainly concerns the autonomous Lagrangians, but most of the conclusions are available for the time periodic case.

\begin{defn}
Let $M$ be a smooth closed manifold. We call $L(x,\dot{x},t)\in C^r(TM\times\mathbb{S}^1,\mathbb{R})$ $(r\geq2)$ a {\bf Tonelli Lagrangian} if it satisfies the following conditions:
\begin{itemize}
\item {\bf Positive Definiteness:} For each $(x,\dot{x},t)\in TM\times\mathbb{S}^1$, the Lagrangian function is strictly convex in velocity, i.e. the Hessian matrix $\partial_{\dot{x}\dot{x}}L$ is positively definite.
\item {\bf Superlinearity:} L is fiberwise superlinear, i.e. for each $(x,t)\in M\times\mathbb{S}^1$, we have $L/\|\dot{x}\|\rightarrow\infty$ as $\|\dot{x}\|\rightarrow\infty$.
\item {\bf Completeness:} All the solutions of the Euler-Lagrangian (E-L) equation corresponding to $L$ are well-defined for $t\in\mathbb{R}$.
\end{itemize}
\end{defn}

For such a Tonelli Lagrangian $L$, the variational minimal problem of $\gamma\in C^{ac}([a,b],M)$ with fixed end points $\gamma(a)=x$, $\gamma(b)=y$ is well posed as:
\[
A_L(\gamma)=\inf_{\substack{\gamma\in C^{ac}([a,b],M)\\
\gamma(a)=x, \gamma(b)=y}}\int_a^bL(\gamma(t),\dot{\gamma}(t),t)dt.
\]
Then we can see that if $\gamma$ is the critical curve of this variational problem, then it must satisfy the Euler-Lagrangian equation:
\begin{equation}\label{E-L}
\frac{d}{dt}\partial_{\dot{x}}L(\gamma(t),\dot{\gamma}(t),t)=\partial_xL(\gamma(t),\dot{\gamma}(t),t),\quad t\in[a,b].
\end{equation}
We call a curve $\gamma\subset M$ is a {\bf solution of E-L equation} if $\forall a,b\in\mathbb{R}$ with $a<b$, $\gamma(t)$ satisfies (\ref{E-L}) for $t\in[a,b]$. Once $\gamma$ is a solution of $L$, we actually see that $\gamma\in C^r(\mathbb{R},M)$ from the Tonelli Theorem \cite{Mo}. We can define a flow map of $\phi_L^t:TM\times\mathbb{S}^1\rightarrow TM\times\mathbb{S}^1$ by 
\[
\phi_L^t(x,v,s)=(\gamma_{x,v}(t),\dot{\gamma}_{x,v}(t),t+s\mod1)\in TM\times\mathbb{S}^1,\quad\forall t\in\mathbb{R}, s\in\mathbb{S}^1,
\]
where $(\gamma_{x,v}(0),\dot{\gamma}_{x,v}(0),s)=(x,v,s)$ and $\gamma_{x,v}$ is a solution of $L$. Then we can generate a $\phi_L-$invariant probability measure $\mu_{\gamma_{x,v}}$ by $\gamma_{x,v}$ with the following ergodic Theorem
\begin{equation}
\lim_{T\rightarrow\infty}\frac{1}{T}\int_0^TL(\gamma_{x,v}(t),\dot{\gamma}_{x,v}(t),t)dt=\int_{TM\times\mathbb{S}^1}L(x,v)d\mu_{\gamma_{x,v}}.
\end{equation}
We denote the set of all the $\phi_L-$invariant probability measures by $\mathfrak{M}_{inv}(L)$.
\begin{Rem}
Here we need to make a convention once for all. A continuous map $\gamma:I\rightarrow M$ is called a {\bf curve}, where $I\subset\mathbb{R}$ is an interval either bounded or unbounded, open or closed. For the time-dependent case, $\cup_{t\in I}(\gamma(t),t)$ is considered as a curve as well. We call $(\gamma,\dot{\gamma},t):I\rightarrow TM\times\mathbb{R}$ an {\bf orbit}, or a {\bf trajectory}, iff it's invariant under the flow map $\phi_L^t$.
\end{Rem}

Let $\eta_c(x)dx$ be a closed 1-form on $M$, with $[\eta_c]=c\in H^1(M,\mathbb{R})$. Then 
\[
L_c\doteq L(x,v,t)-\langle\eta_c,v\rangle,\quad (x,v,t)\in TM\times\mathbb{S}^1
\]
is also a Tonelli Lagrangian and we can see that any solution of $L$ is also a solution of $L_c$, vise versa. This supplies us with a chance to distinguish measures of $\mathfrak{M}_{inv}(L)$ by cohomology class.\\

We define the $\alpha-$function of $L$ by
\begin{equation}
\alpha_L(c)=-\inf_{\mu\in\mathfrak{M}_{inv}(L)}\int_{TM\times\mathbb{S}^1}L-\eta_cd\mu.
\end{equation}
It's a continuous, convex and super-linear function. We call the minimizer of above definition a {\bf c-minimizing measure}, and the set of all c-minimizing measures can be written by $\mathfrak{M}_{inv}(c)$. It's a convex set of the space of all the probability measures, under the weak* topology. If $\mu_c$ is a extremal point of $\mathfrak{M}_{inv}(c)$, it must be an ergodic measure with its support a minimal set for $\phi_L^t$.\\

For $\mu\in\mathfrak{M}_{inv}(L)$, there exists a $\rho(\mu)\in H^1(M,\mathbb{R})$ according to it and
\[
\langle \rho(\mu),[\eta_c]\rangle=\int_{TM\times\mathbb{S}^1}\eta_cd\mu,
\]
for every closed 1-form $\eta_c$ with $[\eta_c]=c\in H^1(M,\mathbb{R})$. Here  $\langle\,,\,\rangle$ denotes the canonical pairing between homology and cohomology. Then we have the following conjugated function
\begin{equation}
\beta_L(h)=\inf_{\substack{\mu\in\mathfrak{M}_{inv}(L)\\
\rho(\mu)=h}}\int_{TM\times\mathbb{S}^1}Ld\mu,\quad h\in H_1(M,\mathbb{R}).
\end{equation}
It's also continuous, convex, and super-linear \cite{Mat}. Similarly, we can define the set of all the minimizers of above   formula by $\mathfrak{M}_{inv}(h)$. Let $D^{-}\alpha_L(c)$ be the sub-differential set of $\alpha_L(c)$ at $c$ and $D^{-}\beta_L(h)$ be the sub-differential set of $\beta_L(h)$ at $h$. Then we have the following properties
\vspace{10pt}
\begin{itemize}
\item $\beta_L(h)+\alpha_L(c)=\langle c,h\rangle$,$\quad\forall c\in D^{-}\beta_L(h)$, $h\in D^{-}\alpha_L(c)$.\vspace{5pt}
\item $\forall\mu_h\in\mathfrak{M}_{inv}(h)$, we have $\mu_h\in\mathfrak{M}_{inv}(c)$ with $c\in D^{-}\beta_L(h)$.
\vspace{5pt}
\item $\forall\mu_c\in\mathfrak{M}_{inv}(c)$, we have $\mu_c\in\mathfrak{M}_{inv}(\rho(\mu_c))$.\vspace{10pt}
\end{itemize}

The union set of all the c-minimizing measures' support is the so-called {\bf Mather set}, which is denoted by $\widetilde{\mathcal{M}}_L(c)$. Its projection to $M\times\mathbb{S}^1$ is the {\bf projected Mather set} $\mathcal{M}_L(c)$. 
%Note that in some occasions we can omit the subscripts `inv', `L' for short. 
From \cite{Mat} we know that $\pi^{-1}\big{|}_{\mathcal{M}(c)}:M\times\mathbb{S}^1\rightarrow TM\times\mathbb{S}^1$ is a Lipschitz graph, where $\pi$ is the standard projection from $TM\times\mathbb{S}^1$ to $M\times\mathbb{S}^1$. \\

Sometimes, the Mather set is too `small' to handle with, so larger invariant sets should be involved in. We define
\begin{equation}
A_c(\gamma)\big{|}_{[t,t']}=\int_t^{t'}L(\gamma(t),\dot{\gamma}(t),t)-\langle\eta_c(\gamma(t)),\dot{\gamma}(t)\rangle dt+\alpha(c)(t'-t),
\end{equation}
\begin{equation}
h_c((x,t),(y,t'))=\inf_{\substack{\xi\in C^{ac}([t,t'],M)\\
\xi(t)=x\\
\xi(t')=y}}A_c(\xi)\big{|}_{[t,t']},
\end{equation}
where $t,t'\in\mathbb{R}$ with $t<t'$, and
\begin{equation}
F_c((x,\tau),(y,\tau'))=\inf_{\substack{\tau=t\mod1\\
\tau'=t'\mod1}}h_c((x,t),(y,t')),
\end{equation}
where $\tau,\tau'\in\mathbb{S}^1$. Then a curve $\gamma:\mathbb{R}\rightarrow M$ is called {\bf c-semi static} if 
\[
F_c((x,\tau),(y,\tau'))=A_c(\gamma)\big{|}_{[t,t']},
\]
for all $t,t'\in\mathbb{R}$ and $\tau=t\mod1$, $\tau'=t'\mod1$. A semi static curve $\gamma$ is called {\bf c-static} if
\[
A_c(\gamma)\big{|}_{[t,t']}+F_c((\gamma(t'),t'),(\gamma(t),t))=0,\quad\forall t,t'\in\mathbb{R}.
\]
The {\bf Ma\~{n}\'e set} which is denoted by $\widetilde{\mathcal{N}}(c)\subset TM\times\mathbb{S}^1$ is the set of all the c-semi static orbits.
\begin{The}{\bf (Upper semicontinuity\cite{CY1, CY2})}
The set-valued function $(c,L)\rightarrow\widetilde{\mathcal{N}}(c)$ is upper semicontinuous.
\end{The}
We can similarly define the {\bf Aubry set} by the set of all the c-static orbits, which can be written by $\tilde{\mathcal{A}}(c)$. Then we have 
\[
\widetilde{\mathcal{M}}(c)\subset\tilde{\mathcal{A}}(c)\subset\widetilde{\mathcal{N}}(c).
\]
Note that from now on we can omit the subscripts `inv', `L' for short. We also denote the {\bf projected Ma\~{n}\'e set} by $\mathcal{N}(c)$ and the {\bf projected Aubry set} by $\mathcal{A}(c)$. From \cite{Mat3} we can see that $\pi^{-1}:\mathcal{A}(c)\subset M\times\mathbb{S}^1\rightarrow \tilde{\mathcal{A}}(c)\subset TM\times\mathbb{S}^1$ is also a Lipschitz graph. Let
\begin{equation}
h_c^{\infty}((x,s),(x',s'))=\liminf_{\substack{s=t\mod1\\
s'=t'\mod1\\
t'-t\rightarrow\infty}}h_c((x,t),(x',t')),
\end{equation}
then we have
\[
h_c^{\infty}((x,\tau),(x,\tau))=0,\quad\forall (x,\tau)\in\mathcal{A}(c).
\]
We can further define a pseudo metric on $\mathcal{A}(c)$ by
\[
d_c((x,\tau),(x',\tau'))=h_c^{\infty}((x,\tau),(x',\tau'))+h_c^{\infty}((x',\tau'),(x,\tau))
\]
and then get an equivalent relationship: $(x,t)\sim(x',t')$ implies $d_c((x,t),(x',t'))=0$. Let $\mathcal{A}(c)/\sim$ be the quotient Aubry set, and the element of $\mathcal{A}(c)/\sim$ is called an {\bf Aubry class}, which can be written by $\mathcal{A}_i(c)$. We can see that $\mathcal{A}(c)=\cup_{i\in\Lambda}\mathcal{A}_i(c)$, where $\Lambda$ is an index set. We can define the {\bf Barrier function} between different Aubry classes by
\begin{eqnarray}
B_{c,i,j}(z,r)=h_c^{\infty}((x,t),(z,r))+h_c^{\infty}((z,r),(y,s))-h_c^{\infty}((x,t),(y,s)),\nonumber\\
\quad\quad\quad\quad\quad\forall (x,t)\in\mathcal{A}_i(c),\;(y,s)\in\mathcal{A}_j(c),\;(z,r)\in M\times\mathbb{S}^1.
\end{eqnarray}

\begin{Rem}
In some case, we need to consider the properties of curves in the finite covering space $\bar{M}$. Analogously, we can copy the conceptions of c-semi static curve and c-static curve on it, and take $\mathcal{N}(c,\bar{M})$ and $\mathcal{A}(c,\bar{M})$ as the according sets. We can see that $\mathcal{N}(c)\subsetneq\mathcal{N}(c,\bar{M})$ and $\pi\mathcal{A}(c,\bar{M})=\mathcal{A}(c)$ with $\pi:\bar{M}\rightarrow M$ the projection map. From \cite{Con} we can see that different Aubry class of $\mathcal{A}(c,\bar{M})$ can always be connected by c-semi static curves of $\mathcal{N}(c,\bar{M})$. This point plays a very important role in our diffusion mechanism.
\end{Rem}
In this following part, we'll give a survey about Fathi's weak KAM theory, which can be seen as a Hamiltonian version of Mather theory. 

\begin{defn}
We call a time-periodic system $H(x,p,t):T^*M\times\mathbb{S}^1\rightarrow\mathbb{R}$ {\bf Tonelli Hamiltonian}, if it satisfies the following:
\begin{itemize}
\item {\bf Positive Definiteness:} For each $(x,p,t)\in TM\times\mathbb{S}^1$, the Hamiltonian is strictly convex in momentum, i.e. the Hessian matrix $\partial_{pp}H$ is positive definite.
\item {\bf Superlinearity:} $H$ is fiberwise superlinear, i.e. for each $(x,t)\in M\times\mathbb{S}^1$, we have $H/\|p\|\rightarrow\infty$ as $\|p\|\rightarrow\infty$.
\item {\bf Completeness:} All the solutions of the Hamiltonian equation corresponding to $H$ are well-defined for $t\in\mathbb{R}$.
\end{itemize}
\end{defn}
We can associate to the Hamiltonian $H$ a Tonelli Lagrangian $L:TM\times\mathbb{S}^1\rightarrow\mathbb{R}$ by the Legendre transformation:
\begin{equation}\label{Legendre}
L(x,v,t)=\sup_{p\in T_x^*M}\langle p,v\rangle-H(x,p,t).
\end{equation}
Then $\partial_pH:T^*M\times\mathbb{S}^1\rightarrow TM\times\mathbb{S}^1$ become a diffeomorphism, whose inverse map is given by $\partial_vL:TM\times\mathbb{S}^1\rightarrow T^*M\times\mathbb{S}^1$. Note that the right side of (\ref{Legendre}) get its maximum for $v=\partial_pH(x,p)$. We can see that once $(\gamma,\dot{\gamma}):\mathbb{R}\rightarrow TM$ satisfies the E-L equation, then $(\gamma(t),\partial_vL(\gamma(t),\dot{\gamma}(t),t)):\mathbb{R}\rightarrow T^*M$ must satisfies the Hamiltonian equation of $H$, i.e. $(\gamma(t),p(t))$ is a trajectory of the flow map $\phi_H^t$ with $p(t)=\partial_vL(\gamma(t),\dot{\gamma}(t),t)$.\\

Let $\eta_c$ be a closed 1-form of $M$ with $[\eta_c]=c$, then we can make $L_c\doteq L-\eta_c$ and 
\[
H_c(x,p,t)\doteq H(x,p+\eta_c(x),t),\quad\forall (x,p)\in T^*M,
\]
with $H(x,p,t)$ the corresponding Hamiltonian of $L(x,v,t)$. We can define such a Lax-Oleinik mapping on $C^0(M\times\mathbb{S}^1,\mathbb{R})$ by 
\begin{equation}
T^{-}_{c,t} u(x,s)=\min_{y\in M}[u(y,s-t)+h_c((y,s-t),(x,s))],
\end{equation}
where $u(x,s)$ is a fixed continuous function of $M\times\mathbb{S}^1$. Then we can get the following fixed point of Lax-Oleinik mapping by 
\[
u_c^{-}(x,s)\doteq\liminf_{t\rightarrow\infty}T^{-}_{c,t}u(x,s).
\]
We call this $u_c^{-}$ a {\bf weak KAM solution} of $H$ system\cite{B3}. For a fixed $t\in[0,1]$, we can see that $u_c^{-}(\cdot,t)$ is semi concave with linear modulus of $x\in M$ (SCL($M$)). This is because the uniformly convexity of $H$.
\begin{defn}\cite{Ca}
We say a function $u:M\rightarrow\mathbb{R}$ is {\bf semi concave with linear modulus} if it's continuous and there exists $C\geq0$ such that 
\[
u(x+h)+u(x-h)-2u(x)\leq C\|h\|^2,
\]
for all $x,h\in M$. The constant $C$ is called the {\bf semi concavity constant} of $u$.
\end{defn}
As a SCL(M) function is $C^2-$differentiable almost everywhere, then we have the following
\begin{equation}
\partial_t u_c^{-}(x,t)+H(x,du_c^{-}(x,t)+c,t)=\alpha_L(c),\quad a.e. \;(x,t)\in M\times\mathbb{S}^1.
\end{equation}
Actually, $u_c^{-}(x,t)$ is a {\bf viscosity solution} of above Hamiltonian-Jacobi equation, which can be seen from \cite{Fa}.\\

For a c-semi static orbit $(\gamma(t),\dot{\gamma}(t))$, we can see that $u_c^{-}$ is differentiable at $\gamma(t)$ and $c+du_c^{-}(\gamma(t))=\partial_vL(\gamma(t),\dot{\gamma}(t),t)$. Besides, $(\gamma(t),c+du_c^{-}(\gamma(t)))$ is a Hamiltonian flow of $\phi_H^t(\gamma(0),du_c^{-}(\gamma(0)))$, and
\[
\partial_t u_c^{-}(\gamma(t))+H(\gamma(t),c+du_c^{-}({\gamma}(t),t),t)=\alpha_L(c),\quad\forall t\in\mathbb{R}.
\]
So we can define 
\[
\widetilde{\mathcal{N}}_H(c)=\bigcup_{\substack{\gamma(t)\in\mathcal{N}_L(c)\\
t\in\mathbb{R}}}(\gamma(t),c+du_c^{-}(\gamma(t),t),t)
\]
and
\[
\widetilde{\mathcal{A}}_H(c)=\bigcup_{\substack{\gamma(t)\in\mathcal{A}_L(c)\\
t\in\mathbb{R}}}(\gamma(t),c+du_c^{-}(\gamma(t),t),t)
\]
by the conjugated Ma\~{n}\'e set and Aubry set. Sometimes, we can change the subscript of $\alpha_L(c)$ to $H$, as long as $H$ is conjugated to $L$.\\

On the other side, Let $\tilde{L}(x,v)\doteq L(x,-v)$ be the symmetrical Lagrangian of $L$, then we have a similar Lax-Oleinik mapping $\tilde{T}_{c,t}^{-}$ on $C^0(M\times\mathbb{S}^1,\mathbb{R})$ and $\forall u\in C^0(M\times\mathbb{S}^1,\mathbb{R})$, 
\begin{equation}\label{symmetrical}
\tilde{u}_{c}^{-}(x,s)\doteq\liminf_{t\rightarrow\infty}\tilde{T}^{-}_{c,t}u(x,s)
\end{equation}
exists for all $(x,s)\in M\times\mathbb{S}^1$. It's the weak KAM solution of $\tilde{H}$, which is of the form
\[
\tilde{H}(x,p,t)=\sup_{v\in T_xM}\langle p,v\rangle-\tilde{L}(x,v,t).
\]
%If we take $-u_c^{+}=\tilde{v}_{c}^{-}$ with 
Take a special function $w=-u(x,t)\in C^0(M\times\mathbb{S}^1,\mathbb{R})$ into the semi-group and get the inferior limit as (\ref{symmetrical}), then we can see that $u_c^{+}\doteq-\tilde{w}_{c}^{-}$ satisfies
\begin{equation}
\partial_t u_c^{+}(x,t)+H(x,du_c^{+}(x,t)+c,t)=\alpha_L(c),\quad a.e. \;(x,t)\in M\times\mathbb{S}^1.
\end{equation}

%Besides, we can see that $u_c^{+}$ is a semi-concovex function with linear modulus (SCVL\cite{Ca}) on $M\times\mathbb{S}^1$.
\begin{defn}
$(\gamma(t),\dot{\gamma}(t)):\mathbb{R}\rightarrow TM$ is called {\bf backward c-semi static orbit}, if there exists a $T^{-}\in\mathbb{R}$ such that 
\[
F_c((\gamma(t),\tau),\gamma(t'),\tau'))=A_c(\gamma)\big{|}_{[t,t']}, 
\]
holds for all $t<t'\in(-\infty,T^{-}]$. Analogously, $(\gamma(t),\dot{\gamma}(t)):\mathbb{R}\rightarrow TM$ is called {\bf forward c-semi static orbit}, if there exists a $T^{+}\in\mathbb{R}$ such that 
\[
F_c((\gamma(t),\tau),\gamma(t'),\tau'))=A_c(\gamma)\big{|}_{[t,t']}, 
\]
holds for all $t<t'\in[T^{+},\infty)$. Here $\tau=t\mod1$ and $\tau'=t'\mod1$.
\end{defn}
\begin{The}\cite{Fa}
Let $(x,t)\in M\times\mathbb{S}^1$ be a differentiable point of $u_c^{-}$ (or $u_c^{+}$). As the initial condition, $(x,du_c^{-}(x))$ ($(x,du_c^{+}(x))$) will decide a unique trajectory of $H$ by $(x^{-}(t),p^{-}(t)):\mathbb{R}\rightarrow T^*M$, ($(x^{+}(t),p^{+}(t)):\mathbb{R}\rightarrow T^*M$) with $(x^{-}(0),p^{-}(0))=(x,du_c^{-}(x))$ ($(x^{+}(0),p^{+}(0))=(x,du_c^{+}(x))$). The corresponding orbit $(x^{-}(t),\partial_pH(x^{-}(t),p^{-}(t),t)):\mathbb{R}\rightarrow TM$ ($(x^{+}(t),\partial_pH(x^{+}(t),p^{+}(t),t)):\mathbb{R}\rightarrow TM$) is backward c-semi static on $(-\infty,0]$ (forward c-semi static on $[0,+\infty)$).
\end{The}

From \cite{Ch} we know, in a proper covering space $\bar{M}$, $\mathcal{A}_H(c,\bar{M})$ may have several classes, even though $\mathcal{A}_H(c)$ is of uniquely class. These different classes of $\mathcal{A}_H(c,\bar{M})$ are disjoint from each other\cite{Mat3}, which can be written by $\mathcal{A}_L^i(c,\bar{M})$, $i\in\Lambda$. Then we can find a sequence of Tonelli Hamiltonians $\{H^i_j:T^*\bar{M}\times\mathbb{S}^1\rightarrow\mathbb{R}\}_{j=1}^{\infty}$ to approximate $H$ under the $C^r-$norm, such that $\mathcal{A}_H^i(c,\bar{M})$ is the unique Aubry class of $H_j^i$. Accordingly, we can find a sequence of weak KAM solutions of $\{u_{c,i,j}^{-}:\bar{M}\times\mathbb{S}^1\rightarrow\mathbb{R}\}_{j=1}^{\infty}$ which converges to a special weak KAM solution $u_{c,i}^{-}$ of system $H$ in $\bar{M}$. That's our {\bf elementary weak KAM solution} of class $\mathcal{A}_i(c,\bar{M})$. Analogously, we get all the elementary weak KAM solutions $\{u_{c,i}^{-}\}_{i\in\Lambda}$.\\

With the help of this definition, we can translate our Barrier function $B_{c,i,j}$ in $\bar{M}\times\mathbb{S}^1$ into a simpler form:
\begin{equation}\label{Barrier}
B_{c,i,j}(z,r)=u_{c,i}^{-}(z,r)-u_{c,j}^{+}(z,r),\quad\forall(z,r)\in\bar{M}\times\mathbb{S}^1.
\end{equation}

\vspace{10pt}
\section{Choose of resonant plan and Fourier properties of functions} \vspace{10pt}

For convenience, we first make a convention on the symbol system once for all. Recall that the system (\ref{5}) is of the form
\[
H(p,q,t)=h(p)+p^{\sigma}f(q,t),\quad(p,q,t)\in \mathbb{D}_0\subseteq T^{*}\mathbb{T}^2\times\mathbb{S},
\]
where $p^{\sigma}f(q,t)$ is actually a polynomial of multi-variables $p=(p_1,p_2)$. Here $h(0)=0$, $\nabla h(0)=\vec{\omega}_0$, $D^2h(0)$ is strictly positively definite and $\|D^2h(0)\|\sim\mathcal{O}(1)$. $f(q,t)\in\mathfrak{B}(0,c_1)\subseteq C^r(\mathbb{T}^2\times\mathbb{S},\mathbb{R})$, where $c_1\sim\mathcal{O}(1)$ is a fixed constant. $\vec{\omega}_0$ is a Diopantine frequency of index $(\tau,C_0)$, i.e. 
\[
\forall \vec{k}\in\mathbb{Z}^2, \;|||\langle \vec{k},\vec{\omega}_0\rangle|||\geq\frac{C_0}{|\vec{k}|^{n-1+\tau}},
\]
where $|\vec{k}|=\max\{|k_1|,|k_2|\},\;\vec{k}=(k_1,k_2)$ and $|||x|||=x-[x]$ is the reminder part of $x$. We add a subscript `0' to the Diophantine frequency $\omega$ to avoid confusion in the following. Notice that $\mathbb{T}=\mathbb{R}/2\pi\mathbb{Z}$ and $\mathbb{S}=\mathbb{R}/\mathbb{Z}$ in our situation.\\
	
We denote the norm $\|\cdot\|_{C^r}$ of $C^r(\mathbb{T}^2\times\mathbb{S},\mathbb{R})$ by $\|f(q,t)\|\doteq\sum_{|\vec{\alpha}|=0}^{r}\|f(q,t)\|_{C^0}$, where $|\vec{\alpha}|\doteq|\alpha_1|+|\alpha_2|+|\alpha_3|$ and $\|\cdot\|_{C^0}$ is the uniform norm.  
	\begin{Lem}\label{Fourier}
		$\forall f(q)\in C^r(\mathbb{T}^3,\mathbb{R})$, $q=(q_1,q_2,q_3)$ we have:
	\begin{enumerate}
	\item $\|f_{\vec{k}}\|_{C^0}\leq (2\pi|\vec{k}|)^{-r}\cdot\|f\|_{C^r}$, here $\vec{k}=(k_1,k_2,k_3)\in\mathbb{Z}^3$ and $|\cdot |$ is denoted as above.
	\item $\kappa_3\doteq\sum_{\vec{k}\in\mathbb{Z}^3}|\vec{k}|^{-3-1}$ is a constant of $\mathcal{O}(1)$, here `3' can be replaced by a dimensional argument $n$.
	\item $R_K f\doteq\sum_{|\vec{k}\geq K|} f_{\vec{k}}\exp^{2\pi i\langle\vec{k},q\rangle}$, then$ \|R_K f\|_{C^2}\leq\kappa_3 K^{-r+3+3}\|f\|_{C^r}$.
	\end{enumerate}
	\end{Lem} 
	\begin{proof}
	\begin{enumerate}
	\item Since $f_{\vec{k}}=\frac{1}{2\pi}\int_0^{2\pi}f(q)\exp{-2\pi i\langle\vec{k}, q\rangle}dq$, we can easily get the estimate by r-times integral by parts.
	\item The number of $\vec{k}$ satisfying $|\vec{k}|=l$ is less than $2^3 l^{3-1}$, so $\sum_{k\in\mathbb{Z}^3}|k|^{-3-1}\leq\sum_{l\in\mathbb{N}}2^3 l^{3-1}\frac{1}{l^{3+1}}\leq 8\sum_{l\in\mathbb{N}}\frac{1}{l^2}<\infty$. Same result for $\mathbb{T}^n$ can be get with replacing $3$ by $n$.
	\item We also use a skill of integral by parts: $\|R_K f\|_{C^2}\leq\sum_{|k|\geq K}|k|^{-r+2}\|f\|_{C^r}\leq K^{-r+3+3}\|f\|_{c^r}\sum_{|k|\geq K}|k|^{-3-1} $, then use the result of (2). By the way, $r\geq7$ is necessary and actually we can choose it properly large.
	\end{enumerate}
	\end{proof}
		 {\bf Notice:} For simplicity, we often omit the vector symbol $\vec{\cdot}$.  Later on , we  also need a norm $\|\cdot\|_{C^r, \mathcal{B}(p^{*}, \delta)}$ on some subdomain of  phase space $T^{*}\mathbb{T}^2\times\mathbb{S}$, which could be defined in the same way as above, except the action variables $p=(p_1,p_2)$ added. Sometimes we denote the norm by $\|\cdot\|_{C^r,\mathcal{B}}$ or $\|\cdot\|_{C^r, \delta}$ for short, as long as there's no ambiguity.\\

From the aforementioned Lemma, $\forall f\in C^r([0,1]^3,\mathbb{R})$, there will be a unique $\mathbb{Z}^3$-real number sequence $\{f_{k}\}_{k\in\mathbb{Z}^3}$ corresponding to it and the rate of decay of $|f_{k}|$ has been given by $(1)$ as $|k|\rightarrow\infty$. Conversely, a $\mathbb{Z}^3$-real sequence $\{f_{k}\}_{k\in\mathbb{Z}^3}$ satisfying Lemma \ref{Fourier} will determine a function $f$ in $C^r([0,1]^3,\mathbb{R})$.
\begin{defn}
We denote the space of $\mathbb{Z}^3$-real sequences by $\mathfrak{C}$ and the subspace of $\mathfrak{C}$ which could decide  $C^r$ functions by $\mathfrak{C}^r$.
\begin{enumerate}
\item A linear subspace of $\mathbb{Z}^3$ is called a {\bf Lattice}, which is written by $\Lambda$.  If we can find a group of irreducible base vectors generating $\Lambda$, then $\Lambda$ is called a {\bf greatest Lattice} and is denoted by $\Lambda^{max}$. We denote the space of all Lattices by $\mathfrak{L}$.
\item We call the linear operator $\mathscr{F}$ a {\bf Pickup} of $\mathfrak{C}^r$, if 
\[
\mathscr{F}: \mathfrak{C}^r\times\mathfrak{L}\rightarrow\mathfrak{C}^k\;via\;(f_k,\Lambda)\rightarrow\bar{f}_k,
\]
where $k\in\mathbb{Z}^3$ and $\bar{f}_k=\left\{\begin{array}{cccc}
	                                                                         f_k,k\in\Lambda\\
	                                                                         0,\;k\notin\Lambda
	                                                                         \end{array}\right.$
\item We call the linear operator $\mathscr{G}$  a {\bf Shear}, if
\[
\mathscr{G}: \mathfrak{C}^r\times\mathfrak{L}\times\mathbb{R}\rightarrow\mathfrak{C}^k\;via\;(f_k,\Lambda,K)\rightarrow\bar{f}_{k,K},
\]
where $k\in\mathbb{Z}^3$ and $\bar{f}_{k,K}=\left\{\begin{array}{cccc}
	                                                                        f_k,k\in\Lambda&|k|\geq K,\\
	                                                                         0,\quad else
	                                                                         \end{array}\right.$
	                       
\end{enumerate}
\end{defn}
Now we make use of these definitions to get our resonant plan. Without loss of generality, We can assume $\omega_0=(\omega_{0,1},\omega_{0,2})\in[0,1]\times[0,1]$. Take $l\in\mathbb{N}\setminus\{1\}$, we can get a $l$-partition of $[0,1]^2$, i.e. $l^2$ little squares which are diffeomorphic to $[0,\frac{1}{l}]^2$. We continue this process $m$-times and get squares of length $\frac{1}{l^m}$. We can always pick a proper $\frac{1}{l^m}$-lattice point $\omega_m$ with dist$(\omega_m,\omega_0)\in[\frac{\sqrt{2}}{l^{m+1}},\frac{\sqrt{2}}{l^m}]$ for each step, $m\in\mathbb{N}$. Additionally, we have 
\[
dist(\omega_m,\omega_0)<dist(\omega_n,\omega_0), \;\forall n<m.
\]
The following demonstration will give the readers a straightforward explanation for this:
\begin{demo}
We express $\sqrt{2}$ by its decimal fraction $\sqrt{2}=1.41421\cdots$, then $1.4,1.41,1.414,1.4142,1.41421,\cdots$ will become a candidate sequence of rational numbers. Once we have $\sqrt{2}=1.\underbrace{41421\cdots\ast}_{m-1},\underbrace{0,0,\cdots,0}_{n},\star,\cdots$, where $\ast,\star\in\{1,2,3,4,5,6,7,8,9\}$ and in this demonstration we can assume $\ast=4=\star$, then the $(m-1)-$th number of this sequence should be $1.\underbrace{41421\cdots3}_{m-1}$, the $(m-1+i)-$th one $1.\underbrace{41421\cdots3}_{m-1}\underbrace{99\cdots9}_{i}$, $i\in\{1,\cdots,n\}$ and the $(m+n)-$th one $1.\underbrace{41421\cdots\ast}_{m-1}\underbrace{00\cdots0}_{n}\star\cdots$. With the same rule, we can modify all the places several $0$s come out one by one.
\end{demo}
Now we get a sequence of 2-resonant points $\{\omega_m\}_{m\in\mathbb{N}}$ which approaches $\omega$ step by step. It's a slow but steady approximated process. Accordingly, we can find $(a_m,b_m)\in\mathbb{N}^2$ such that $\omega_m=(\frac{a_m}{l^m},\frac{b_m}{l^m})$. Between $\omega_m$ and $\omega_{m+1}$, along these partition lines, we can find a $\omega_{m,\frac{1}{2}}$ as a medium 2-resonant point (see figure \ref{fig3}), which can be expressed formally by $(\frac{a_m}{l^m},\frac{b_{m+1/2}}{l^{m+1}})$ or $(\frac{a_{m+1/2}}{l^{m+1}},\frac{b_{m}}{l^m})$, where $a_{m+1/2}$, $b_{m+1/2}\in\mathbb{N}$. We could only consider the $\omega_{m+1/2}$ of a former case in this paper.\\

\begin{figure}
\begin{center}
\includegraphics[width=6cm]{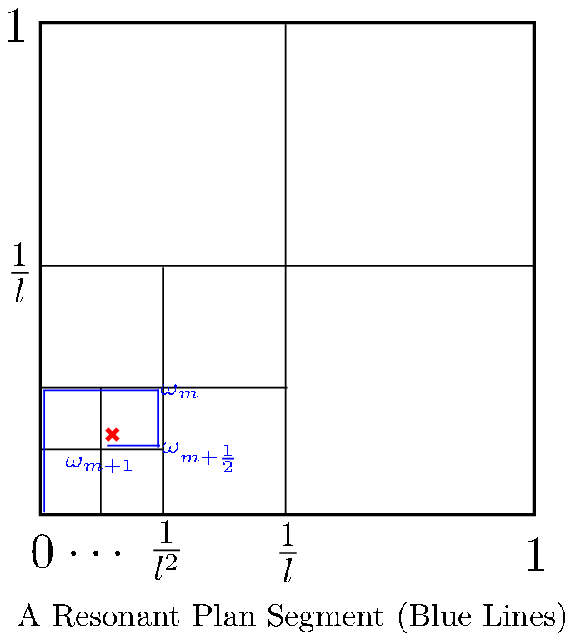}
\caption{ }
\label{fig3}
\end{center}
\end{figure}
Finally, we can connect all these $\{\omega_m,\omega_{m+1/2}\}_{m\in\mathbb{N}}$ along the Lattice lines and get one asymptotic resonant plan $\mathfrak{P}^{\omega}=\cup_{m\in\mathbb{N}}\{\Gamma_m^{\omega}\}$. We call $\Gamma_m^{\omega}\doteq\{\omega_m\stackrel{\Gamma_{m,1}^{\omega}}{\longrightarrow}\omega_{m+1/2}\stackrel{\Gamma_{m,2}^{\omega}}{\longrightarrow}\omega_{m+1}\}$ {\bf one-step transport process} and show its several fine properties.
\begin{defn}
In our case of 2.5 degrees of freedom, time variable will be involved in. Let ${\tilde{\omega}}\doteq(\omega,1)\in\mathbb{R}^3$ be the frequency, $\Lambda_m\doteq span_\mathbb{Z}\{(l^m,0,-a_m),(0,l^m,-b_m)\}$ be the Lattice vertical to $\tilde{\omega}_m$ and $\Lambda_{m+\frac{1}{2}}$ be the one vertical to $\tilde{\omega}_{m+1/2}$. The corresponding maximal Lattices can be denoted by $\Lambda_m^{max}$ and $\Lambda_{m+1/2}^{max}$. $d_m\doteq dist(\omega_m,\omega_0)$ and $d_{m+1/2}$ the distance of $\omega_{m+1/2}$ to $\omega_0$.
\end{defn}
It's obvious that $d_{m+1}\leq d_{m+1/2}\leq d_{m}$ and $d_m\leq\frac{\sqrt{2}}{l^m}$. Recall that $\frac{a_m}{l^m}$, $\frac{b_{m+1/2}}{l^{m+1}}$, $\frac{a_{m+1/2}}{l^{m+1}}$ and $\frac{b_m}{l^m}$ may be reducible, so $\Lambda_m= span_\mathbb{Z}\{(l^m,0,-a_m),(0,l^m,-b_m)\}$ and $\Lambda_{m+1/2}= span_\mathbb{Z}\{(l^m,0,-a_m),(0,l^{m+1},-b_{m+1/2})\}$ are unnecessarily maximal Lattices. Besides, we will face a new difficulty: there may be `stronger' 2-resonant obstructions in $\Gamma_m^{\omega}$,  i.e. $\omega*=(\frac{a_m}{l^m},\frac{c_m}{\iota_m})\in\Gamma_m^{\omega}$ with $l^m\geq\iota_m$, even $l^m\gg\iota_m$. Later we will transform these difficulties into several conditions of $\{f_k\}_{k\in\mathbb{Z}^3}$ and solve them, with the aforementioned `Pickup' and `Shear' operators.
	\begin{Lem}\label{Properties of resonance plan}
	\begin{enumerate}
	\item$ {\Gamma_{m,i}^{\omega}}_{m\in\mathbb{N}}$ are line segments parallel but not collinear with each other, $i=1,2$.
	\item $\Gamma_m^{\omega}\cap\Gamma_{m+1}^{\omega}=\omega_{m+1}$, and $\Gamma_m^{\omega}\cap\Gamma_n^{\omega}=\emptyset$, here $m\in\mathbb{N}$, $n\neq m\pm 1$.
	\item $\omega\neq\omega'$ and they both locate on $\mathfrak{P}^{\omega}$,  then $\Lambda_{\omega}\cap\Lambda_{\omega'}$ is either a one-dimensional Lattice, or $(0,0,0)\in\mathbb{Z}^3$. The former case happens iff $\omega$ lies on the same $\Gamma_{m,i}$ with $\omega'$, $i=1,2$. The latter case happens iff they lie on different $\Gamma_{m,i}$ segments.
	\end{enumerate}
	\end{Lem}
	\begin{proof}
	We omit the proof here since these can be easily deduced from our construction.
	\end{proof}
	In the next, we will find the Stable Normal Forms of system (\ref{5}) in different domains which are valid for all the resonant segments $\{\Gamma_m^p\}_{m\in\mathbb{N}}$ with a KAM iteration approach. Since $h(p)$ is strictly positive definite, we get the $\Gamma_m^p$ via a diffeomorphism from $\Gamma_m^{\omega}$. We just need to give the demonstration on $\Gamma_{m,1}$ and other resonant segments can be treated in the same way. This process can be operated for any $m\in\mathbb{N}$, so we could assume $m\gg1$ sufficient large.
\vspace{10pt}	
	
\section{Stable Normal Form and unified expression of Hamiltonian systems}

First, we need to divide the Stable Normal Form into 2-resonance and 1-resonance two different cases as $\omega\in\Gamma_{m,1}^{\omega}$ is differently chosen.
\subsection{2-Resonant case}
Except for the two end points $\omega_m$ and $\omega_{m+1/2}$, there are infinite many other 2-resonant points on $\Gamma_{m,1}^{\omega}$. We just need to consider finitely many of them, which we call `sub 2-resonant' points. In other words, we only consider the 2-resonant case $\omega^*=(\frac{a_m}{l^m},\frac{c_m}{\iota_m})$, and $\iota_m$ could be chosen any integer between $1$ and $l^{m(1+\xi)}$, here $\xi>0$ could be chosen a proper real number later.\\

For the KAM iteration's need, we expand system (\ref{5}) to an autonomous quasi-convex system as following:
\begin{equation}
\label{6}
H(p,q,I,t)=I+h(p)+p^{\sigma}f(q,t)\doteq\tilde{h}(p,I)+p^{\sigma}f(q,t).
\end{equation}
We can formally give one step KAM iteration to this system:
\begin{eqnarray*}
H_{+}=H\circ\Phi&=&H+\{H,W\}+\int_0^1(1-t)\{\{H,W\}, W\}\circ\phi_t(p,q)dt\\
&=&\tilde{h}+p^{\sigma}f(q,t)+\{\tilde{h},W\}+\{p^{\sigma}f(q,t), W\}\\
&+&\int_0^1(1-t)\{\{\tilde{h},W\}, W\}\circ\phi_t dt+h.o.t.\\
\end{eqnarray*}
Here $\Phi$ is an exact symplectic transformation defined in the domain $\mathcal{B}(p^{*},\delta_{+})\times\mathbb{T}^2\times\mathbb{S}$ of phase space, with $W$ its Hamiltonian function and $\Phi=\phi_{t=1}$ the time-1 mapping. $\nabla h(p^*)=\omega^*$ and $\delta_{+}\leq\delta$, with $\delta$ as the available radius at first. Recall that we could make $m\gg1$ sufficiently large and $D^2h(0)\doteq\mathbf{A}\sim\mathcal{O}(1)$ relative to $m$.\\

We could take $W=p^{\sigma}g(q,t)$ formally and $g(q,t)=\frac{1}{T_m^{*}}\int_0^{T_m^*}f(q+\omega^*t,t)tdt$. Here $T_m^*$ is the period of frequency $\omega^*$ and we know $T_m^*=lcm(l^m,\iota_m)\leq l^m\cdot\iota_m$. Then we can solve the cohomology equation and get:
\begin{eqnarray*}
H_{+}&=&\tilde{h}(p,I)+p^{\sigma}[f](q,t)+\langle\Delta\omega,p^{\sigma}\nabla_{\theta}g\rangle+\sigma p^{2\sigma-1}(f\nabla_{\theta}g-\nabla_{theta}f g)\\
&+&\int_0^1(1-t)\{\{\tilde{h},W\}, W\}\circ\phi_t dt+h.o.t
\end{eqnarray*}
Here $\theta\doteq(q,t)$ and $\Delta\omega\doteq\omega-\omega^*$. Again we recall that $p$, $\theta$, $f$ are $g$ all vectors. %We just use a simple symbol system as long as there's no ambiguity. 
Besides, we have
\[
\{\{\tilde{h},W\},W\}=(W_{tp}+h_{pp}W_q+\omega W_{qp})W_q-(W_{tq}+\omega W_{qq})W_p,
\]
of which only the value under $\|\cdot\|_{C^2,\mathcal{B}}$ norm we care.\\

Recall that it's just a formal derivation, so we need a list of conditions to make it valid. \\

$\bullet$ First, we should control the drift of action variable $p$, i.e. restricted in the domain $\mathcal{B}(p_{+}^*,\delta_{+})$, the quantity of drift doesn't exceed $\Delta\delta\doteq\delta-\delta_{+}$. Without loss of generality, we can assume $\int_{\mathbb{T}^3}f(q,t)dqdt=0$ and $p_{+}^*=p^*$. Then we need
\[
\|\int_0^1\frac{\partial W}{\partial q}\circ\phi_t dt\|_{\mathcal{B}(p^*,\delta_{+})}\leq\|W_q\|_{\mathcal{B}(p^*,\delta)}\leq\Delta\delta.
\]
As we know,
\begin{equation}\label{distance}
\frac{1}{l^{m}}\leq d_m^*\leq\frac{\sqrt{2}}{2l^m}, \;m\gg1.
\end{equation}
This is because the special resonant plan we choose. So we just need $\|p\|^{\sigma}\|g_q\|_{\mathcal{B}(p^*,\delta)}\leq\Delta\delta$. On the other side, 
\[
\|p\|^{\sigma}\leq(\|p-p^*\|+\|p^*\|)^{\sigma}\leq\sum_{i=0}^{\sigma}C_{\sigma}^{i}\|p-p^*\|^i\|p^*\|^{\sigma-i},
\] 
so we need 
\begin{equation}
\|p^{\sigma}g_q\|_{\mathcal{B}(p^*,\delta)}\leq c_2d_m^{*\sigma}T_m^*\|f\|_{C^2,\mathcal{B}}\leq\frac{\delta}{4},
\end{equation}
by taking $\Delta\delta=\frac{\delta}{4}$ and $\delta\leq d_m^*$. Here $c_2$ is a constant depending on $c_1$ and $\sigma$. Actually, the estimation here is rather loose, owing to the robustness of KAM method.\\

$\bullet$ Second, we must assure that the tail term $R(p,q,t)$ and the resonant term $Z(p,q,t)$ are strictly separated, i.e. $\|R\|_{C^0,\delta_{+}}\ll\|Z\|_{C^0,\delta_{+}}$. Here $Z(p,q,t)=p^{\sigma}[f](q,t)$, $R=R_1+R_2$, $R_1=\Delta\omega\cdot p^{\sigma}g_{\theta}$, and
\[
R_2=\sigma p^{2\sigma-1}(f\nabla_{\theta}g-\nabla_{theta}f g)+\int_0^1(1-t)\{\{\tilde{h},W\}, W\}\circ\phi_t dt+h.o.t.
\]
Actually, we have $R_2(p,q,t)=d_m^{*2\sigma-1}\tilde{R}_2(q,t)+h.o.t$, and
\begin{equation*}
\|\tilde{R}_2\|_{C^2}\leq T_m^{*2}\|f\|_{C^2},
\end{equation*}
\begin{equation*}
\|R_1\|_{C^2}\leq\delta d_m^{*\sigma} T_m^*\|f\|_{C^2}.
\end{equation*}
On the other side, the resonant term satisfies:
\begin{eqnarray*}
Z=p^{\sigma}[f]&=&p^{\sigma}\sum_{\substack{(k,l)\bot\tilde{\omega}^*\\
                                                                      (k,l)\in\mathbb{Z}^3\setminus\{0\}\\
                                                                }
                                                     }f_{k,l}\exp^{2\pi i(\langle k,q\rangle+l\cdot t)},                                          
\end{eqnarray*}
where $\tilde{\omega}^*=(\frac{a_m}{l^m},\frac{c_m}{\iota_m},1)$. Here $\frac{c_m}{\iota_m}$ is irreducible, but $\frac{a_m}{l^m}$ may be not. There are two aforementioned difficulties we should face:
\begin{enumerate}\vspace{5pt}
\item $\frac{a_m}{l^m}$ may be reducible and $\frac{a_m}{l^m}=\frac{\lambda a_m^{'}}{\lambda l_m^{'}}$. $l_m^{'}\ll l_m$ could even happen and make the estimation of $Z$ of $m$ ambiguous.\\
\item the case $\iota_m\ll l_m$ may happen. Later we will see that this may cause a big `obstruction' to the persistence of NHIC according to $Z_1$.\\
\end{enumerate}

We denote the maximal Lattice vertical to $\tilde{\omega}^*$ by $\Lambda_{\omega^*}^{max}=span_{\mathbb{Z}}\{\vec{e}_1,\vec{e}_2\}$, with $\vec{e}_1=(l_m^{'},0,-a_m^{'})$ and $\vec{e_2}=(0,\iota_m,-c_m)$. Then we can translate $Z$ into:
\begin{eqnarray*}
Z=p^{\sigma}[f]&=&p^{\sigma}[f]_1+p^{\sigma}[f]_2\\
                         &=&p^{\sigma}\sum_{l\in\mathbb{Z}}f_{l\vec{e}_1}\exp^{2\pi il(l_m^{'}q_1-a_m^{'}t)}\\ 
                         &+&p^{\sigma}\sum_{\substack{l,k\in\mathbb{Z}\\
                                                                            k\neq0}} f_{l\vec{e}_1+k\vec{e}_2}\exp^{2\pi i [l(l_m^{'}q_1-a_m^{'}t)+k(\iota_m q_2-c_m t)]}.
\end{eqnarray*}

We have $(l_m,0,-a_m)=\lambda\vec{e}_1$, $\Lambda_{\lambda\vec{e}_1}\doteq span_{\mathbb{Z}}\{\lambda\vec{e}_1\}$ and $\Lambda_{\omega^*}\doteq span_{\mathbb{Z}}\{\lambda\vec{e}_1,\vec{e}_2\}$. If (2) happens, $[f]_2$ may be much larger than $[f]_1$ by comparing their Fourier coefficients. So we need $f(q,t)\in\mathfrak{B}(0,c_1)$ to be well chosen and satisfies the following conditions:\\

{\bf C1:}  $f\in\mathfrak{B}(0,c_1)$ satisfies: $
\mathscr{F}(f,\Lambda_{\omega^*})=\mathscr{F}(f,\Lambda_{\omega^*}^{max})$, i.e. $f_{k}\equiv0$, $\forall k\in\Lambda_{\omega^*}^{max}\setminus\Lambda_{\omega^*}$.

{\bf C2:} $f\in\mathfrak{B}(0,c_1)$ satisfies: $\mathscr{G}(f,\Lambda_{\omega^*}^{max},l^m)=\mathscr{F}(f,\Lambda_{\omega^*}^{max})$, i.e. $f_{k}\equiv0$, $\forall k\in\Lambda_{\omega^*}^{max}$ and $|k|\leq l^m$.\\

Since $f\in\mathfrak{B}(0,c_1)$, we have $|f_k|\leq\frac{c_3}{|k|^r}$ from Lemma \ref{Fourier}. Here $c_3=c_3(r,c_1)$ is a constant. Later we also use a symbol $\lessdot$ ($\gtrdot$) to avoid too much $c_i$ constant involved, which means $\leq$ ($\geq$) by timing a $\mathcal{O}(1)$ constant on the right side. These symbols are firstly used by J. p\"oschel in \cite{P}.\\

Based on the two Fourier conditions above, we will give the first uniform restriction on $[f]_1$.\\

{\bf U1:} As a single-variable function of $\langle \lambda\vec{e}_1, \theta\rangle$, $p^{*\sigma}[f]_1$ has a unique maximal value point, at which it is strictly nondegenerate with an eigenvalue not less than $\frac{c_4 d_m^{*\sigma}}{l^{m(r+2)}}$, since $m\geq M\gg1$. Here $c_4\geq1$ is a new constant.
\begin{Rem}
This restriction assures the strength of normal hyperbolicity corresponding to the main direction of $\Gamma_{m,1}$. Its order of $m$ is controllable and uniform. It's a necessary demand to resist infinitely many cusp remove.
\end{Rem}
\begin{Rem}
We also recall that the index $r+2$ in {\bf U1} can be replaced by any $r+\zeta$ ($\zeta\geq2$), but new conditions of $\zeta$ will be involved in to assure $|R|\ll|Z|$.  The stronger hyperbolicity is, the easier to assure the existence of NHICs. So we just consider the case of $r+2$.
\end{Rem}
\begin{Rem}
During the whole $\Gamma_{m,1}$ except $\omega_{m+1/2}$, transformations between different resonant lines are not involved in. We just need to construct a NHIC `transpierce' the whole $\Gamma_{m,1}$, so we don't give any restriction to $[f]_2$ temporarily (see figure \ref{fig4}).
\end{Rem}
\begin{figure}
\begin{center}
\includegraphics[width=8cm]{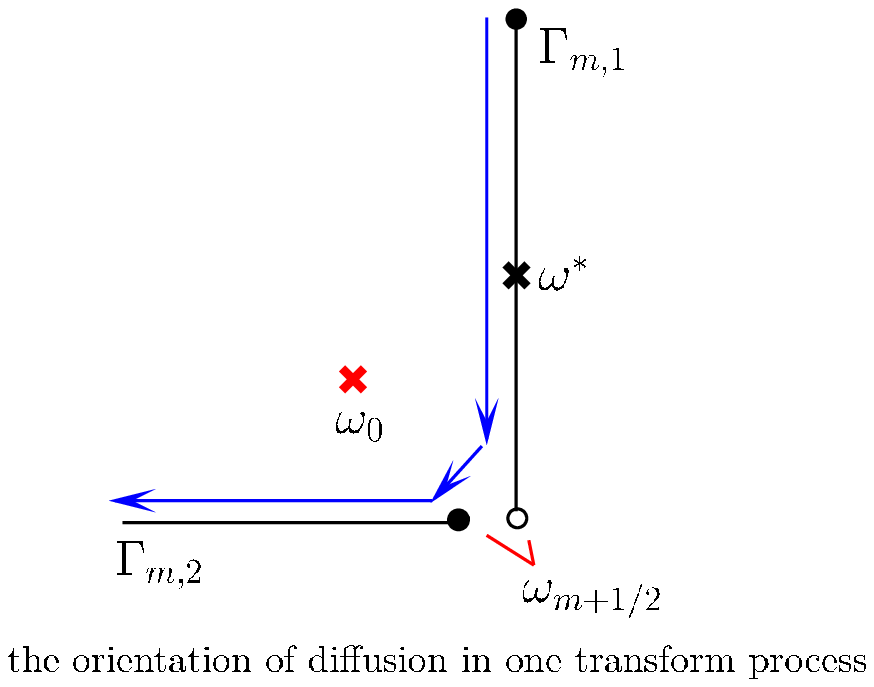}
\caption{ }
\label{fig4}
\end{center}
\end{figure}
Now let's satisfy the two bullets above and finish this KAM iteration with:
\begin{eqnarray}
d_m^{*\sigma-1}T_m^*\lessdot\delta,\vspace{6pt}\\
\delta T_m^*\ll d_m^{*r+2},\vspace{6pt}\\
\delta\leq d_m^{*}, \vspace{6pt}\\
m\geq M\gg1.\vspace{6pt}
\end{eqnarray}
We can sufficiently take 
\begin{equation}\label{index1}
\sigma>r+2
\end{equation}
 and 
\begin{equation}\label{radius}
\delta\ll l^{-m(r+4+\xi)}.
\end{equation}
%Once $\omega_0$ is chosen, Diophantine index $\tau$ is actually fixed. So we can assume $r\geq\tau$ without loss of generality. Then we just need
%\[
%\sigma>r+7+2\xi,
%\]
%and
%\[
%\delta\ll l^{-m(r+4+\xi)}.
%\]
Recall that $\sigma$ can be chosen properly large from Lemma \ref{robust}.
\begin{Rem}
In the above process, we left an index $\xi>0$ not dealt with. We also know that the larger $\xi$ is, the more `sub 2-resonant' points we should consider. However, we hope the number of these points as little as possible, since the  diffusion mechanism of 2-resonance is much complex than that of 1-resonance. In other words, we'll apply 1-resonant mechanism along $\Gamma_{m,1}$ as much as possible. That needs an estimation of the lower-bound of $\xi$ in the next subsection.
\end{Rem}
\vspace{10pt}

\subsection{1-resonant case}\label{1-resonance}
We first revise several symbols which are valid only in this subsection. As Figure \ref{fig5} shows us, $\Gamma_{m,1}$ is devided into several 1-resonant segments by sub 2-resonant points. Of each segment $\mathscr{S}$ we could find a tube-neighborhood with radius $\delta$, on which we hope to get a similar Stable Normal Form by KAM iterations. Let $\delta_{+}$ be the radius of ball-neighborhood of 2-resonant points, for which the restriction (\ref{radius}) holds. In order to make the tube-neighborhoods approach 2-resonant points as near as possible, we will choose $\delta_{+}$ as less as possible under the premise that 2-resonant Stable Mornal Form is valid.\\

%Recall that there's an abuse of $\delta$ both in frequency space of $\omega$ and action variable space of $p$. But the strictly positive definiteness of $\mathbf{ A}$ supplies a diffeomorphism between the two.\\

We devide $f(q,t)=\sum_{(k,l)\in\mathbb{Z}^3}f_{k,l}\exp^{i2\pi(\langle k,q\rangle+lt)}$ into $ T_{K}f$ and $R_K f$, to express the partial sums of Fourier series of $|(k,l)|\leq K$ and $|(k,l)|>K$. From Lemma \ref{Fourier} we know that $\|R_K f\|_{C^2}\leq \kappa_3 K^{-r+6}\|f\|_{C^r}$. Still we could give a formal KAM iteration in the tube-neighborhood of $\mathscr{S}$:
\begin{eqnarray*}
H_{+}=H\circ\Phi&=&H+\{H,W\}+\int_0^1(1-t)\{\{H,W\},W\}\circ\phi_t(p,q)dt\\
&=&\tilde{h}+p^{\sigma}f+\{\tilde{h}, W\}+\{p^{\sigma}f, W\}\\
&\quad\quad+&\int_0^1(1-t)\{\{\tilde{h},W\},W\}\circ\phi_tdt+h.o.t\\
&=&\tilde{h}+p^{\sigma}T_{l^{m(1+\xi)}}f+\tilde{\omega}W_{\theta}+\{p^{\sigma}f,W\}+p^{\sigma}R_{l^{m(1+\xi)}}f\\
&\quad\quad+&\int_0^1(1-t)\{\{\tilde{h},W\},W\}\circ\phi_tdt+h.o.t.
\end{eqnarray*}
Here $\Phi$ is an exact symplectic transformation defined in the domain $\mathcal{T}(\mathscr{S},\delta)\times\mathbb{T}^2\times\mathbb{S}$ of phase space, with $W$ its Hamiltionian and $\Phi=\phi_{t=1}$ the time-1 mapping.
\begin{figure}
\begin{center}
\includegraphics[width=13cm]{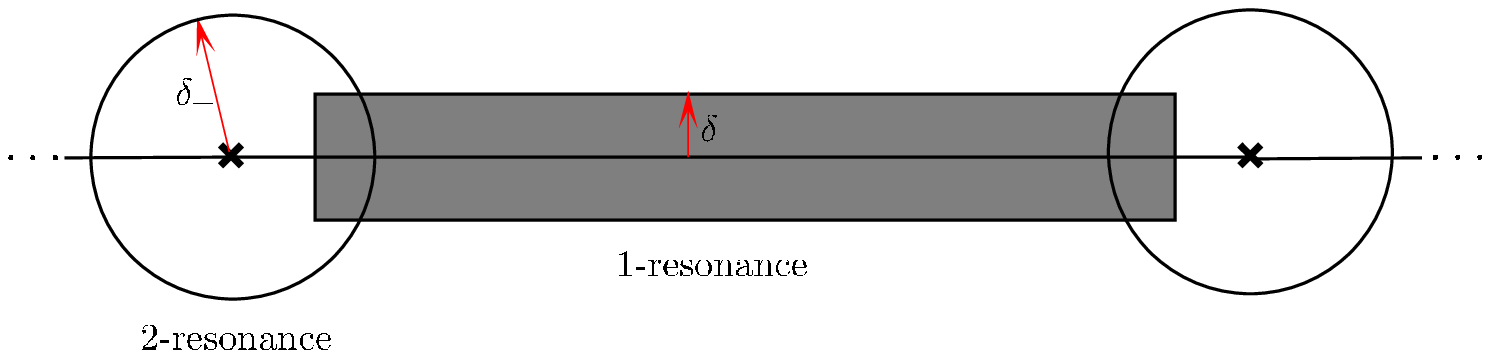}
\caption{ } 
\label{fig5}
\end{center}
\end{figure}
We denote by $R_{\Lambda}\doteq\{\omega\in\mathbb{R}^3|\langle k,\omega\rangle=0,\forall k\in\Lambda\subset\mathbb{Z}^3, \omega_3\equiv 1\}$ the set of frequencies vertical to $\Lambda$-Lattice, then $\Gamma_{m,1}\subset R_{\Lambda_{\lambda\vec{e}_1}}$. Recall that we only consider the $\omega$ of which there exists a $k'\in\mathbb{Z}^3$ not in $\Lambda_{\vec{e}_1}$, such that $k'\bot\omega$ and $|k'|\leq l^{m(1+\xi)}$. We denote by $R_{\Lambda_{\lambda\vec{e}_1,+}^{\xi}}$ the set of all these frequencies. Then we can solve the cohomology equation formally in the domain $\mathcal{T}(\Gamma_{m,1},\delta)\setminus\mathcal{B}(R_{\Lambda_{\lambda\vec{e}_1,+}^{\xi}}\cap\Gamma_{m,1},\delta_{+})$ and get the resonant term 
\[
Z(q,t)=p^{\sigma}\sum_{\substack{(k,l)\in\Lambda_{\lambda\vec{e}_1}\\
|(k,l)|\leq l^{m(1+\xi)}}}f_{k,l}\exp^{i2\pi(\langle k,q\rangle+lt)},
\]
owing to {\bf C1} and {\bf C2} conditions.\\

To ensure this formal KAM iteration valid, we also face the two difficulties as the bullet parts in previous subsection of 2-resonance.\\

$\bullet$ First, we need to control the drift value of action variable $p$. Let's sufficiently take
\begin{equation}
\|\int_0^1\frac{\partial W}{\partial q}\circ\phi_t dt\|_{C^0,\frac{\delta}{2}}\leq\|W_q\|_{C^0,\delta}\leq\frac{\delta}{2}.
\end{equation}
Here we directly shrink the radius of tube-neighborhood of $\mathscr{S}$ to $\frac{\delta}{2}$. Also we know that 
\[
W=p^{\sigma}\sum_{\substack{
(k,l)\notin\Lambda_{\lambda\vec{e}_1}\\
|(k,l)|\leq l^{m(1+\xi)}}}\frac{f_{k,l}}{2\pi i \langle \tilde{\omega}, (k,l)\rangle}\exp^{i2\pi(\langle k,q\rangle+lt)}.
\]

$\bullet$ Second, we need $\|R\|_{C^0,\delta}\ll\|Z\|_{C^0,\delta}$. As the case of 2-resonance, we need a uniform condition of $Z$ here:
\vspace{10pt}

{\bf U2:} As a single-variable function of $\langle \lambda\vec{e}_1, \theta\rangle$, $\forall p\in\mathscr{S}$, $-p^{\sigma}[T_{l^{m(1+\xi)}}f]_1$ has a unique minimal value point, which is strictly nondegenerate with the eigenvalue not less than $\frac{c_5 d_m^{\sigma}}{l^{m(r+2)}}$ ($m\geq M\gg1$). Here $c_5=c_5(c_4,\xi)\geq\frac{1}{2}$ is a new constant, and $d_m$ is the distance between $p\in\Gamma_{m,1}^p$ and $p_0$.
%\vspace{10pt}

\begin{Rem}
Here {\bf U2} can be considered as a reinforcement of {\bf U1}, which helps us avoid a complicated 1-resonant bifurcation problem which is faced in \cite{Ch} and \cite{BKZ}. A special example satisfying {\bf U2} is to take $p^{\sigma}f(q,t)=(p_1^2+p_2^2)^{\sigma/2}\cos(\langle\theta,\lambda\vec{e}_1\rangle)$ with even $\sigma$. We can see that in this case NHIC corresponding to $[f]_1$ does exist as a single connected cylinder without  decatenation.\\

Actually, we can loosen {\bf U2} to a general condition with bifurcation points (see Appendix \ref{U2}). To avoid the verbosity of narration, we still assume {\bf U2} as a proper uniform restriction. Later, in the Appendix \ref{U2} we will deal with the bifurcated case use a genericity developed in \cite{Ch}.
\end{Rem}

\begin{defn}
Let $P_{\Lambda}:\mathbb{R}^3\rightarrow\mathbb{R}^3$ be the projection of a vector ${\omega}$ to the real-expanded space  $span_{\mathbb{R}}\{\Lambda\}$ and $P_{\Lambda_{+}}$ be the projection to $span_{\mathbb{R}}\{\Lambda,k\}$, if there exists $k\in\mathbb{Z}^3$ not lie on $\Lambda$.
\end{defn}

In the 1-resonant situation we have $\Lambda=\Lambda_{\lambda\vec{e}_1}$, $P=P_{\Lambda_{\lambda\vec{e}_1}}$ and $P_{+}=P_{\Lambda_{\lambda\vec{e}_1,+}}$. We don't care the concrete form of $k$, but $|k|\leq l^{m(1+\xi)}$ is demanded in the following estimation. $\forall \tilde{\omega}\in\mathcal{T}(R_{\Lambda_{\lambda\vec{e}_1}}\cap\Gamma_{m,1},\delta)\setminus\mathcal{B}(R_{\Lambda_{\lambda\vec{e}_1,+}^{\xi}},\delta_{+})$, we have:
\begin{eqnarray*}
\langle k,\tilde{\omega}\rangle=\langle k, P_{+}\tilde{\omega}\rangle&=&\langle Qk,Q\circ P_{+}\tilde{\omega}\rangle+\langle Pk,P\circ P_{+}\tilde{\omega}\rangle\\
&=&\langle Qk,(P_{+}-P)\tilde{\omega}\rangle+\langle Pk,P\tilde{\omega}\rangle,
\end{eqnarray*}
with $Q=Id-P$, then
\[
|\langle k,\tilde{\omega}\rangle|\geq |\langle Qk,(P_{+}-P)\tilde{\omega}\rangle|-|\langle Pk,P\tilde{\omega}\rangle|.
\]
Notice that $Qk$ is parallel to $(P_{+}-P)\tilde{\omega}$ and $Pk$ is parallel to $P\tilde{\omega}$. We actually get
\[
|\langle k,\tilde{\omega}\rangle|\geq |Qk||(P_{+}-P)\tilde{\omega}|-| Pk||P\tilde{\omega}|,
\]
and furthermore
\begin{equation}\label{small denominator}
|\langle k,\tilde{\omega}\rangle|\geq\frac{\sqrt{\delta_{+}^2-\delta^2}}{\sqrt{l^{2m}+a_m^2}}-l^{m(1+\xi)}\delta.
\end{equation}
We can write the right side of above inequality by $\alpha$. That's the so called `small denominator' problem, so we need $\alpha>0$ as large as possible.
\begin{Rem}
This estimate of $\langle k,\omega\rangle$ was firstly given in \cite{P2}. Here is just a direct application of that.
\end{Rem}
 On the other side, we can estimate the tail term $R$ by:
 \begin{equation}\label{1-resonant 1}
 \|R_1\|_{C^2}=\|p^{\sigma}R_{l^{m(1+\xi)}}f\|_{C^2}\lessdot\frac{\kappa_3 d_m^{\sigma}\|f\|_{C^2}}{l^{m(1+\xi)(r-6)}},
 \end{equation}
 If we take $R_2\doteq\{p^{\sigma}f,W\}+\int_0^1(1-t)\{\{H,W\},W\}\circ\phi_t(p,q)dt+h.o.t$, then we have
\begin{equation}\label{1-resonant 2}
\|R_2\|_{C^2}\leq d_m^{2\sigma-1}\frac{\|f\|_{C^2}}{\alpha^2}l^{2m(1+\xi)}.
\end{equation}
Recall that $d_m$ is the distance between $p\in\mathscr{S}$ and $p_0$, and formula (\ref{distance}) is still valid. Based on {\bf U2}, 
%$\forall p\in\mathscr{S}$, we have $Z$ resonant term have a unique maximal-value point as a one-variable function of $\langle \lambda\vec{e}_1,\theta\rangle$, at which $Z$ is strictly non-degenerate  with a eigenvalue not less than $\frac{c_5 d_m^{\sigma}}{l^{m(r+2)}}$, $m\gg M$ sufficient large and $c_5>\frac{1}{2}$ a constant.
we need the followings to ensure the KAM iteration valid:
\vspace{6pt}
\begin{enumerate}
\item $\|R_1\|_{C^2}\ll\|Z\|_{C^2}\;\Rightarrow\;(r-6)(1+\xi)>r+2\;\Rightarrow\;\xi>\frac{8}{r-6}$,
\vspace{6pt}
\item $\|R_2\|_{C^2}\ll\|Z\|_{C^2}\;\Rightarrow\;\frac{d_m^{\sigma-1}}{\alpha^2}\ll l^{-m(r+4+2\xi)}$,
\vspace{10pt}
\item $\frac{d_m^{\sigma}}{\alpha}l^{m(1+\xi)}\lessdot\delta\quad$ (drift value control).
\end{enumerate} 
We can roughly take $\alpha\sim\mathcal{O}(\frac{\delta_{+}}{l^m})$. Recall that $\delta_{+}\ll l^{-m(r+4+\xi)}$ from (\ref{radius}) and $d_m\sim\mathcal{O}(l^{-m})$, then we have:
\begin{eqnarray*}
\delta_{+}^{-1}l^{m(2+\xi)}d_m^{\sigma}\lessdot\delta,\\
\delta_{+}\gg d_m^{\frac{\sigma-1}{2}}l^{\frac{m}{2}(6+r+2\xi)},\\
\delta\cdot l^{m(2+\xi)}\lessdot\delta_{+},\\
\xi>\frac{8}{r-6}.
\end{eqnarray*}
These can be further transformed into 
 \begin{equation}\label{1-resonant 3}
l^{-\frac{m(\sigma-7-r-2\xi)}{2}}\ll \delta_{+}\ll l^{-m(r+4+\xi)},
 \end{equation}
 \begin{equation}\label{1-resonant 4}
 l^{-m(\sigma-2-\xi)}\frac{1}{\delta_{+}}\lessdot\delta\lessdot l^{-m(2+\xi)}\delta_{+},
 \end{equation}
\begin{equation}
\xi>\frac{8}{r-6}.
\end{equation}
So we need the following index inequalities:
\begin{eqnarray}
r+4+\xi<\frac{1}{2}(\sigma-7-r-2\xi),\\
\xi>\frac{8}{r-6}.
\end{eqnarray}
A new index restriction of
\begin{equation}
\sigma>3r+4\xi+15
\end{equation}
will replace formula (\ref{index1}).\\

Here we give a lower bound for $\xi>\frac{8}{r-6}$. As is known from the previous subsection, we'll take $\xi$ as small as possible, actually $\xi=\frac{9}{r-6}$ is enough. We can see that $\xi\rightarrow 0$ as $r\rightarrow\infty$. On the other side, we know the strict lower bound of $\delta_{+}$ is $l^{-\frac{m(\sigma-7-r-2\xi)}{2}}$, and $\|Z\|_{C^2}\sim\mathcal{O}(l^{-m(\sigma+r+2)})$. So we can roughly estimate the order relationship by 
\begin{equation}\label{index2}
\inf\delta_+\sim\mathcal{O}(\|Z\|_{C^2}^{\frac{\sigma-7-r-2\xi}{2(\sigma+r+2)}}).
\end{equation}
We can see that the index in the above formula tends to $\frac{1}{2}$ as $\sigma\rightarrow+\infty$. Later we'll see that the index `$\frac{1}{2}$' plays a key role in the `homogenized' method, which was firstly used in \cite{Ch} and \cite{Mar}. Nonetheless, we can take $\sigma$ properly large such that the index of (\ref{index2}) greater than $\frac{1}{6}$. Similar estimation was also obtained in \cite{Ch} and \cite{BKZ}, with an $(\epsilon,\delta)$-language.
\begin{Rem}
This approach of Stable Normal Form was firstly developed by Lochak P. and P\"oschel J. in \cite{Lo} and \cite{P2}, in solving a Nekhoroshev estimation problem. Notice that we can get a better estimation with more steps of KAM iterations, but we will face some new difficulties: the loss of regularity and non-linearity of operators $\mathscr{F}$ and $\mathscr{G}$ about $f$. To avoid the technical verbosity, only one-step iteration is operated in this paper. This is enough for our construction and makes the whole proof easy to read.
\end{Rem}

Anyway, we can get Stable Normal Forms for both 1-resonance and 2-resonance, in the domain covering the whole $\Gamma_{m,1}^p\times\mathbb{T}^2\times\mathbb{S}^1$. Similarly, we can repeat this process for $\Gamma_{m,2}\times\mathbb{T}^2\times\mathbb{S}$, and then the whole resonant plan $\mathfrak{P}^{\omega}$. During this process, new versions of {\bf U1} and {\bf U2} can be raised on $\Gamma_{m,2}$ parallelly.
\vspace{10pt}

\subsection{Canonical coordinate transformations for Stable Normal Forms}

Recall that the resonant term $Z(p,q,t)$ of the Stable Normal Form is resonant with respect to $\omega$, the current frequency, so we have $Z=Z(p,\langle \vec{e}_1,\theta\rangle)$ (1-resonance) or $Z=Z(p,\langle \vec{e}_1,\theta\rangle,\langle\vec{e}_2,\theta\rangle)$ (2-resonance). So we can transform the corresponding Stable Normal Form into a canonical form which is universal for the whole $\mathfrak{P}^{\omega}$.\vspace{10pt}

{\bf $\bullet$ 2-Resonant Case} 
\vspace{10pt}

We know the Stable Normal Form of this case is
\begin{equation}\label{2-resonant stable}
H=h(p)+I+p^{\sigma}([f]_1+[f]_2)+R(p,q,t),\;on\;\mathcal{B}(p^*,\delta)\times\mathbb{R}\times\mathbb{T}^2\times\mathbb{S}.
\end{equation}
Here $[f]_1$ only depends on $\langle\lambda\vec{e}_1,\theta\rangle$, and $[f]_2$ depends on $\langle\lambda\vec{e}_1,\theta\rangle$ and $\langle\vec{e}_2,\theta\rangle$. Recall that $\iota_m$ varies from $1$ to $l^{m(1+\xi)}$, which brings some difficulties to our canonical transformation. Actually, this canonical transformation is a linear symplectic matrix, so we need the following condition to make the elements of matrix homogeneous.\\

{\bf C2' :}  If $\mu=\min_{k\in\mathbb{Z}^{+}}\{|k\vec{e}_2|>l^m\}$ and $\Lambda_{\omega^*}^{\lambda,\mu}\doteq span_{\mathbb{Z}}\{\lambda\vec{e}_1,\mu\vec{e}_2\}$, we take $f\in\mathfrak{B}(0,c_1)$ satisfying $\mathscr{F}(f,\Lambda_{\omega^*})=\mathscr{F}(f,\Lambda_{\omega^*}^{\lambda,\mu})$.\\

Let $\vec{e}_3=(0,0,\frac{1}{\mu\iota_m l^m})$ and
\begin{equation}
\Xi\doteq(\lambda\vec{e}_1,\mu\vec{e}_2,\vec{e}_3)^{t}=\begin{pmatrix}
l^m & 0 & -a_m\\
0 & \mu\iota_m & -\mu c_m\\
0 & 0 & \frac{1}{\mu\iota_m l^m}
\end{pmatrix}_{3\times3}
\end{equation}
be a unimodular matrix. We can get a symplectic transformation via:
\begin{equation}
\begin{pmatrix}
x\\s
\end{pmatrix}=\Xi
\begin{pmatrix}
q\\t
\end{pmatrix},\quad
\begin{pmatrix}
p\\I
\end{pmatrix}=\Xi^t
\begin{pmatrix}
y\\J
\end{pmatrix}+
\begin{pmatrix}
p^*\\0
\end{pmatrix}.
\end{equation}
Under this transformation, we can change system (\ref{2-resonant stable}) into
\begin{equation}\label{2-resonant canonical}
H=\frac{J}{\mu\iota_m l^m}+h'(y_1,y_2)+p_m^{*\sigma}[f]_1(x_1)+p_m^{*\sigma}[f]_2(x_1,x_2)+R'(x,y,s),
\end{equation}
with $(y_1,y_2)\in[-\frac{\delta}{l^m},\frac{\delta}{l^m}]\times[-\frac{\delta}{\mu\iota_m},\frac{\delta}{\mu\iota_m}]$. Here we move the higher order terms of $Z$ into the tail term and get a new $R'$. Besides, we have
\[
\|R'\|_{C^0}\lessdot\mathcal{O}(\delta d_m^{\sigma-2}T_m^*),
\]
and
\[
h'(y)=h'(0)+\nabla h'(0)y+\frac{1}{2!}D^2h'(0)y^2+\frac{1}{3!}D^3h'(0)\cdots
\]
Witout loss of generality, we can assume $h'(0)=0$. We also have $\nabla h'(0)=0$ of this formula, and 
\begin{equation}\label{rescale}
D^n h'(0)=\langle \Theta, D^n h(p^*)\underbrace{\Theta^t\rangle,\Theta^t\rangle,\cdots,\Theta^t\rangle}_{n-1},
\end{equation}
where
\[
\Theta=
\begin{pmatrix}
l^m & 0\\
0 & \mu\iota_m
\end{pmatrix}
\]
is an amplified matrix.
\begin{Rem}
In the later paragraph, we often apply the rescaled system $H'=\mu\iota_m l^m H$ which is more convenient. 
%Actually $H'=H'(x,y,t,I,)$, with $t$ still the time variable.
\end{Rem}

\vspace{10pt}
{\bf $\bullet$ 1-Resonant Case}
\vspace{10pt}

In the same way, we know the Stable Normal Form of this case is
\begin{equation}\label{1-resonant stable}
H=h(p)+I+p^{\sigma}[f]_1(\langle\lambda\vec{e}_1,\theta\rangle)+R(p,q,t),
\end{equation}
where $(p,I,q,t)\in\mathcal{T}(R_{\Lambda_{\lambda\vec{e}_1}},\delta)\setminus\mathcal{B}(R_{\Lambda_{\lambda\vec{e}_1,+}},\delta_{+})\times\mathbb{R}\times\mathbb{T}^2\times\mathbb{S}$. For each segment $\mathscr{S}$ between two 2-resonant points, we can find a finite sequence of open balls to cover it, i.e. $\mathscr{S}\subset\{\mathcal{B}(p_i^*,\delta)\}_{i=1}^{N_m}$ (see Figure \ref{fig6}). Recall that $\omega_i^*$ is a 1-resonant frequency corresponding to $p_i^*$. In the domain $\mathcal{B}(p_i^*,\delta)\times\mathbb{R}\times\mathbb{T}^2\times\mathbb{S}$ we can find a similar linear symplectic transformation with
\begin{figure}
\begin{center}
\includegraphics[width=13cm]{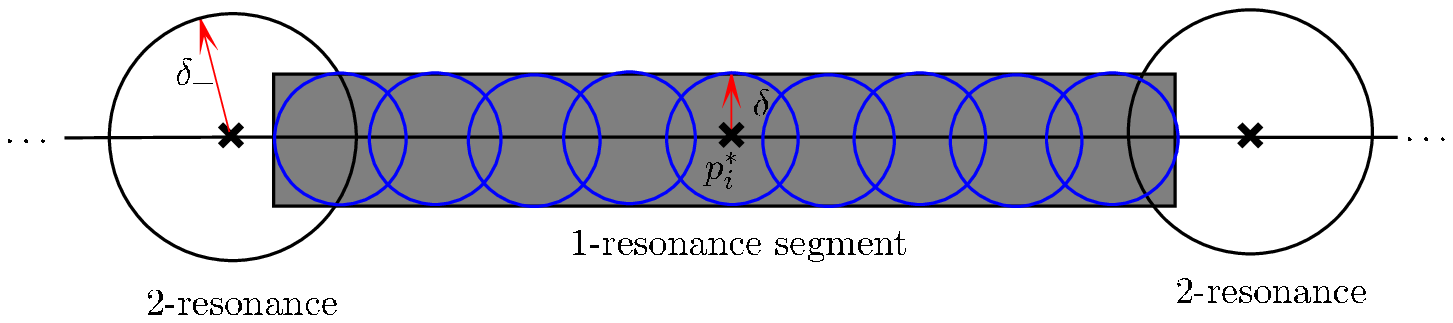}
\caption{ }
\label{fig6}
\end{center}
\end{figure}
\begin{equation}
\Xi=\begin{pmatrix}
l^m & 0 & -a_m\\
0 & l^m & 0\\
0 & 0 & \frac{1}{l^{2m}}
\end{pmatrix},
\end{equation}
and
\begin{equation}
\begin{pmatrix}
x\\t
\end{pmatrix}=
\Xi
\begin{pmatrix}
q\\t
\end{pmatrix},\quad
\begin{pmatrix}
p\\I
\end{pmatrix}=
\Xi^t
\begin{pmatrix}
y\\J
\end{pmatrix}+
\begin{pmatrix}
p_i^*\\0
\end{pmatrix}.
\end{equation}
Based on this transformation, system (\ref{1-resonant stable}) will become:
\begin{equation}\label{parameter system}
H=\frac{J}{l^{2m}}+h'(y,p_i^*)+p_i^{*\sigma}[f]_1(x_1)+R'(x,y,s),
\end{equation}
where $y\in\mathcal{B}(0,\frac{\delta}{l^m})$. Here $p_i^*$ is a parameter to mark the ball neighborhoods in which we apply the transformation. Also we throw the higher order terms of $Z$ into tail terms and get a new $R'$. Notice that
\[
\nabla h'(0,p_i^*)=(0,l^m\omega_{i,2}^*),\quad \omega_i^*=(\omega_{i,1}^*,\omega_{i,2}^*), 
\]
and
\[
D^n h'(0,p_i^*)=l^{nm}D^n h(p_i^*).
\]

%\begin{Rem}
%The same with the 2-resonant case above, we have a rescaled system $H'(x,y,t,I)=l^{2m}H$. Sometimes this system is much convenient than $H$.
%\end{Rem}
\begin{Rem}
From  system (\ref{parameter system}), we can see that $h'(y,p_i^*)$ and $Z(x_1,p_i^*)$ depend on a parameter variable $p_i^*$. 
%Thanks to {\bf U2} conditions, we can avoid the complicated bifurcation problem as parameter $p_i^*$ changes. 
Later, in the second part of Appendix we will deal with the bifurcation cases with a genericity of \cite{Ch}.
\end{Rem}
\vspace{10pt}

\subsection{Transition from 1-resonance to 2-resonance}\label{1 to 2}
$\newline$

Along $\Gamma_{m,1}$, we need to construct diffusion orbits with the frequency changing from $\omega_m$ to $\omega_{m+1/2}$. As $\omega_{m+1/2}$ is a typical 2-resonant point, once the diffusion orbits pass it, we can repeat this process along $\Gamma_{m,2}$ because {\bf U1} and {\bf U2} are still valid for it. So the key to achieve this process is to  overpass $\omega_{m,1/2}$. Once one step transport process $\Gamma_m$ is finished, all the transition plan $\mathfrak{P}^{\omega}$ could be overcome because of our self-similar structure.\\

Recall that there still exists finitely many `sub 2-resonant' points at $\Gamma_{m,1}$ to be overpassed. Different from $\omega_{m+1/2}$, there's no transitions between resonant lines at these points. So our plan is to find a NHIC with $(x_2,y_2)$ as `fast-variables' (in a rough sense) to overcome these points which persist under small perturbations (see Figure \ref{fig7}).\\
\begin{figure}
\begin{center}
\includegraphics[width=7cm]{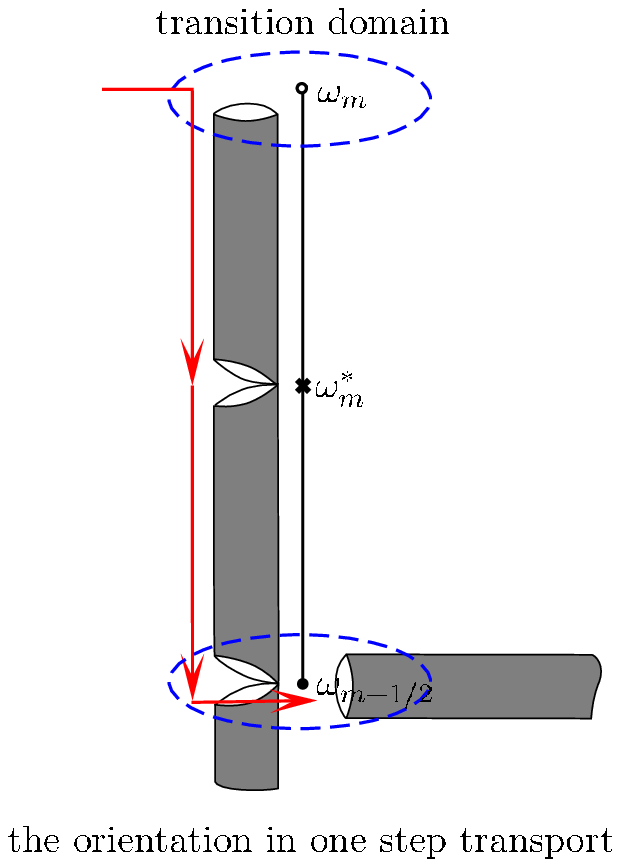}
\caption{ }
\label{fig7}
\end{center}
\end{figure} 

To achieve these, we need to weaken the hyperbolicity corresponding to fast variables $(x_2,y_2)$ for the sub 2-resonant points, and create a proper domain in which different NHICs can be connected with each other for $\omega_{m+1/2}$. More uniform conditions and a `weak-coupled' mechanism will be involved in this section. \\

From (\ref{index2}) we know, the index is contained in $[\frac{1}{6},\frac{1}{2})$. So there must be a overlapping domain in which both the Stable Normal Forms of 1-resonance and of 2-resonance valid. We can carry out one-step KAM iteration again in a domain $(\mathcal{B}(p_m^*,\delta_{+})\setminus\mathcal{B}(p_m^*,\mathbb{K}\|Z\|_{C^2}^{\frac{1}{2}}))\cap\mathcal{T}(\Gamma_{m,1}^p,\|Z\|_{C^2}^{\frac{1}{2}})$ (see Figure \ref{fig8}). Here $\mathbb{K}\gg1$ will be a posteriori constant determined later.
\begin{figure}
\begin{center}
\includegraphics[width=11cm]{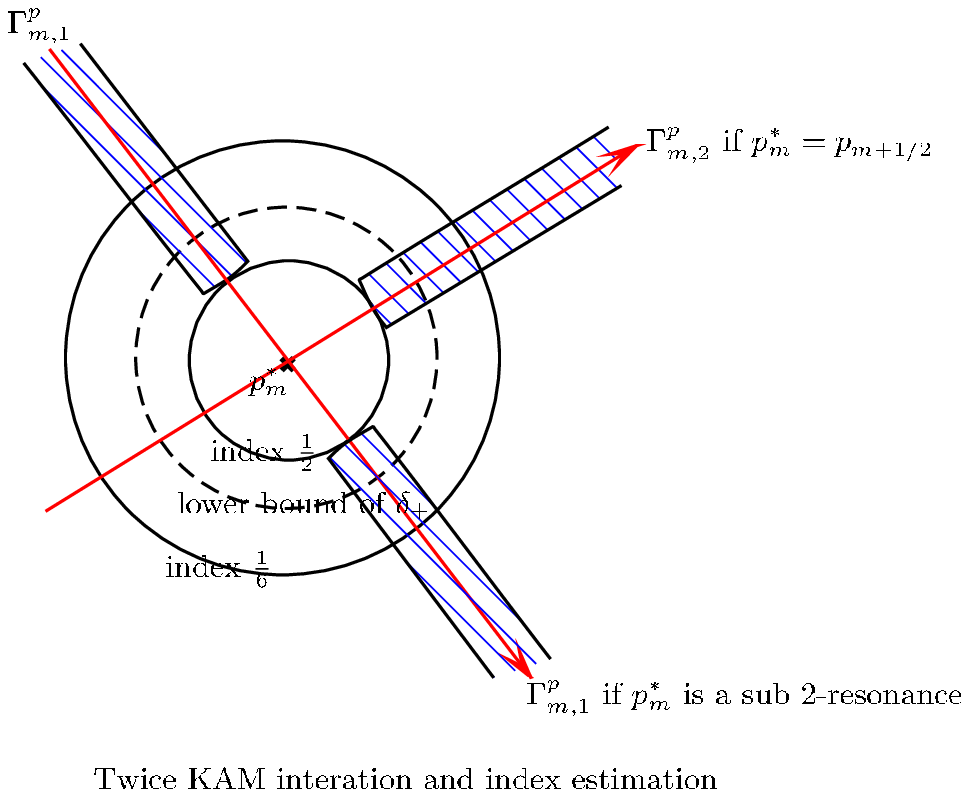}
\caption{ }
\label{fig8}
\end{center}
\end{figure}
In this domain, $H'$ system corresponding to (\ref{2-resonant canonical}) can be rewritten as:
\[
H'=J+\mu\iota_m l^m(h'(y)+p_m^{*\sigma}[f]_1(x_1)+p_m^{*\sigma}[f]_2(x_1,x_2)+R'(x,y,t)),
\] 
with $y\in[-\frac{\delta_{+}}{l^m},\frac{\delta_{+}}{l^m}]\times[-\frac{\delta_{+}}{\mu\iota^m},\frac{\delta_{+}}{\mu\iota^m}]$. We can divide $[f]_2$ into:
\begin{eqnarray*}
[f]_2(x)&=&\sum_{\substack{l,k\in\mathbb{Z}\\
                                                   k\neq0}} f_{l\lambda\vec{e}_1+k\mu\vec{e}_2}\exp^{2\pi i [lx_1+kx_2]}\\
           &=&\sum_{\substack{k\in\mathbb{Z}\\k\neq0}}f_{k\mu\vec{e}_2}\exp^{2\pi i kx_2}+\sum_{\substack{l,k\in\mathbb{Z}\\k,l\neq0}} f_{l\lambda\vec{e}_1+k\mu\vec{e}_2}\exp^{2\pi i [lx_1+kx_2]}\\
           &\doteq&[f]_{2,1}(x_2)+[f]_{2,2}(x_1,x_2)
\end{eqnarray*}
and raise a new uniform condition:
\vspace{10pt}

{\bf U3:} As a single-variable function, $p_m^{*\sigma}[f]_{2,1}(x_2)$ has a unique maximal point (without loss of generality we assume this point by $x_2=0$) at which $[f]_{2,1}$ is non-degenerate. Besides, $\|p_m^{*\sigma}[f]_{2,1}\|_{C^2} \leq\frac{c_6 d_m^{*\sigma}}{(\mu\iota_m)^{r+2}}$ with $\frac{1}{2}\leq c_6<\frac{c_5}{2}$ for $m\geq M\gg1$.
\begin{Rem}
This condition is aiming to weaken the hyperbolicity according to $(x_2,y_2)$ variables. Recall that when $p_m^*=p_{m+1/2}$, also ${\bf U1}$ should be satisfied and that's why a comparison of $c_5$ and $c_6$ is involved.
\end{Rem}

{\bf U4:} If $\|p_m^{*\sigma}[f]_{2,1}\|_{C^2}\sim\mathcal{O}(\frac{d_m^{*\sigma}}{(l^m)^{r+2+\eta}})$ with $\eta\geq0$, we restrict that $\|p_m^{*\sigma}[f]_{2,2}\|_{C^2}\lessdot\frac{1}{\mathbb{L}(l^m)^{r+2+\eta}}$, $\mathbb{L}\gg0$ will be properly chosen later on. Here $m\geq M\gg\mathbb{L}\gg1$.
\begin{Rem}
This is the so-called `weak-coupled' mechanism. We actually weaken the coupled Fourier coefficients to make the system (\ref{2-resonant canonical}) more like a nearly integrable system at a proper domain of 2-resonance (later in the section of homogenization we will see that). Notice that here a sufficiently large number $\mathbb{L}$ is involved, we will give {\it a priori} estimation of $\mathbb{L}$ then take $m$ sufficiently large comparing to it in the following proof.
\end{Rem}

In the domain $(\mathcal{B}(p_m^*,\delta_{+})\setminus\mathcal{B}(p_m^*,\mathbb{K}\|Z\|_{C^2}^{\frac{1}{2}}))\cap\mathcal{T}(\Gamma_{m,1}^p,\|Z\|_{C^2}^{\frac{1}{2}})$, we can carry out one more step KAM iteration for the system (\ref{2-resonant canonical}):
\begin{eqnarray*}
H'_{+}=H'\circ\Phi&=&H'+\{H',W\}+\int_0^1(1-t)\{\{H',W\},W\}\circ\phi_t(x,y)dt\\
&=&J+\mu\iota_m l^m(h'+p_m^{*\sigma}[f]_1+p_m^{*\sigma}[f]_2)+\mu\iota_m l^m\{h',W\}+R'_{2,+},\\
\end{eqnarray*}
where the new tail term $R'_{2,+}$ satisfying:
\[
R'_{2,+}=\mu\iota_m l^m (R'+\{p_m^{*\sigma}[f]_1+p_m^{*\sigma}[f]_2+R',W\})+\int_0^1\{\{H',W\},W\}\circ\phi_t(x,y)dt.
\]
Here the cohomology equation to solve is 
\[
p_m^{*\sigma}[f]_2+\{h',W\}=p_m^{*\sigma}[[f]_2](x_1)+\Delta\omega W_x,
\]
with $\Delta\omega=\omega-\omega^*$, $\omega^*=\nabla h'(y^*)=(0,\omega_2^*)$ and $y^*\in\Gamma_{m,1}^y\cap(\mathcal{B}(p_m^*,\delta_{+})\setminus\mathcal{B}(p_m^*,\mathbb{K}\|Z\|_{C^2}^{\frac{1}{2}}))$. So we can formally take
\[
W=\frac{1}{T^*}\int_{0}^{T^*}p_m^{*\sigma}[f]_2(x_1,x_2+\omega_2^*s)sds
\]
and
\[
p_m^{*\sigma}[[f]_2](x_1)=\frac{1}{T^*}\int_{0}^{T^*}p_m^{*\sigma}[f]_2(x_1,x_2+\omega_2^*s)ds,
\]
where we denote by $T^*=\frac{1}{\omega_2^*}$.\\

Since this iteration is operated in $\mathcal{B}(y_1^*,\frac{\delta}{l^m})\times\mathcal{B}(y_2^*,\frac{\delta}{\mu\iota_m})$, we need to consider the drift value of $y$ variables. Recall that $\mu\iota_m\geq l^m$, so we can sufficiently take:
\begin{equation}
\|\int_0^1\frac{\partial W}{\partial x}\circ\phi_tdt\|_{\mathcal{B}(y^*,\frac{\delta}{4})\times\mathbb{T}^2}\leq\|W_x\|_{\mathcal{B}(y^*,\delta)}\leq\frac{\delta}{4\mu\iota_m}.
\end{equation}
On the other side, we have:
\begin{eqnarray*}
\|W_x\|_{\mathcal{B}(y^*,\delta)}&\leq& T^*\|p_m^{*\sigma}[f]_2\|_{C^1,\mathcal{B}}\\
&\lessdot&\frac{d_m^{*\sigma}}{(\mu\iota_m)^{r+2+\eta}\omega_2^*}\quad\text{(from {\bf U4})}\\
&\lessdot&\frac{1}{\mathbb{K}}\cdot\frac{d_m^{*\sigma}}{(\mu\iota_m)^{r+2+\eta}}\cdot\frac{1}{\mu\iota_m\|Z\|^{\frac{1}{2}}}\quad\text{(from {\bf U2} and (\ref{index2}))}\\
&\lessdot&\frac{\delta}{\mathbb{K}\mu\iota_m}\ll\frac{\delta}{4\mu\iota_m},\quad\mathbb{K}\;\text{{\it a posteriori} sufficiently large} 
\end{eqnarray*}
with $\delta=\|Z\|^{\frac{1}{2}}$. Besides, the new tail term 
\[
R'_{+}\doteq R'_{1.+}+R'_{2,+}=\Delta\omega W_x+R'_{2,+},
\]
and we have
\begin{eqnarray}\label{crumple}
\|R'_{+}\|_{C^0}&\lessdot&\mu\iota_m l^m (T^*\mu\iota_m\delta d_m^{*\sigma}\frac{1}{(\mu\iota_m)^{r+2+\eta}})\nonumber\\
&\lessdot&\mu\iota_m l^m\frac{\|Z\|}{\mathbb{K}}\cdot\frac{l^{m(r+2)}}{(\mu\iota_m)^{r+2+\eta}},
\end{eqnarray}
during which the largest term is $\Delta\omega W_x$.

Notice that the new resonant term $p_m^{*\sigma}[[f]_2](x_1)$ won't influence the main value of $p_m^{*\sigma}[f]_1(x_1)$ because of {\bf U3} condition. Actually we can first take $\mathbb{L}$ properly large such that ${\mathbb{L}^{-1}}\ll c_i$, $i=1,2,3,\cdots,6$, then take $m\geq M\gg\mathbb{L}$ accordingly. But the difference between $Z$ and $R'_{+}$ is just limited by a multiplier ${\mathbb{K}^{-1}}$ (see inequality (\ref{crumple})). Later we will see that the NHICs in this domain present some kind of `crumpled' form.
\vspace{10pt}

\subsection{Homogenized system of 2-resonance}\label{homogenize}

After passing the transition parts from 1-resonance to 2-resonance, now we reach a domain $(y_1,y_2)\in\mathcal{B}(0,\mathcal{O}(\frac{\mathbb{K}\delta}{l^m}))\times\mathcal{B}(0,\mathcal{O}(\frac{\mathbb{K}\delta}{\mu\iota_m}))$, with $\delta=\|Z\|_{C^0}^{1/2}$. Now we can homogenize system (\ref{2-resonant canonical}) into a classical mechanical system with small perturbation, which benefits us with many fine properties. We will see that in the next section.\\

For convenience, we still rewrite system (\ref{2-resonant canonical}) here:
\[
H'=J+\mu\iota_m l^m(h'(y)+p_m^{*\sigma}[f]_1(x_1)+p_m^{*\sigma}[f]_2(x_1,x_2)+R'(x,y,t)),
\] 
where $(y_1,y_2)\in\mathcal{B}(0,\mathcal{O}(\frac{\mathbb{K}\delta}{l^m}))\times\mathcal{B}(0,\mathcal{O}(\frac{\mathbb{K}\delta}{\mu\iota_m}))$. We can transform this system with a rescale symplectic transformation, via:
\begin{equation}
\begin{pmatrix}
x_1\\x_2\\s\\y_1\\y_2\\J
\end{pmatrix}=
\begin{pmatrix}
{l^m} & 0 & 0 & 0 & 0 & 0\\
0 & {\mu\iota_m} & 0 & 0 & 0 & 0\\
0 & 0 & \frac{1}{\delta l^m\mu\iota_m} & 0 & 0 & 0\\
0 & 0 & 0 & \frac{\delta}{l^m} & 0 & 0\\
0 & 0 & 0 & 0 & \frac{\delta}{\mu\iota_m} & 0\\
0 & 0 & 0 & 0 & 0 & \delta^2 l^m\mu\iota_m
\end{pmatrix}\cdot
\begin{pmatrix}
X_1\\X_2\\S\\Y_1\\Y_2\\e
\end{pmatrix},
\end{equation}
where $\delta\sim\mathcal{O}(\frac{d_m^{*\sigma/2}}{l^{m(r+2)/2}})$. These new variables will satisfy a new O.D.E equation with a different proportion:
\begin{equation}
X'\doteq\frac{\partial X}{\partial S}=\frac{\partial X}{\partial s}\cdot\frac{1}{\delta l^m\mu\iota_m}=\frac{1}{\delta l^m\mu\iota_m}
\begin{pmatrix}
\frac{1}{l^m} & 0\\[5pt]
0 & \frac{1}{\mu\iota_m}
\end{pmatrix}
\begin{pmatrix}
\frac{\partial x_1}{\partial s}\\[5pt]
\frac{\partial x_2}{\partial s}
\end{pmatrix}.
\end{equation}
Since we have
\begin{equation}
\begin{pmatrix}
\frac{\partial x_1}{\partial s}\\[5pt]
\frac{\partial x_2}{\partial s}
\end{pmatrix}=l^m\mu\iota_m
\begin{pmatrix}
\dot{x}_1\\[5pt]
\dot{x}_2
\end{pmatrix}=l^m\mu\iota_m
\left[\begin{pmatrix}
\frac{\partial h'}{\partial y_1}\\[5pt]
\frac{\partial h'}{\partial y_2}
\end{pmatrix}+
\begin{pmatrix}
\frac{\partial R'}{\partial y_1}\\[5pt]
\frac{\partial R'}{\partial y_2}
\end{pmatrix}
\right],
\end{equation} 
we can throw $\mathcal{O}(y^3)$ term of $h'$ into tail and get:
\begin{eqnarray}\label{X}
X'&=&D^2h(p_m^*)Y+\frac{1}{\delta}\Theta_{2\times2}^{-1}
(\mathcal{O}(y^2)+\nabla_{y}R')\nonumber\\
&=&D^2h(p_m^*)Y+\mathcal{O}(\delta Y^2)+\frac{1}{\delta^2}\nabla_{Y}R'.
\end{eqnarray}
On the other hand, we have:
\begin{equation}
Y'\doteq\frac{\partial Y}{\partial S}=\frac{\partial Y}{\partial s}\cdot\frac{1}{\delta l^m\mu\iota_m}=\frac{1}{\delta l^m\mu\iota_m}
\begin{pmatrix}
\frac{l^m}{\delta} & 0\\[5pt]
0 & \frac{\mu\iota_m}{l^m}
\end{pmatrix}
\begin{pmatrix}
\frac{\partial y_1}{\partial s}\\[5pt]
\frac{\partial y_2}{\partial s}
\end{pmatrix},
\end{equation}
\begin{equation}
\begin{pmatrix}
\frac{\partial y_1}{\partial s}\\[5pt]
\frac{\partial y_2}{\partial s}
\end{pmatrix}=l^m\mu\iota_m
\begin{pmatrix}
\dot{y}_1\\[5pt]
\dot{y}_2
\end{pmatrix}=-l^m\mu\iota_m
\left[\begin{pmatrix}
\frac{\partial Z'}{\partial x_1}\\[5pt]
\frac{\partial Z'}{\partial x_2}
\end{pmatrix}+
\begin{pmatrix}
\frac{\partial R'}{\partial x_1}\\[5pt]
\frac{\partial R'}{\partial x_2}
\end{pmatrix}
\right]
\end{equation}
with $Z'(x)\doteq p_m^{*\sigma}[f]_1(x_1)+p_m^{*\sigma}[f]_2(x_1,x_2)$, then we get
\begin{equation}\label{Y}
Y'=-\frac{1}{\delta^2}\nabla_{X}Z'(X_1,X_2)-\frac{1}{\delta^2}\nabla_X R'.
\end{equation}
From (\ref{X}) and (\ref{Y}) we can rescale $H'$ into a new system of variables $(X,S,Y,e)$
\begin{equation}\label{mechanical}
\tilde{H}=\frac{1}{2}\langle Y^t, D^2h(p_m^*)Y\rangle+\tilde{Z}(X)+\tilde{R}(X,Y,S),
\end{equation}
where $\tilde{Z}(X)=\frac{Z'}{\delta^2}$ and $\tilde{R}(X,Y,S)=\frac{R'}{\delta^2}+\mathcal{O}(\delta Y^2)$. Actually, $\tilde{Z}(X)$ can be devided into:
\[
\tilde{Z}(X)=\tilde{Z}_1(X_1)+(\frac{l^m}{\mu\iota_m})^{r+2}\tilde{Z}_2(X_2)+\frac{1}{\mathbb{L}}(\frac{l^m}{\mu\iota_m})^{r+2}\tilde{Z}_3(X_1,X_2),
\]
where $\|\tilde{Z}_i\|_{C^2}\sim\mathcal{O}(1)$, $i=1,2,3$. Recall that $l^m\leq\mu\iota_m\leq l^{m(1+\xi)}$, so the `hardest' case to be considered is of a form
\begin{equation}\label{hardest}
\tilde{Z}(X)=\tilde{Z}_1(X_1)+\tilde{Z}_2(X_2)+\frac{1}{\mathbb{L}}\tilde{Z}_3(X_1,X_2).
\end{equation}
On the other side, the new tail terms $\|\tilde{R}\|_{C^0}\lessdot\mathbb{K}d_m^{*\sigma/2}l^{m(r+6+2\xi)/2}$. It's sufficiently small comparing with $\tilde{Z}$ as $m\gg1$ can be sufficiently large.
\begin{Rem}
This homogenization actually rectifies the `stretched' effect of previous canonical transformation of $(x,y)$ variables and recover the phase space into a normal $\mathcal{O}(1)-$scale. Recall that when $\omega_m^*$ is a sub 2-resonant point, there is no transformation between different resonant lines and we just need to prove the persistence of NHIC with $(x_1,y_1)$ as slow-variables (see Figure \ref{fig7}). So we just need to prove that for the case $\omega_m^*=\omega_{m+1/2}$, which corresponds to the hardest $\tilde{Z}$ case (\ref{hardest}).
\end{Rem}

\section{Existence of NHICs and Location of Aubry sets}\label{NHIC and Aubry}\vspace{10pt}

After the previous conditions and uniform restrictions been satisfied, we can prove the persistence of NHICs corresponding to different Stable Normal Forms in different domains of the phase space. We can divide them into 3 cases of different mechanisms to deal with: 1-resonance, transition from 1-resonance to 2-resonance and 2-resonance. Since our construction is self-similar, we just need to prove that for $\Gamma_{m,1}$ and get the same persistence for that whole $\mathfrak{P}^{\omega}$. It's remarkable that the homogenization does help us greatly in the latter two cases.\\

Actually, we will prove the persistence of `weak-invariant' NHICs in proper domains where KAM iterations work. Here `weak-invariant' means the vector field is just tangent at each point of the cylinder, but unnecessarily vanished at the boundary. We call a NHIC `strong-invariant' if it contains the whole flow of each points. These conceptions were firstly used by Bernard P. in \cite{Ber}. \\

In the following we will firstly prove the former 2-cases with the skill used in \cite{Ber}, then prove the 2-resonance case with the help of the method developed in \cite{Ch} and our special `weak-coupled' construction. \\

\subsection{the persistence of wNHICs for 1-resonance}
$\newline$

From (\ref{parameter system}) we know the canonical system for 1-resonance can be rewritten as:
\[
H={J}+l^{2m}(h'(y,p_i^*)+p_i^{*\sigma}[f]_1(x_1)+R'(x,y,t)),\quad y\in\mathcal{B}(0,\frac{\delta}{l^m}).
\]
The following holds:
\begin{Lem}
There exists a modified system 
\begin{eqnarray*}
\widehat{H}=J&+&l^{2m}[h'_0(y_2)+\frac{1}{2}B(y_2)(y_1-Y_1(y_2))^2+\frac{1}{2}p_i^{*\sigma}[f]''_1(0)x_1^2]\\
&+&l^{2m}[\frac{1}{3!}S(y)(y_1-Y_1(y_2))^3\chi(\frac{l^m(y_1-Y_1(y_2))}{\delta\varpi})+\frac{1}{3!}p_i^{*\sigma}[f]'''_1(x_1)x_1^3\chi(\frac{x_1}{\varrho})\\
&+&y\tilde{R}'\chi(\frac{l^m y}{\delta})],
\end{eqnarray*}
with the assumption that $x_1=0$ is the unique maximal point of $p_i^{*\sigma}[f]_1(x_1)$ and $\chi(\cdot)$ is a compactly supported smooth function which equals to $1$ on the unit ball $\mathcal{B}(0,1)$ and $0$ outside of $\mathcal{B}(0,2)$. Here $(Y_1(y_2),y_2)$ with $\|y_2\|\leq\frac{\delta}{l^m}$ satisfies $ \frac{\partial h'}{\partial y_1}(Y_1,y_2,p_i^*)=0$, $B(y_2)=\partial_{11} h'(Y_1(y_2), y_2)$ and $S(y)$ is a 3-linear form on $\mathbb{R}$ depending smoothly on $y\in\mathbb{R}^2$. 
\end{Lem} 

The value of $\varpi\ll1$ and $\frac{1}{\varrho}\ll1$ will be properly chosen later. Then system $\tilde{H}$ coincides with $H$ on the domain 
\[
\{\|x_1\|\leq\varrho,\|y_1-Y_1(y_2)\|\leq\frac{\delta\varpi}{l^m},x_2\in\mathbb{T}^1,\|y_2\|\leq\frac{\delta}{l^m}\},
\]
and coincides with the integrable system
\begin{eqnarray*}
\overline{H}=J+l^{2m}[h'_0(y_2)+\frac{1}{2}B(y_2)(y_1-Y_1(y_2))^2+\frac{1}{2}p_i^{*\sigma}[f]''_1(0)x_1^2]
\end{eqnarray*}
on the domain
\[
\{\|x_1\|\geq2\varrho,\|y_1-Y_1\|\geq\frac{2\delta\varpi}{l^m},x_2\in\mathbb{T}^1,\|y_2\|\leq\frac{\delta}{l^m}\}.
\]

\begin{proof}
We just need to expand the system $H$ at the point $\{x_1=0,y_1=Y_1(y_2),\|y_2\|\leq\frac{\delta}{l^m}\}$ into a finite Taylor series and then smooth it by multiplying a compactly supported bump function. Recall that $R'=y\widehat{R}'+h.o.t.$ is the new tail term from (\ref{parameter system}).
\end{proof}

First, we can show that $\{(0,x_2,Y_1(y_2),y_2)\big{|}x_2\in\mathbb{T}^1,\|y_2\|\leq\frac{\delta}{l^m}\}\subset T^*\mathbb{T}^2$ is a NHIC according to system $\overline{H}$. In order to simplify the corresponding equations, we set
\[
h_2(y)=h'_0(y_2)+\frac{1}{2}B(y_2)(y_1-Y_1(y_2))^2,
\]
then we have
\begin{eqnarray*}
\begin{pmatrix}
\dot{x}_1\\[8pt]\dot{x}_2\\[8pt]\dot{y}_1\\[8pt]\dot{y_2}
\end{pmatrix}=\frac{1}{l^{2m}}
\begin{pmatrix}
\frac{\partial\overline{H}}{\partial y_1}\\[5pt]
\frac{\partial\overline{H}}{\partial y_2}\\[5pt]
-\frac{\partial\overline{H}}{\partial x_1}\\[5pt]
-\frac{\partial\overline{H}}{\partial x_2}
\end{pmatrix}=
\begin{pmatrix}
0\\[8pt]\partial_{y_2}h_2\\[8pt]0\\[8pt]0
\end{pmatrix}+
\begin{pmatrix}
0 & 0& B(y_2) &0\\[8pt]
0 & 0 & 0 &0\\[8pt]
-p_i^{*\sigma}[f]''_1(0) & 0 & 0 & 0\\[8pt]
0 & 0 & 0 & 0
\end{pmatrix}\cdot
\begin{pmatrix}
x_1\\[8pt]x_2\\[8pt]y_1-Y_1(y_2)\\[8pt]y_2
\end{pmatrix}.
\end{eqnarray*}
We also get the eigenvalues $\sqrt{-B(y_2)p_i^{*\sigma}[f]''_1(0)}$, $-\sqrt{-B(y_2)p_i^{*\sigma}[f]''_1(0)}$, $0$ and $0$ for the linear matrix above, with the corresponding eigenvectors
\[
(\sqrt{B(y_2)},0,\sqrt{-p_i^{*\sigma}[f]''_1(0)},0)^t,
\]
\[
(-\sqrt{B(y_2)},0,\sqrt{-p_i^{*\sigma}[f]''_1(0)},0)^t,
\]
\[
(0,0,0,1)^t,
\]
and
\[
(0,1,0,0)^t.
\]
The existence of NHIC according to $\overline{H}$ can be easily proved.\\

Second, we set
\[
\widehat{R}=\frac{1}{3!}S(y)(y_1-Y_1(y_2))^3\chi(\frac{l^m(y_1-Y_1(y_2))}{\delta\varpi})+\frac{1}{3!}p_i^{*\sigma}[f]'''_1(x_1)x_1^3\chi(\frac{x_1}{\varrho})+y\tilde{R}'\chi(\frac{l^m y}{\delta}),
\]
then the equation corresponding to $\tilde{H}$ can be written by
\begin{eqnarray*}
\begin{pmatrix}
\dot{x}_1\\[8pt]\dot{x}_2\\[8pt]\dot{y}_1\\[8pt]\dot{y_2}
\end{pmatrix}=\frac{1}{l^{2m}}
\begin{pmatrix}
\frac{\partial\tilde{H}}{\partial y_1}\\[5pt]
\frac{\partial\tilde{H}}{\partial y_2}\\[5pt]
-\frac{\partial\tilde{H}}{\partial x_1}\\[5pt]
-\frac{\partial\tilde{H}}{\partial x_2}
\end{pmatrix}=\frac{1}{l^{2m}}
\begin{pmatrix}
\frac{\partial\overline{H}}{\partial y_1}\\[5pt]
\frac{\partial\overline{H}}{\partial y_2}\\[5pt]
-\frac{\partial\overline{H}}{\partial x_1}\\[5pt]
-\frac{\partial\overline{H}}{\partial x_2}
\end{pmatrix}+\frac{1}{l^{2m}}
\begin{pmatrix}
\frac{\partial\widehat{R}}{\partial y_1}\\[5pt]
\frac{\partial\widehat{R}}{\partial y_2}\\[5pt]
-\frac{\partial\widehat{R}}{\partial x_1}\\[5pt]
-\frac{\partial\widehat{R}}{\partial x_2}
\end{pmatrix}.
\end{eqnarray*}
To get the persistence of NHIC corresponding to $\tilde{H}$, we need to control the value of $\widehat{R}$ under the norm $\|\cdot\|_{C^2,\mathcal{B}}$ and make it strictly separated from the spectrum radius $\sqrt{-B(y_2)p_i^{*\sigma}[f]''_1(0)}$, based on the classical theory of NHIC in \cite{HPS}. So we get
\begin{eqnarray*}
\|S(y)\|\frac{\delta\varpi}{l^m}\ll\sqrt{-B(y_2)p_i^{*\sigma}[f]''_1(0)}\Rightarrow l^{2m}\delta\varpi\ll \sqrt{\frac{d_i^{*\sigma}}{l^{mr}}},\\
\frac{d_i^{*\sigma}}{l^{m(r+2)}}\varrho\ll\sqrt{-B(y_2)p_i^{*\sigma}[f]''_1(0)}\Rightarrow\varrho\sim\mathcal{O}(1) \text{\;small},\\
\frac{1}{\delta}\frac{d_i^{*\sigma-1}}{l^{mr}}\ll\sqrt{-B(y_2)p_i^{*\sigma}[f]''_1(0)}\Rightarrow\frac{d_i^{*\sigma/2-1}}{l^{mr/2}}\ll\delta,\\
\frac{d_i^{*\sigma}}{l^{m(1+\xi)(r-6)}}\frac{l^{2m}}{\delta^2}\ll\sqrt{-B(y_2)p_i^{*\sigma}[f]''_1(0)}\Rightarrow\frac{d_i^{*\frac{\sigma}{4}}}{l^{\frac{m(r+2\xi r-16-12\xi)}{4}}}\ll\delta,
\end{eqnarray*}
from (\ref{1-resonant 1}), (\ref{1-resonant 2}), (\ref{1-resonant 3}) and (\ref{1-resonant 4}). Then we can get the strict lower bound of $\delta$ by $\frac{d_i^{*\frac{\sigma}{4}}}{l^{\frac{m(r+2\xi r-16-12\xi)}{4}}}$. Once the aforementioned inequalities satisfied, we actually verify the persistence of NHIC for $\tilde{H}$ in the domain:
\[
\{\|x_1\|\leq2\varrho,\|y_1-Y_1\|\leq\frac{2\delta\varpi}{l^m},x_2\in\mathbb{T}^1,\|y_2\|\leq\frac{\delta}{l^m}\}.
\]
On the other side, $\tilde{H}$ coincides with $H$ in the domain:
\[
\{\|x_1\|\leq\varrho,\|y_1-Y_1(y_2)\|\leq\frac{\delta\varpi}{l^m},x_2\in\mathbb{T}^1,\|y_2\|\leq\frac{\delta}{l^m}\},
\]
So the cylinder is also `weak-invariant' for $H$ in the sense that at the two ends there may be overflow (or interflow). Besides, we can see that the wNHIC (weak Normally Hyperbolic Invariant Cylinder) is of the form 
\[
\{(t,x_1(x_2,y_2,t),x_2,y_1(x_2,y_2,t),y_2)\big{|}t\in\mathbb{S}^1\},
\]
where we have
\[
\|y_1(x_2,y_2,t)-Y_1(y_2)\|\lessdot\frac{p_i^{*\sigma+r+1}}{B(y_2)}
\]
and
\[
\|x_1(x_2,y_2,t)\|\lessdot\varrho^2.
\]
By changing $p_i^*$ along the $\Gamma_{m,1}^{p}$, the `short' wNHICs under different canonical coordinations can be joint into a `long' wNHIC with only two ends overflow (or interflow).\\

Notice that we take the lower bound of $\delta$ into (\ref{1-resonant 4}) and get a restriction for $\delta_{+}$ as well. As $d_i\sim\mathcal{O}(1)$, we can get an index estimation replacing (\ref{index2}):

\begin{equation}\label{index3}
d_i^{*\frac{\sigma}{4}}l^{\frac{m(24+16\xi-r-2\xi r)}{4}}\sim\mathcal{O}(\|Z\|^{\frac{\sigma+(1+2\xi)r-24-16\xi}{4(\sigma+r+2)}}).
\end{equation}

Here the index tends to 1/4 as $\sigma\rightarrow\infty$. This implies that the wNHIC of a 1-resonant mechanism can be expanded into the place at least $\mathcal{O}(\|Z\|^{1/6})$ approaching the 2-resonant points, as long as $\sigma$ is chosen properly large (see figure \ref{fig8}). 
%Recall that the cylinder is invariant just for $\tilde{H}$ system, but in the domain $\{\|x_1\|\leq\varrho,\|y_1-Y_1(y_2)\|\leq\frac{\delta\varpi}{l^m},x_2\in\mathbb{T}^1,\|y_2\|\leq\frac{\delta}{l^m}\}$, $\tilde{H}$ coincides with $H$. So the cylinder is also `weak-invariant' for $H$ in the sense that at the two ends there may be overflow (or interflow). By changing $p_i^*$ along the $\Gamma_{m,1}^{p}$, the `short' wNHICs in each canonical coordinations can be joint into a `long' wNHIC with only two ends overflow (or interflow).
\vspace{10pt}

\subsection{the persistence of wNHICs for the transition part from 1-resonance to 2-resonance }
$\newline$

In this subsection, we will further expand the wNHIC got in the previous subsection into the places $\mathcal{O}(\mathbb{K}\|Z\|^{\frac{1}{2}})$ approaching 2-resonant points. Different from the wNHIC of 1-resonant mechanism, here we will use the methods developed in subsection \ref{1 to 2} and \ref{homogenize}, i.e. both the twice KAM iteration and the homogenization are involved.\\

\begin{figure}
\begin{center}
\includegraphics[width=7cm]{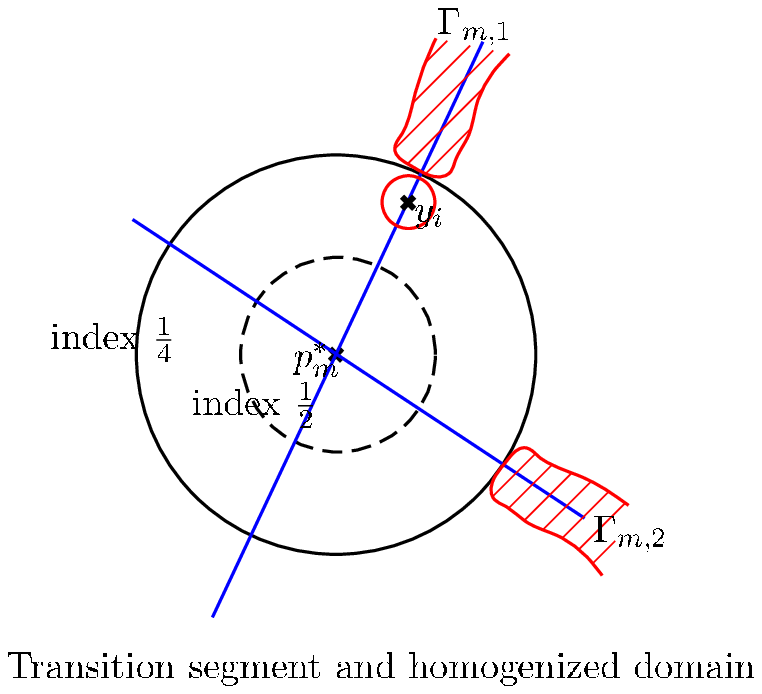}
\caption{ }
\label{fig9}
\end{center}
\end{figure}
As Figure \ref{fig9} shows us, we can pick up finitely many $y_i$ points on the $\Gamma_{m,1}^{y}$ from       index `1/6' to `1/2' domains of 2-resonant points $p_m^*$. Recall that system (\ref{2-resonant canonical}) is valid, and we can carry out twice KAM iteration in the domain $\mathcal{B}(y_{i,1},\frac{\delta}{l^m})\times\mathcal{B}(y_{i,2},\frac{\delta}{\mu\iota_m})$, where $y_i=(y_{i,1},y_{i,2})$ and $\delta=\|Z\|^{1/2}$. We rewrite the system as:
\begin{equation}\label{transition twice iteration}
H'=J+\mu\iota_m l^m(h'(y,y_i)+p_m^{*\sigma}[f]'_1(x_1)+R'(x,y,y_i,t)), 
\end{equation}
where {\bf U1$\rightarrow$4} restrictions are available. Here the parameter $y_i$ reminds us in which domain we carry out the twice KAM iteration. Recall that the estimation (\ref{crumple}) is still valid for this system \ref{transition twice iteration} and we have $\|R'\|\lessdot\frac{\|Z\|}{\mathbb{K}}$ for all chosen $y_i$. This comparison of the nearly same order failed to ensure the persistence of a common wNHIC and we have to seek for other mechanisms to ensure the persistence of that. Homogenization will be involved to find the so-called `crumpled' wNHIC (this definition was firstly proposed in \cite{BKZ}).\\

With the same idea as subsection \ref{homogenize}, we can `rescale' system \ref{transition twice iteration} via:
\[
x_1=l^mX_1,\;x_2=\mu\iota_mX_2,\; y_1-y_{i,1}=\frac{\delta}{l^m}Y_1,\;y_2-y_{i,2}=\frac{\delta}{\mu\iota_m}Y_2,\;s=\frac{S}{\delta\mu\iota_m l^m}.
\]
Then system (\ref{transition twice iteration}) will become:
\[
\tilde{H}= \frac{\omega_{i,2}}{\mu\iota_m\delta} Y_2+1/2\langle Y^t, D^2h(p_i)Y\rangle+\tilde{Z}(X_1)+\frac{\tilde{R}(X,Y,S)}{\mathbb{K}},
\]
where $p_i$ is the corresponding point of $y_i$ and $\nabla h'(y_i)=(0,\omega_{i,2})$. Recall that the tail term takes the largest value $\mathcal{O}(\frac{1}{\mathbb{K}})$ when $y_i\in\mathcal{B}(0,\mathcal{O}(\frac{\mathbb{K}\delta}{l^m}))\times\mathcal{B}(0,\mathcal{O}(\frac{\mathbb{K}\delta}{\mu\iota_m}))$, so we write it down in a form $\frac{\tilde{R}(X,Y,S)}{\mathbb{K}}$. After this rescale we actually have $Y\in\mathcal{B}(0,c_7)\subset\mathbb{R}^2$, $\|\tilde{Z}\|_{C^2}\sim\mathcal{O}(1)$ and $\|\tilde{R}\|_{c^2}\sim\mathcal{O}(1)$. Here $c_7\sim\mathcal{O}(1)$ is a constant.\\

Under the new coordination, we have the following modified system:
\begin{eqnarray*}
\widehat{H}=&\frac{\omega_{i,2}}{\mu\iota_m\delta}Y_2+\frac{1}{2}(Y_1(Y_2),Y_2)^t\cdot D^2h(p_i)\cdot\begin{pmatrix}
Y_1(Y_2)\\Y_2
\end{pmatrix}+\frac{1}{2}a_i(Y_1-Y_1(Y_2))^2+\frac{1}{2}\tilde{Z}''(0)X_1^2\\
&\quad+\frac{1}{3!}\chi(\frac{X_1}{\varrho})\tilde{Z}'''(X_1)X_1^3+\frac{1}{\mathbb{K}}\chi(\frac{Y_1-Y_1(Y_2)}{\varpi},\frac{2Y_2}{c_7})\tilde{R}(X,Y,S),
\end{eqnarray*}
where $\chi(\cdot)$ is a smooth bump function as it's given in the previous subsection and we formally denote by
\[
D^2h(p_i)=\begin{pmatrix}
a_i & b_i\\
b_i & c_i
\end{pmatrix}.
\]
We can see that $\widehat{H}$ coincides with $\tilde{H}$ in the domain:
\[
\{\|X_1\|\leq\varrho,\|Y_1-Y_1(Y_2)\|\leq\varpi,\|Y_2\|\leq\frac{c_7}{2}\}.
\]
For convenience we can assume:
\[
\widehat{R}=\frac{1}{3!}\chi(\frac{X_1}{\varrho})\tilde{Z}'''(X_1)X_1^3+\frac{1}{\mathbb{K}}\chi(\frac{Y_1-Y_1(Y_2)}{\varpi},\frac{2Y_2}{c_7})\tilde{R}(X,Y,S),
\]
and
\[
h_2(Y)=\frac{\omega_{i,2}}{\mu\iota_m\delta}Y_2+\frac{1}{2}(Y_1(Y_2),Y_2)^t\cdot D^2h(p_i)\cdot\begin{pmatrix}
Y_1(Y_2)\\Y_2
\end{pmatrix}+\frac{1}{2}a_i(Y_1-Y_1(Y_2))^2,
\]
where $Y_1(Y_2)$ satisfies: $\partial_{Y_1}h_2(Y_1(Y_2),Y_2)=0$.
Then we write down the corresponding equations as:
\[
X_1'=\frac{\partial\widehat{H}}{\partial Y_1}=a_i(Y_1-Y_1(Y_2))+\frac{\partial\widehat{R}}{\partial Y_1},
\]
\[
Y_1'=-\frac{\partial\widehat{H}}{\partial X_1}=-\tilde{Z}''(0)X_1-\frac{\partial\widehat{R}}{\partial X_1},
\]
\[
X_2'=\frac{\partial\widehat{H}}{\partial Y_2}=\partial_{Y_2}h_2+\frac{\partial\widehat{R}}{\partial Y_2},
\]
\[
Y_2'=-\frac{\partial\widehat{H}}{\partial X_2}=-\frac{\partial\widehat{R}}{\partial X_2}.
\]
Also we can show that $\{(0,X_2,Y_1(Y_2),Y_2)\big{|}X_2\in\mathbb{T}^1,\|Y_2\|\leq c_7\}$ is a NHIC corresponding to system 
\[
\overline{H}=h_2(Y)+\frac{1}{2}\tilde{Z}''(0)X_1^2.
\]
Recall that the eigenvalues of the linear system $\overline{H}$ is $\pm\sqrt{-a_i\tilde{Z}''(0)}$ and $0$ (2 order). To prove the persistence of NHIC for system $\widehat{H}$, we need to control the value of $\widehat{R}$ under the norm $\|\cdot\|_{C^2,\mathcal{B}}$. So the undetermined variables $\varrho$, $\varpi$ should satisfy the following:
\begin{eqnarray*}
\frac{1}{3!}\varrho\ll\sqrt{-a_i\tilde{Z}''(0)},\\
\frac{1}{\mathbb{K}\varpi^2}\ll\sqrt{-a_i\tilde{Z}''(0)},
\end{eqnarray*}
 where we take $\varpi\leq\frac{c_7}{8}$ for convenience. Since $\sqrt{-a_i\tilde{Z}''(0)}\sim\mathcal{O}(1)$ and $\mathbb{K}$ can be chosen properly large, we roughly take $\varrho=\frac{1}{\mathbb{K}^{1/2}}$ and $\varpi=\frac{1}{\mathbb{K}^{1/4}}$ and verify the persistence of NHIC for $\widehat{H}$ in the domain:
 \[
\{\|X_1\|\leq2\varrho,\|Y_1-Y_1(Y_2)\|\leq2\varpi,\|Y_2\|\leq{c_7}\}.
 \]
On the other side, $\widehat{H}$ coincides with $\tilde{H}$ in the domain:
\[
\{\|X_1\|\leq\varrho,\|Y_1-Y_1(Y_2)\|\leq\varpi,\|Y_2\|\leq\frac{c_7}{2}\},
\]
so we actually proved the persistence of wNHIC for $\tilde{H}$ in this domain which is of the form
\begin{equation}
\{(S,X_1(X_2,Y_2,S),Y_1(X_2,Y_2,S),X_2,Y_2)\}.
\end{equation} 
We can see that 
\[
\|Y_1(X_2,Y_2,S)-Y_1(Y_2)\|\lessdot\frac{1}{\mathbb{K}^{3/4}}\rightarrow0,
\] 
and 
\[
\|X_1(X_2,Y_2,S)\|\lessdot\frac{1}{\mathbb{K}}\rightarrow0.
\]
uniformly as $\mathbb{K}\rightarrow\infty$. But we can see that in the original coordination
\[
\|\frac{\partial x_1(x_2,x_2,s)}{\partial y_2}\|\sim\mathcal{O}(\frac{\mu\iota_m l^m}{\mathbb{K}\delta}),
\]
the right side of which is quite large. That's the meaning of `crumpled' and we can see that the nearer $y_i$ approaches $0$, the more violent the crumple is. $\mathcal{O}(\frac{\mu\iota_m l^m}{\mathbb{K}\delta})$ is actually the largest estimation of crumple for $y_i\in\mathcal{B}(0,\mathcal{O}(\frac{\mathbb{K}\delta}{l^m}))\times\mathcal{B}(0,\mathcal{O}(\frac{\mathbb{K}\delta}{\mu\iota_m}))$ (see Figure \ref{fig10}).
\begin{figure}
\begin{center}
\includegraphics[width=8cm]{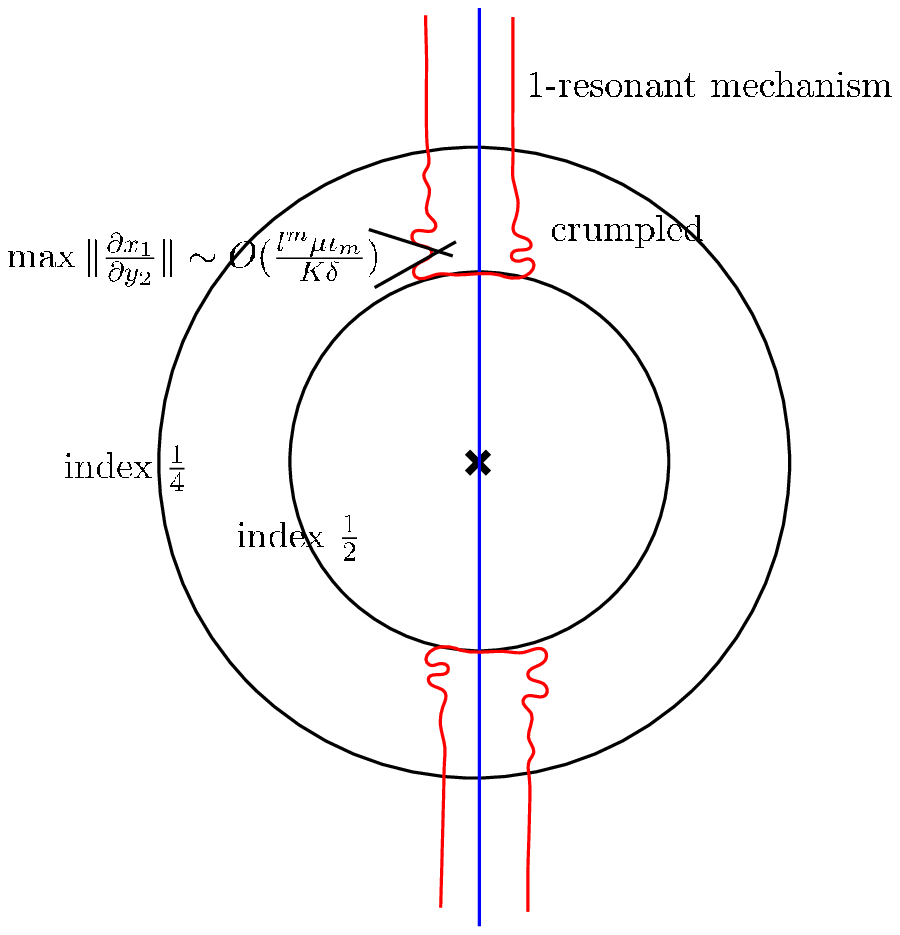}
\caption{ }
\label{fig10}
\end{center}
\end{figure}
\vspace{10pt}

\subsection{the persistence of wNHICs for the 2-resonance ($\omega_{m,\frac{1}{2}}$)} 
$\newline$

After the homogenization of subsection \ref{homogenize}, now we consider the following mechanical system with the `hardest' case potential function of (\ref{hardest}):
\begin{equation}
H(X,Y,S)=\overline{H}(X,Y)+\epsilon R(X,Y,S),
\end{equation}
\begin{equation}\label{model}
\overline{H}(X,Y)=\frac{1}{2}\langle Y^t, D^2h(p_m^*)Y\rangle+Z_1(X_1)+Z_2(X_2)+\varepsilon Z_3(X_1,X_2),
\end{equation}
where $\varepsilon=\frac{1}{\mathbb{L}}$ and $\epsilon=\mathbb{K}d_m^{*\sigma/2}l^{m(r+6+2\xi)/2}$. \\

Recall that the variable $(X_1,Y_1)$ corresponding to the resonant line $\Gamma_{m,1}$ and $(X_2, Y_2)$ corresponding to $\Gamma_{m,2}$ when $p_m^*=p_{m,1/2}$ (We only need to consider this case since there's no transition between resonant lines at other sub 2-resonant points). To complete our `weak-coupled' structure, another condition is needed:\\

{\bf C3:} $D^2h(0)=Id_{2\times2}$.

\begin{Rem}
Once {\bf C3} is satisfied, we have $D^2h(p_m^*)=Id+\mathcal{O}(\frac{1}{l^m})$ for sufficiently large $m\gg1$. Then the `model system' (\ref{model}) is actually of a form:
\begin{equation}\label{model system}
\overline{H}(X,Y)=\frac{1}{2}\langle Y^t, Y\rangle+Z_1(X_1)+Z_2(X_2)+\varepsilon{Z}_3(X_1,X_2)+\mathcal{O}(\frac{\|Y\|^2}{l^m}),
\end{equation}
where $\frac{1}{l^m}\ll\varepsilon$ as long as $m\gg1$. Later we will see that the perturbation term $\mathcal{O}(\frac{\|Y\|^2}{l^m})$ will not damage the qualitative properties of the following system 
\begin{equation}\label{new model}
\overline{H}=\frac{1}{2}\langle Y^t, Y\rangle+Z_1(X_1)+Z_2(X_2)+\varepsilon{Z}_3(X_1,X_2),
\end{equation}
So we take this (\ref{new model}) as our `model system'  and research it instead.
\end{Rem}
\begin{Rem}
From the following analysis we can see that $D^2h(0)$ being diagonal is enough. But we still take {\bf C3} condition for convenience of our symbolism.
\end{Rem}
Since we can take $\varepsilon$ {\it a priori} small, so this system 
\begin{equation}\label{integrable}
\overline{\overline{H}}=\frac{1}{2}\langle Y^t, Y\rangle+Z_1(X_1)+Z_2(X_2)
\end{equation}
will be involved and $\overline{H}$ can be seen as its $\mathcal{O}(\varepsilon)$ perturbation, $\varepsilon\ll1$. This `weak-coupled' structure is enlighten by the classical Melnikov method \cite{Tr2}, which was firstly discovered by V. Melnikov in 1960s. Systems of this kind have several fine properties from the viewpoint of Mather Theory. We will elaborate these in the following.

\begin{Pro}\label{autonomous}
\begin{itemize}
\item For the model system (\ref{new model}), we can find two NHICs $\overline{N}_1^{+}$, $\overline{N}_2^{+}$ corresponding to homology class $g_1=(1,0)$, $g_2=(0,1)$. The bottom of $\overline{N}_1^{+}$ (or $\overline{N}_2^{+}$) is a unique $g_1-$ (or $g_2-$) homoclinic orbit. Since $\overline{H}$ is a mechanical system, we can also find NHICs $\overline{N}_1^{-}$ and $\overline{N}_2^{-}$ corresponding to $-g_1$ and $-g_2$, with the bottom $-g_1$ and $-g_2$ homoclinic orbits.
\item Based on our {\bf U3} restrictions, we can expand $\overline{N}_2^{\pm}$ to the minus energy surfaces, i.e. there exists a NHIC $\overline{N}_2^{-}\cup\overline{N}_2^{0,-e_0}\cup\overline{N}_2^{+}$, with $1<e_0<4$.
\end{itemize}
\end{Pro}

\begin{proof}
First, we can see that $\overline{\overline{H}}_1=\frac{1}{2}Y_1^2+Z_1(X_1)$ and $\overline{\overline{H}}_2=\frac{1}{2}Y_1^2+Z_1(X_1)$ are two uncoupled first-integrals of $\overline{\overline{H}}$. Then at the energy surface $\mathcal{S}_{E}\doteq\{(X,Y)\in T^*\mathbb{T}^2\big{|}\overline{\overline{H}}=E\geq0\}$, we can find two 1-dimensional normally hyperbolic tori each located in the foliation $\mathcal{S}_{E}^{1}\doteq\{\overline{\overline{H}}_1=E,\overline{\overline{H}}_2=0\}$ and $\mathcal{S}_{E}^{2}\doteq\{\overline{\overline{H}}_2=E,\overline{\overline{H}}_1=0\}$, which can be denoted by $\mathcal{T}_{g_1,E}$ and $\mathcal{T}_{g_2,E}$. Notice that we can define $\mathcal{T}_{-g_1,E}$ and $\mathcal{T}_{-g_2,E}$ in the same way, but we just need to deal with the positive homology case owing to the symmetry property of mechanical systems. Without loss of generality, we can assume $Z_1+Z_2$ take its maximal value $0$ at the point $(0,0)\in\mathbb{T}^2$. Also we can transform {\bf U3} into a following {\bf U3'} for this homogenized case, which is more convenient to use:\\

{\bf U3':} $-Z_1''(0)=\lambda_1>0$ and $-Z_2''(0)=\lambda_2>0$. Besides, $\lambda_1-\lambda_2\geq c_8>0$ and $\frac{\lambda_1}{\lambda_2}\geq1+c_9>1$. Here $c_8$ and $c_9$ are constants depending only on $c_5$ and $c_6$, and they are uniformly taken for $\forall m\geq M\gg1$.\\

Since the vector fields $J\nabla\overline{\overline{H}}_1$ and $J\nabla\overline{\overline{H}}_2$ are independent of each other at the place $\mathcal{S}_{E}\setminus(\mathcal{T}_{\pm g_1,E}\cup\mathcal{T}_{\pm g_2,E})$, we can define the stable manifold $W_{g_i,E}^s$ and unstable manifold $W_{g_i,E}^u$ from the trend of trajectories on it, $i=1,2$. Actually they are invariant Lagrangian graphs and $\mathcal{T}_{g_i,E}\subset W_{g_i,E}^u\cap W_{g_i,E}^s$. So we can express $W_{g_1,E}^{s,u}$ as $\{(X, dS_{g_1,E}^{s,u}(X))+(0,h(\mathcal{T}_{g_1,E}))\big{|}X\in\mathbb{T}^1\times[-\pi-\delta,\pi+\delta]\subset\mathbb{R}^2\}$. Here $h(\mathcal{T}_{g_1,E})\in\mathbb{R}^2$ is the average velocity of $\mathcal{T}_{g_1,E}$. In the same way we express $W_{g_2,E}^{s,u}$ as $\{(X, dS_{g_2,E}^{s,u}(X))+(h(\mathcal{T}_{g_2,E}),0)\big{|}X\in[-\pi-\delta,\pi+\delta]\times\mathbb{T}^1\subset\mathbb{R}^2\}$. Notice that $0<\delta<1$ can be chosen properly small, which is to ensure that the definition domains of these graphs can cover a whole copy of $\mathbb{T}^2$ in the universal covering space $\mathbb{R}^2$. Actually we can see that $S_{g_2,E}^{s,u}(X)$ only depends on $X_1$ and $S_{g_1,E}^{s,u}(X)$ only depends on $X_2$ in their corresponding domains. So we have
\begin{enumerate}
\item $\frac{\partial(S_{g_2,E}^{u}-S_{g_2,E}^{s})(X_1)}{\partial X_1}\big{|}_{\mathcal{T}_{g_2,E}}=0$, $\frac{\partial^2(S_{g_2,E}^{u}-S_{g_2,E}^{s})(X_1)}{\partial X_1^2}\big{|}_{\mathcal{T}_{g_2,E}}=2\sqrt{2\lambda_1}$.\\
\item $\frac{\partial(S_{g_1,E}^{u}-S_{g_1,E}^{s})(X_2)}{\partial X_2}\big{|}_{\mathcal{T}_{g_2,E}}=0$, $\frac{\partial^2(S_{g_1,E}^{u}-S_{g_1,E}^{s})(X_2)}{\partial X_2^2}\big{|}_{\mathcal{T}_{g_2,E}}=2\sqrt{2\lambda_2}$.\\
\end{enumerate}
Then the projection set $\{0\}\times\mathbb{T}^1$ of $\mathcal{T}_{g_2,E}$ is the set of minimizers of $S_{g_2,E}^{u}-S_{g_2,E}^{s}$. In the same way $\mathbb{T}^1\times\{0\}$ is that of $(S_{g_1,E}^{u}-S_{g_1,E}^{s})$'s. Notice that the former analysis is valid for all $E\geq0$, and $\mathcal{T}_{g_i,E}$ are unique in their small neighborhoods, $i=1,2$.\\

For $\overline{H}$ now, since $\varepsilon\ll1$ can be chosen {\it a priori} small, we can still find the perturbed manifolds $W_{g_i,E,\varepsilon}^{s,u}$ and their generating functions $S_{g_i,E,\varepsilon}^{s,u}$ in the corresponding domains, $i=1,2$. Besides, we have 
\[
S_{g_i,E,\varepsilon}^{s,u}=S_{g_i,E,0}^{s,u}+\varepsilon S_{g_i,E,1}^{s,u}+\mathcal{O}(\varepsilon^2),\;i=1,2,
\]
and
\[
\overline{H}(X, \nabla S_{g_i,E,\varepsilon}^{s,u}+h(\mathcal{T}_{g_i,E}^{\varepsilon}))=E,\;i=1,2\;\&\;E>0.
\]
Here $h(\mathcal{T}_{g_i,E}^{\varepsilon})$ is the average velocity of $\mathcal{T}_{g_i,E}^{\varepsilon}$. The existence of $\mathcal{T}_{g_i,E}^{\varepsilon}\subset W_{g_i,E,\varepsilon}^{s}\cap W_{g_i,E,\varepsilon}^{u}$ can be proved by the theorem of implicit function from the two formulas above. Actually, we can take a section $[-\sqrt{\varepsilon},\sqrt{\varepsilon}]\times\{X_2^*\}$ and restrict $S_{g_2,E,\varepsilon}^{u}-S_{g_2,E,\varepsilon}^{s}$ on it. We will find the unique minimizer $X_1(X_2^*)$. So we have proved the existence of $\mathcal{T}_{g_2,E}^{\varepsilon}$ with $\{(X_1(X_2^*),X_2^*)$\big{|}$X_2^*\in\mathbb{T}^1\}$ its projection. Besides, $\mathcal{T}_{g_2,E}^{\varepsilon}$ is hyperbolic and as $E$ changes these tori make up a NHIC corresponding to homology class $g_2$. In the same way we get similar results for $\mathcal{T}_{g_1,E}^{\varepsilon}$. Now we raise another restriction for $Z_3(X)$. This restriction is just for convenience and is not necessary.\\

{\bf C4:} $Z_1+Z_2+\varepsilon Z_3$ reaches its maximum at $(0,0)\in\mathbb{T}^2$. Besides, its two eigenvalues are still $\lambda_1$ and $\lambda_2$, which have the same corresponding eigenvectors with $Z_1+Z_2$ at $(0,0)$.\\

From this restriction we can see that the Ma\~{n}\'e Critical Value of $\overline{H}=0$ is the same with $\overline{\overline{H}}$. We denote the energy surface of $\overline{H}$ by $\mathcal{S}_{E}^{\varepsilon}\doteq\{(X,Y)\in T^*\mathbb{T}^2\big{|}\overline{H}=E\}$. Then in the similar way as above we can prove the existence of $g_1-$ and $g_2-$ type homoclinic orbits at $\mathcal{S}_{0}^{\varepsilon}$, then prove the uniqueness of them.\\

For $\overline{\overline{H}}$ system, we can suspend the generating function $S_{g_i,0}$ into the universal covering space $\mathbb{R}^2$. So we have a couple of $\tilde{S}_{0,\vec{n}}^{u,s}$ defined in the domain $[-\pi-\delta,\pi+\delta]\times[-\pi-\delta,\pi+\delta]+\vec{n}$, where $i=1,2$ and $\vec{n}=(n_1, n_2)\in\mathbb{Z}^2$. Here $+\vec{(\cdot)}$ is a parallel move in $\mathbb{R}^2$. Based on {\bf C4}, then $\tilde{S}_{0,\vec{n}}^{u}-\tilde{S}_{0,\vec{n}}^{s}$ takes $\vec{n}$ as its unique minimizer in the domain of definition. Taking $(X,d\tilde{S}_{0,\vec{n}}^{u})$ as the initial condition, where $X\in([-\pi-\delta,\pi+\delta]\times[-\pi-\delta,\pi+\delta]+\vec{n})$, the trajectory of $\overline{\overline{H}}$ will exponentially tend to $(\vec{n},0)\in T^*\mathbb{R}^2$ as $t\rightarrow-\infty$. Similarly, the trajectory with initial condition $(X,d\tilde{S}_{0,\vec{n}}^{s})$ tends to $(\vec{n},0)$ exponentially as $t\rightarrow+\infty$.\\

We take $\{X_1=\pi\}$ as the common section and restrict $S_{0,(0,0)}^u$ and $S_{0,(2\pi,0)}^s$ to it, then we have
\[
\frac{\partial(S_{0,(0,0)}^{u}-S_{0,(2\pi,0)}^{s})(\pi,0)}{\partial X_2}=0, \;\frac{\partial^2(S_{0,(0,0)}^{u}-S_{0,(2\pi,0)}^{s})(\pi,0)}{\partial X_2^2}=2\sqrt{2\lambda_2}.
\]
Similarly, we take $\{X_2=\pi\}$ as the common section and restrict $S_{0,(0,0)}^u$ and $S_{0,(0,2\pi)}^s$  to it and get
\[
\frac{\partial(S_{0,(0,0)}^{u}-S_{0,(0,2\pi)}^{s})(0,\pi)}{\partial X_2}=0, \;\frac{\partial^2(S_{0,(0,0)}^{u}-S_{0,(0,2\pi)}^{s})(0,\pi)}{\partial X_2^2}=2\sqrt{2\lambda_1}.
\]
Since $\overline{H}$ is just a $\mathcal{O}(\epsilon)$ perturbation of  $\overline{\overline{H}}$, we also have:
\[
S_{0,\vec{n},\varepsilon}^{s,u}=S_{0,\vec{n}}^{s,u}+\varepsilon S_{0,\vec{n},1}^{s,u}+\mathcal{O}(\varepsilon^2),\;i=1,2.
\]
So $(S_{0,(0,0),\varepsilon}^{u}-S_{0,(2\pi,0),\varepsilon}^{s})(\pi,X_2)$ has a unique minimizer in $[-\pi-\delta,\pi+\delta]$ as a single variable function of $X_2$. Also $(S_{0,(0,0),\varepsilon}^{u}-S_{0,(0,2\pi),\varepsilon}^{s})(X_1,\pi)$ takes its unique maximal value in $[-\pi-\delta,\pi+\delta]$ as a single variable function of $X_1$. Then we get the existence and uniqueness of $g_1-$ and $g_2-$ type homoclinic orbits. The first bullet of this proposition has been proved.

\begin{Rem}
Similar results have been proved by A. Delshams and etc in \cite{Del}, where they call the homoclinic orbits we find `isolated' type and also get the uniqueness with the same Melnikov method. 
\end{Rem}

Second, for the uncoupled system $\overline{\overline{H}}$, we can find a closed trajectory of zero-homology in $\mathcal{S}_{-e}^{2}$ which is denoted by $O_{-e}$ with a period $T_{-e}$, $e>0$. We can take a section $\Sigma_{\{X_2=\pi,Y_2>0\}}^{+}$ and restrict it in a small neighborhood of $O_{-e}$. Here `$+$' means the restriction of $Y_2>0$. Then we have a Poincar\'e mapping 
\[
\phi^{T(X,Y)}: \Sigma_{\{X_2=\pi\}}^{+}\cap\mathcal{B}(O_{-e},\delta)\cap\mathcal{S}_{-e}\rightarrow\Sigma_{\{X_2=\pi\}}^{+}\cap\mathcal{B}(O_{-e},\delta)\cap\mathcal{S}_{-e}.
\]
Obviously $\phi^{T(X,Y)}$ has a unique hyperbolic fixed point $(0,\pi,0,\sqrt{-2e-2Z_2(\pi)})$, with the eigenvalues $\pm2\sqrt{-2\lambda_1}$. So we have a expanded NHIC $\overline{\overline{N}}_2^{-}\cup\overline{\overline{N}}_2^{0,-2e_{0}}\cup\overline{\overline{N}}_2^{+}$, where $e_{0}\sim\mathcal{O}(1)$ is a proper positive constant.\\

For the system $\overline{H}$, we can prove the persistence of $\overline{N}_2^{-}\cup\overline{N}_2^{0,-e_{0}}\cup\overline{N}_2^{+}$ as a wNHIC via the following theorem \ref{wiggins}. This wNHIC can be seen as a deformation of 
$\overline{\overline{N}}_2^{-}\cup\overline{\overline{N}}_2^{0,-2e_{0}}\cup\overline{\overline{N}}_2^{+}$. \\
\end{proof}
\begin{Rem}
Notice that the approach of theorem \ref{wiggins} is also valid for system $H$.
\end{Rem}
\vspace{10pt}

\begin{The}\label{wiggins}
There exists $ \varepsilon_0>0$ sufficiently small such that $\forall 0<\varepsilon\leq\varepsilon_0$, $\overline{N}_2^{-}\cup\overline{N}_2^{0,-e_{0}}\cup\overline{N}_2^{+}$ persists as a $C^1$ wNHIC of system $\overline{H}$ and satisfies the following:
\begin{itemize}
\item $\overline{N}_2^{-}\cup\overline{N}_2^{0,-e_{0}}\cup\overline{N}_2^{+}$ is $C^1$ in $\varepsilon$.\\
\item $\overline{N}_2^{-}\cup\overline{N}_2^{0,-e_{0}}\cup\overline{N}_2^{+}$ $\varepsilon-$close to $\overline{\overline{N}}_2^{-}\cup\overline{\overline{N}}_2^{0,-2e_{0}}\cup\overline{\overline{N}}_2^{+}$ and can be represented as a graph over it as
\[
\overline{N}_2^{-}\cup\overline{N}_2^{0,-e_{0}}\cup\overline{N}_2^{+}=\{(X_1,X_2,Y_1,Y_2)\in T^*\mathbb{T}^2\big{|}X_1=X_1^{\varepsilon}(X_2,Y_2), Y_1=Y_1^{\varepsilon}(X_2,Y_2)\}
\]
with $\|X_1^{\varepsilon}\|_{C^1}\sim\mathcal{O}(\varepsilon)$ and $\|Y_1^{\varepsilon}\|_{C^1}\sim\mathcal{O}(\varepsilon)$.\\
\item There exist locally invariant manifolds $W^{s,u}_{loc}(\overline{N}_2^{-}\cup\overline{N}_2^{0,-e_{0}}\cup\overline{N}_2^{+})$ of $\overline{N}_2^{-}\cup\overline{N}_2^{0,-e_{0}}\cup\overline{N}_2^{+}$ which are $C^1$ in $\varepsilon$.
\end{itemize}
\end{The}
\begin{proof}
We can modify $\overline{H}$ into a new system $\tilde{H}\doteq\overline{\overline{H}}+\varepsilon\rho(\frac{2e_{0}+\overline{\overline{H}}}{e_{0}})Z_3(X)$, where $\rho(x):\mathbb{R}\rightarrow\mathbb{R}$ is a $C^{\infty}$ function taking value $0$ when $x\leq0$ and $1$ when $x\geq1$ (see Figure \ref{fig11}). Then we can see that $\tilde{H}=\overline{H}$ restricted in the domain $\{\overline{\overline{H}}\geq-e_{0}\}$ and $\tilde{H}=\overline{\overline{H}}$ restricted in $\{\overline{\overline{H}}\leq-2e_{0}\}$. Besides, we have 
\begin{equation}\label{modify}
\|\varepsilon\rho(\frac{2e_{0}+\overline{\overline{H}}}{e_{0}})Z_3(X)\|_{C^2, \{|\overline{\overline{H}}|\leq e_9\sim\mathcal{O}(1)\}}\lessdot\mathcal{O}(\varepsilon).
\end{equation}
Since we know that the existence of NHIC $\overline{\overline{N}}_2^{-}\cup\overline{\overline{N}}_2^{0,-2e_{0}}\cup\overline{\overline{N}}_2^{+}$, which can be written by $\overline{\overline{N}}_{2,\pm, e_0}$ for short, we can restrict  the tangent bundle on $\overline{\overline{N}}_{2,\pm, e_0}$ and split it by 
\[
T(T^*\mathbb{T}^2)\big{|}_{T\overline{\overline{N}}_{2,\pm, e_0}}=T\overline{\overline{N}}_{2,\pm, e_0}\oplus T\overline{\overline{N}}_{2,\pm, e_0}^{\bot}. 
\]
On the other side, we can calculate the Lyapunov exponents of each of the splitting spaces because $\overline{\overline{H}}$ is autonomous and uncoupled. First we will give a definition of Lyapunov exponents for our special case of NHIC.\\
\begin{figure}
\begin{center}
\includegraphics[width=5cm]{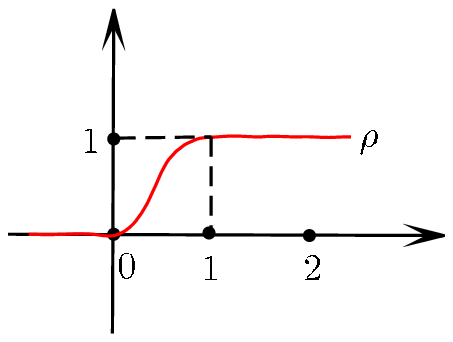}
\caption{ }
\label{fig11}
\end{center}
\end{figure}

We define the {\bf positive Lyapunov exponent} restricted on $T\overline{\overline{N}}_{2,\pm, e_0}$ by
\[
\nu^{+}_{\parallel}(Z)=\limsup_{t\rightarrow+\infty}\frac{1}{t}\ln(\frac{\|D\phi^t(Z)\cdot V\|}{\|V\|}),
\]
and the {\bf negative Lyapunov exponent}
\[
\nu^{-}_{\parallel}(Z)=\liminf_{t\rightarrow+\infty}\frac{1}{t}\ln(\frac{\|D\phi^t(Z)\cdot V\|}{\|V\|}),
\]
where ${(Z,V)\in T\overline{\overline{N}}_{2,\pm, e_0}}$ and $\|\cdot\|$ is the Euclid norm. Also we can define the {\bf positive (or negative) Lyapunov exponents} restricted on $T\overline{\overline{N}}_{2,\pm, e_0}^{\bot}$ by
\[
\nu^{+}_{\bot}(Z)=\limsup_{t\rightarrow+\infty}\frac{1}{t}\ln(\frac{\|D\phi^t(Z)\cdot V\|}{\|V\|}),
\]
and
\[
\nu^{-}_{\bot}(Z)=\liminf_{t\rightarrow+\infty}\frac{1}{t}\ln(\frac{\|D\phi^t(Z)\cdot V\|}{\|V\|}),
\]
where ${(Z,V)\in T\overline{\overline{N}}_{2,\pm, e_0}^{\bot}}$ (see \cite{Po} for preciser definitions of these). Obviously we can see that 
\[
\max_{Z\in\overline{\overline{N}}_{2,\pm,e_0}}\nu^{+}_{\parallel}(Z)=\lambda_2,
\]
\[
\min_{Z\in\overline{\overline{N}}_{2,\pm,e_0}}\nu^{-}_{\parallel}(Z)=-\lambda_2,
\]
and
\[
\nu^{\pm}_{\bot}(Z)=\pm\lambda_1,\quad\forall Z\in\overline{\overline{N}}_{2,\pm,e_0}^{\bot}.
\]
There exist $\varepsilon_0>0$ sufficient small and $\forall \varepsilon<\varepsilon_0$, we can finish the proof with the help of estimation (\ref{modify}) and the following Lemma.
\end{proof}

\begin{Lem}{\bf (Fenichel, Wiggins)}\label{F&W}
${M}$ is a compact, connected $C^r$ $(r\geq1)$ embedded manifold of $\mathbb{R}^n$, which is also invariant of the vector field $\mathcal{X}$. We have the splitting $TM\bigoplus TM^{\bot}=T\mathbb{R}^n\big{|}_{M}$ and 
\[
\frac{\nu^{+}_{\bot}}{\nu^{+}_{\parallel}}>r,\;\frac{\nu^{-}_{\bot}}{\nu^{-}_{\parallel}}>r,\quad\forall x\in M.
\]
Then $\exists\;\varepsilon>0$, $\forall\;\mathcal{Y}\in\mathcal{B}(\mathcal{X},\varepsilon)$ vector field, there exists an invariant set $M_{\mathcal{Y}}$ which is $C^r$ differmorphic to $M$.
\end{Lem}
\begin{proof}
Here we omit the proof since you can find the details in section 7 of \cite{W}, or in the paper \cite{F}.
\end{proof}
\begin{figure}
\begin{center}
\includegraphics[width=8cm]{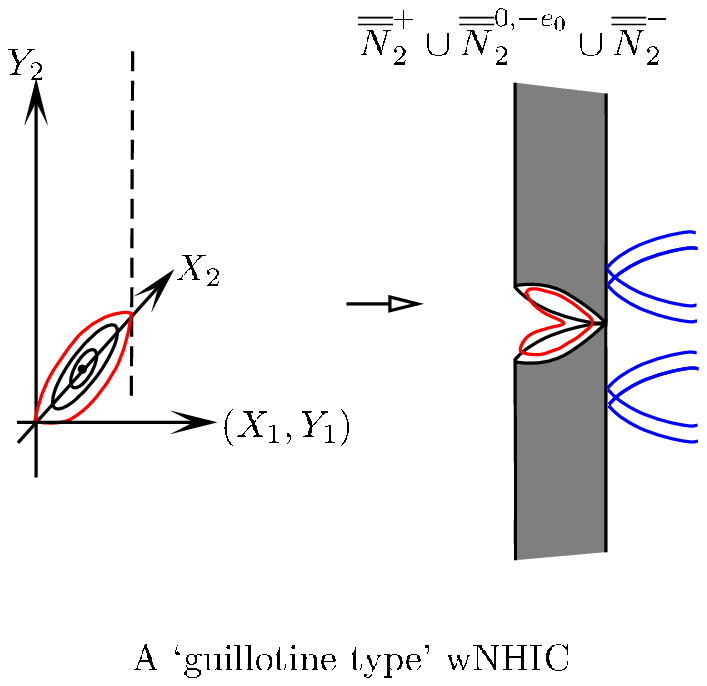}
\caption{ }
\label{fig12}
\end{center}
\end{figure}

For the system $H$, we can similarly use Theorem \ref{wiggins} and get the persistence of wNHIC $N_2^{-}\cup N_2^{0,-e_{0}}\cup{N}_2^{+}$. That's because $\forall m\geq M\gg1$, $l^{-m}\ll\varepsilon$. Besides, as a perturbation function of $(X,Y)$ variables, $\|\epsilon R\|_{C^2}\leq\mathbb{K}d_m^{*\sigma/2}l^{m(r+6+2\xi)/2}$.

\begin{Cor}\label{wiggins time-dependent}
$\exists M\gg1$ and $\forall m\geq M$, there exists a wNHIC $N_2^{-}\cup N_2^{0,-e_{0}}\cup{N}_2^{+}$ corresponding to system $H$,  for which the same properties hold as Theorem \ref{wiggins}.
\end{Cor}
\begin{Rem}
We actually don't care the exact value of $e_0$, since the locally connecting orbits are just constructed in energy surfaces with the energy larger than the Ma\~{n}\'e Critical Value. 
\end{Rem}
\begin{Rem}
 For the case of sub 2-resonant points, it's enough for us to get the persistence of this wNHIC $N_2^{-}\cup N_2^{0,-e_{0}}\cup{N}_2^{+}$. That's because there isn't any transition between resonant lines we will adopt a `transpierce' mechanism (see Figure \ref{fig12}). But for the case of $\omega_{m,1/2}$, Lemma \ref{F&W} is invalid for us to prove the persistence of wNHIC of a $g_1-$type. That's because the Lyapunov exponents $\nu_{\parallel}^{\pm}$ in this case satisfies:
 \[
\max_{Z\in\overline{\overline{N}}_1^{\pm}}\| \nu_{\parallel}^{\pm}(Z)\|=\lambda_1,
 \]
 and
 \[
\|\nu_{\bot}^{\pm}(Z)\|\equiv\lambda_2,\quad\forall Z\in\overline{\overline{N}}_2^{\pm}.
 \]
So we need to know more details about the dynamic behaviors of homoclinic orbits of system $\overline{H}$. With these new discoveries we can prove the persistence of $N_1^{\pm}$ wNHICs by sacrificing a small part near the margins. Of course, new restrictions are necessary and a much preciser calculation will be involved later (see Figure \ref{fig13}).\vspace{5pt}
\end{Rem}
\begin{figure}
\begin{center}
\includegraphics[width=8cm]{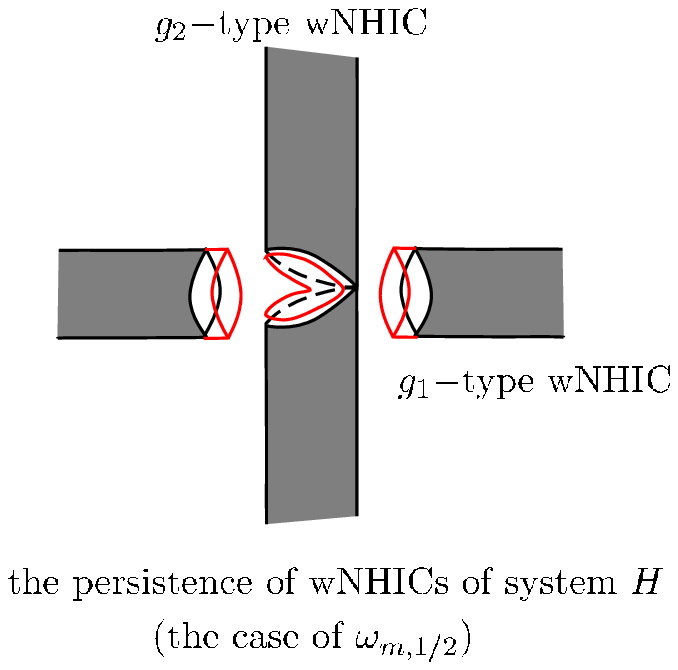}
\caption{}
\label{fig13}
\end{center}
\end{figure}

{\bf U5:} $\frac{\lambda_1}{\lambda_2}\in\mathbb{R}\setminus\mathbb{Q}$, $\forall m\in\mathbb{N}$.\\

Based on {\bf U5} restriction, we can transform system $\overline{H}$ into a Normal Form in a $\mathcal{B}(0,r)$ neighborhood of $(0,0)\in T^*\mathbb{T}^2$. Here $r=r(\lambda_1,\lambda_2)\sim\mathcal{O}(1)$.
\begin{The}{\bf (Belitskii \cite{Be}, Samovol \cite{Sa})}\label{Bel}
$\forall l\in\mathbb{N}$ and $\vec{\lambda}=(\lambda_1,\cdots,\lambda_n)\in\mathbb{C}^n$ with {\bf Re}$\lambda_i\neq0\;(i=1,2,\cdots,n)$, there exists an integer $k=k(l,\lambda)$ such that the following holds: If vector fields $V_1$ and $V_2$ have the same fixed point and their jets till order $k$ coincide with each other, then these two vector fields are $C^l-$conjugate.
\end{The}

In our case we have $n=2$ and just need $l\geq2$. Since our system $\overline{H}$ is sufficiently smooth, we can firstly transform it into a normal form of:
\begin{equation}\label{jet}
\overline{H}=\lambda_1X_1Y_1+\lambda_2X_2Y_2+\mathcal{O}(X,Y,k+2),
\end{equation}
where $(X,Y)\in\mathcal{B}(0,r)$ and $k=k(l,\lambda)$ is the order we needed in Theorem \ref{Bel}. Then we can find a $C^l$ transformation to convert (\ref{jet}) into the linear one:
\begin{equation}\label{linear jet}
\overline{H}=\overline{H}_{1,k}(X_1Y_1)+\overline{H}_{2,k}(X_2Y_2),
\end{equation}
where $(X,Y)\in\mathcal{B}(0,r)$ and $\overline{H}_{i,k}(\cdot)$ is a polynomial of order $k$, $i=1,2$. Notice that 
\[
\overline{H}'_{1,k}(0)=\lambda_1,\quad\overline{H}'_{2,k}(0)=\lambda_2.
\]
\begin{Rem}
In fact, {\bf U5} can be loosened. Now we give an explanation of this for $\overline{H}$. More general case can be found in Sec. 2 of \cite{Sh}.
\begin{defn}
We call a vector $\vec{\lambda}=(\lambda_1,\lambda_2)\in\mathbb{R}^2$  $k-$nonresonant if $\forall\vec{m}=(m_1,m_2,m_3,m_4)\in\mathbb{N}^4$ with $\sum_{i=1}^{i=4}m_i\leq k$, $m_1\neq m_2$ and $m_3\neq m_4$, we have
\[
(m_1-m_2)\lambda_1+(m_3-m_4)\lambda_2\neq0.
\]
\end{defn}
If $\lambda$ is $k-$nonresonant with $k(l,\lambda)$ satisfying Theorem \ref{Bel}, then we can also find a smooth transformation to convert $\overline{H}$ into a form of (\ref{jet}) in a small neighborhood of 0. So we actually loosen {\bf U5} into the following:\\

{\bf U5':} $\lambda$ is $k-$nonresonant, $\forall m\in\mathbb{N}$.
\end{Rem}
For system (\ref{linear jet}), we can get the local stable (unstable) manifolds of $(0,0)\in T^*\mathbb{T}^2$ by
\[
\overline{W}_{loc}^s=\{(0,0,Y_1,Y_2)\in\mathcal{B}(0,r)\},\quad \overline{W}_{loc}^{u}=\{(X_1,X_2,0,0)\in\mathcal{B}(0,r)\}.
\]
We can further get the parameter function $Y_1=\hat{C}Y_2^{\frac{\lambda_1}{\lambda_2}}$ for trajectories in $\overline{W}_{loc}^{s}$ and $X_1=\check{C}X_2^{\frac{\lambda_1}{\lambda_2}}$ for trajectories in $\overline{W}_{loc}^{u}$. On the other side, we can translate (\ref{linear jet}) into a Tonelli form 
\begin{equation}\label{linear tonelli}
\overline{H}(Q,P)=\overline{H}_{1,k}(\frac{P_1^2-Q_1^2}{2})+\overline{H}_{2,k}(\frac{P_2^2-Q_2^2}{2}),\quad(Q,P)\in\mathcal{B}(0,r)
\end{equation}
via
\[
\begin{pmatrix}
X_i\\Y_i
\end{pmatrix}=
\begin{pmatrix}
\frac{1}{\sqrt{2}} & \frac{1}{\sqrt{2}}\\
-\frac{1}{\sqrt{2}} & \frac{1}{\sqrt{2}}\\
\end{pmatrix}\cdot
\begin{pmatrix}
Q_i\\P_i
\end{pmatrix},\quad
i=1,2.
\]
This system is more convenient for us to compare the action value of trajectories. \\

From Proposition (\ref{autonomous}) we know there exists a unique $g_1-$type homoclinic orbit as the bottom of NHIC $\overline{N}_1^{+}$, which can be denoted by $\gamma$. Without loss of generality, we can project it onto the configuration space $\mathbb{T}^2$  and then suspend it in the universal space $\mathbb{R}^2$. So in the basic domain $(0,2\pi)\times(0,2\pi)$, $\gamma$ tends to $(0,0)$ as $t\rightarrow-\infty$ and tends to $(0,2\pi)$ as $t\rightarrow+\infty$ (see Figure \ref{fig14}). Besides, we need $\gamma$ leaves $(0,0)$ along the direction $\partial_{Q_2}$ and raise a new uniform restriction:\\

{\bf U6:} Under the canonical coordinations of $(Q,P)$ in the small neighborhood $\mathcal{B}(0,\frac{r}{2})$ of $(0,0)\in\mathbb{R}^2$, $\gamma$ leaves $(0,0)$ according with the trajectory function:
\[
Q_1=\hat{C}Q_2^{\frac{\lambda_1}{\lambda_2}},\quad0<\hat{C}\leq c_8(\varepsilon),
\]
which is valid for all $m\in\mathbb{N}$.\\

To make {\bf U6} satisfied, we just need to restrict $Z_3$ of a certain form in the domain $\mathcal{B}([0,2\pi]\times\{0\},\sqrt{\varepsilon})\cap\mathcal{B}((0,0),r)$. Since in $\mathcal{B}((0,0),r)$ the normal form (\ref{linear tonelli}) is valid, the coordinate of $\gamma$ should satisfy 
\[
Q_1^{\gamma}(t)=P_1^{\gamma}(t),\; Q_2^{\gamma}(t)=P_2^{\gamma}(t),
\]
for $ t\leq-t_0$ such that $(Q_1^{\gamma}(t),Q_2^{\gamma}(t),P_1^{\gamma}(t),P_2^{\gamma}(t))\in\mathcal{B}((0,0),r)$. The extreme case is that $Q_2^{\gamma}\equiv0$. Once this happens, $\gamma$ will leave $(0,0)$ along the direction $\partial_{Q_1}$, i.e. $\hat{C}=\frac{Q_1^{\gamma}}{0}=+\infty$ in a rough meaning. Now we will make a local surgery to make $\hat{C}$ finite.\\

\begin{figure}
\begin{center}
\includegraphics[width=8cm]{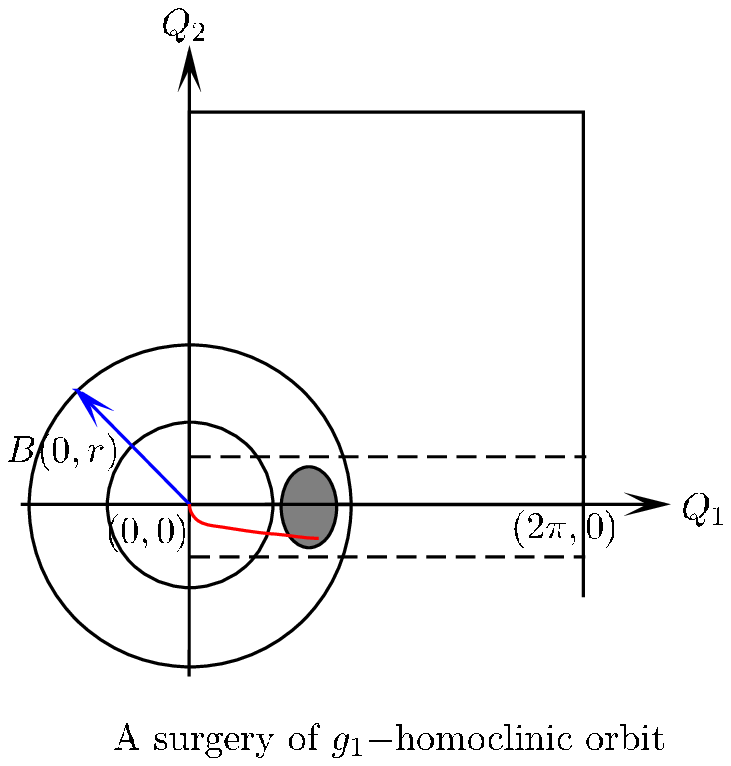}
\caption{ }
\label{fig14}
\end{center}
\end{figure}

Recall that $(Q^{\gamma}(t),P^{\gamma}(t))\in\overline{W}^u_{(0,0)}\cap\overline{W}_{(2\pi,0)}^{s}$ for all $t\in\mathbb{R}$. We can take $\varepsilon$ sufficiently small such that $\overline{W}_{(2\pi,0)}^s$ is a graph covering the domain $\mathcal{B}([0,2\pi]\times\{0\},\sqrt{\varepsilon})\cap(\mathcal{B}(0,r)\setminus\mathcal{B}(0,\frac{r}{2}))$. Besides, we know that $\gamma$ is the unique $g_1-$homoclinic orbits from Proposition \ref{autonomous}. So we just need to change the intersectional point of $\overline{W}^u_{(0,0)}$ and $\overline{W}_{(2\pi,0)}^{s}$ in this domain (see Figure \ref{fig14}). The following Lemma will help us to achieve this.
\begin{Lem}{\bf (Figalli, Rifford\cite{Fi})}\label{Rifford}
Let $\overline{H}:T^*M\rightarrow\mathbb{R}$ be a Tonelli Hamiltonian of class $C^k$ with $k\geq2$. $(Q(\cdot),P(\cdot))$ is a solution of $\overline{H}(Q,P)=0$ for $t\in[0,T]$, and satisfies the following:
\begin{itemize}
\item $(Q(0),P(0))=(0,0)$, $Q_1(T)=T$ and $\dot{Q}(0)=\partial_{Q_1}=\dot{Q}(T)$.\\
\item $\|\dot{Q}(t)-\partial_{Q_1}\|\leq\frac{1}{2},\quad\forall t\in[0,T]$.
\end{itemize}
Then $\forall \varsigma>0$ and $\varepsilon>0$, $\exists\delta=\delta(\varsigma,\varepsilon)$ such that $\forall Q_0,P_0,Q_f,P_f$ satisfying
\[
Q_0=(0,Q_{0,2}),\;Q_f=(T,Q_{f,2}),\;\|Q_{0,2}\|\leq\delta,\;\|P_0-P(0)\|\leq\delta,
\] 
we have
\[
\|(T,Q_{f,2})-(T,Q_{0,2}(T(Q_0,P_0)))\|,\;\|Q_f-P_0(T(Q_0,P_0))\|<\varsigma\varepsilon,
\]
\[
\overline{H}(Q_0,P_0)=\overline{H}(Q_f,P_f)=0,
\]
where $T(Q_0,P_0)$ satisfying $Q_{0,1}(T(Q_0,P_0))=T$ is the arrival time of flow $(Q_0(t),P_0(t))$ (see figure \ref{fig15}). There exists a time $T_f>0$, a constant $K>0$ and a potential $V:M\rightarrow\mathbb{R}$ of class $C^k$ such that:
\begin{itemize}
\item supp$(V)\subset\mathcal{C}(Q_0,T(Q_0,P_0),\varsigma)$;\\
\item $\|V\|_{C^2}\leq K\varepsilon$;\\
\item $\|T_f-T(Q_0,P_0)\|<Kr\varepsilon$;\\
\item $\phi_{\overline{H}+V}^{T_f}(Q_0,P_0)=(Q_f,P_f)$.
\end{itemize}
Here $\mathcal{C}(Q_0,T(Q_0,P_0),\varsigma)$ is the tube neighborhood defined as
\[
\mathcal{C}(Q_0,T(Q_0,P_0),\varsigma)=\{Q_0(t)+(0,y)\big{|}t\in[0,T_0(Q_0,P_0)],\|y\|\leq\varsigma\},
\]
and $\phi_{\overline{H}+V}^t(\cdot)$ is the Hamiltonian flow of $\overline{H}+V$.
\end{Lem}

\begin{figure}
\begin{center}
\includegraphics[width=9cm]{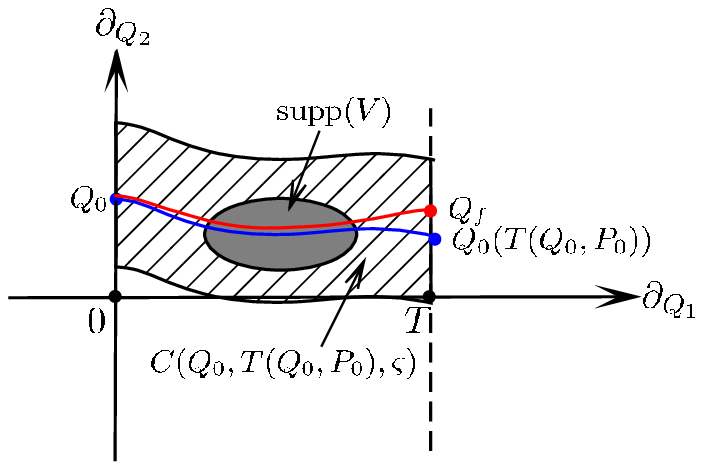}
\caption{ }
\label{fig15}
\end{center}
\end{figure}

We can use this Lemma to modify the homoclinic orbits in the open set $\Omega$ (the gray set of figure \ref{fig14}). Assuming that $\gamma$ leaves $(0,0)$ along the direction $\partial_{Q_1}$, then we have $Q_2^{\gamma}\equiv0$. 
Besides, we assume $\gamma(t)\in\Omega$ $\forall t\in[a,b]$. Then we have $T=(b-a)$ and two sections,
\[
\Sigma_{Q_1^{\gamma}(a)}\doteq\{Q_1=Q_1^{\gamma}(a)\}\cap\Omega,\quad
\Sigma_{Q_1^{\gamma}(b)}\doteq\{Q_1=Q_1^{\gamma}(b)\}\cap\Omega,
\]
as is showed in figure \ref{fig15}. If $\varsigma\sim\mathcal{O}(1)$ and $\varepsilon\ll1$ sufficiently small, we can find $Q_0\in\Sigma_{Q_1^{\gamma}(a)}$ with $\|Q_0-Q^{\gamma}(a)\|=\delta(\varsigma,\varepsilon)$, such that 
\[
Q_0(a+T(Q_0,P_0))\in\Sigma_{Q_1^{\gamma}(b)}
\]
and 
\[
\|Q_f-Q_0(a+T(Q_0,P_0))\|<\varsigma\varepsilon,
\]
where $Q_f=(Q_1^{\gamma}(b),0)$ and $(Q_1^{\gamma}(b),0,Q_1^{\gamma}(b),0)\in\overline{W}_{(2\pi,0)}^s$. Then we can take $V$ supported in $\mathcal{C}(Q_0,T(Q_0,P_0),\varsigma)$ (the gray set of figure \ref{fig15}) and the flow of $\overline{H}+V$ connects $(Q_0,P_0)$ and $(Q_f,P_f)$ by Lemma \ref{Rifford}. Since $\mathcal{B}((0,0),\frac{r}{2})\cap supp(V)=\emptyset$, the normal form of $\overline{H}+V$ is still of the form (\ref{linear tonelli}) and then we have
\[
Q_{0,1}=c_{10}(\delta)Q_{0,2}^{\frac{\lambda_1}{\lambda_2}},
\]
where $c_{10}$ depends only on $\varepsilon$. So we can find a $c_11(\varepsilon)$ constant with $c_{10}\leq c_{11}$ and {\bf U6} can be satisfied.\\

We can repeat this process in the domain $\mathcal{B}([0,2\pi]\times\{0\},\sqrt{\varepsilon})\cap(\mathcal{B}((2\pi,0),r)\setminus\mathcal{B}((2\pi,0),\frac{r}{2}))$ and modify $\gamma$ to approach $(2\pi,0)$ along the direction $\partial_{Q_2}$. Accordingly, we have:\\

{\bf U7:} Under the canonical coordinate of $(Q,P)$ in the small neighborhood $\mathcal{B}((2\pi,0),\frac{r}{2})$ of $(0,0)\in\mathbb{R}^2$, $\gamma$ approaches $(2\pi,0)$ according to the trajectory function:
\[
Q_1=\check{C}Q_2^{\frac{\lambda_1}{\lambda_2}},\quad0<\check{C}\leq c_8(\varepsilon),
\]
which is valid for all $m\in\mathbb{N}$.
\begin{Rem}
Based on these two restrictions, the unique $g_1-$homoclinic orbit will approach (and leave) the hyperbolic fixed point along the direction of $\partial_{Q_2}$. In the phase space, $(\dot{Q}(\pm\infty),\dot{P}(\pm\infty))$ are parallel to the $\pm\lambda_2-$eigenvectors (see Figure \ref{fig16}). Recall that $\|\lambda_2\|<\|\lambda_1\|$ from {\bf U3'}, so the NHIC $\overline{N}_1^{+}$ will get extra normal hyperbolicity from the hyperbolic fixed point, so does $\overline{N}_1^{-}$ from the symmetry of mechanical systems. Notice that $\overline{N}_1^{+}$ are foliated by periodic orbits, so we can get such a conclusion: the nearer the periodic orbit approaches the $g_1-$homoclinic orbit, the more normal hyperbolicity it will get as it pass by the small neighborhood of $0\in T^*\mathbb{T}^2$. On the other side, the nearer the periodic orbit approaches the homoclinic orbit, the longer its periods is. So we need a precise calculation of this competition relationship to persist as large as possible part of NHIC for the ultimate system $H$.
\end{Rem}
\begin{figure}
\begin{center}
\includegraphics[width=12cm]{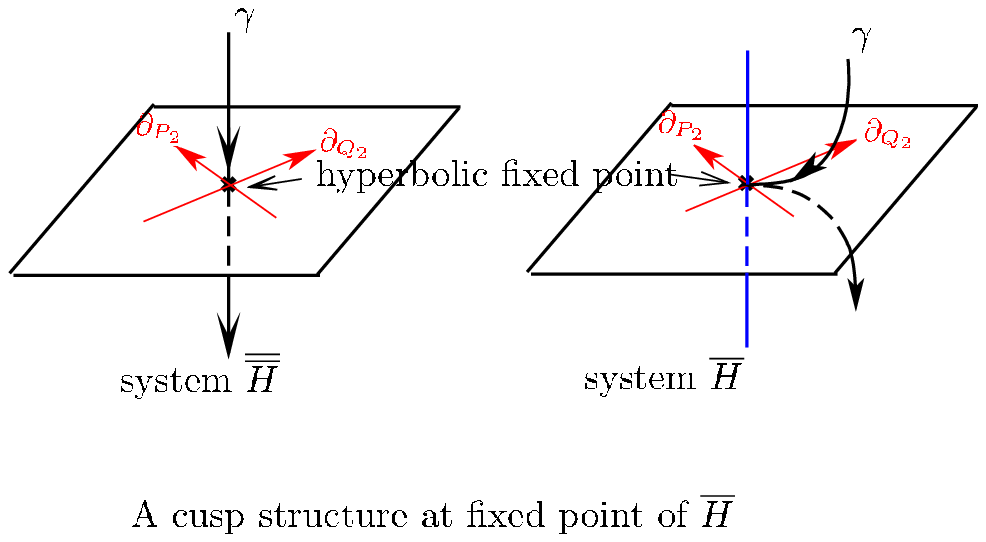}
\caption{ }
\label{fig16}
\end{center}
\end{figure}
\vspace{5pt}
As mentioned before, $\overline{N}_1^{+}$ is foliated by a list of periodic orbits written by $\gamma_{E}$ with different periods $T_E$. We use the subscript `E' to remind the readers in which energy surface $\gamma_{E}$ lies.
\begin{Lem}
When $E_0\ll c_{12}(\varepsilon)$ sufficiently small, we can estimate the period by $T_E=-\frac{C_E}{\lambda_2}\ln\frac{E}{c_{10}}$, where $C_E\sim\mathcal{O}(1)$ is uniformly bounded for all $E\in[0,E_0]$. 
\end{Lem}
\begin{proof}
That's estimation was gotten in \cite{Ch} with a tedious but simple computation. Here we just explain the idea of that. When $E$ sufficiently small, the time for $\gamma_E$ passing by $\mathcal{B}(0,c_{12})$ will be much longer than what it spends outside. Besides, normal form (\ref{linear tonelli}) is available in this domain $\mathcal{B}(0,c_{12})$. As long as $\gamma_E$ approaches the homoclinic orbit $\gamma$ closely, {\bf U6,7} will restrict the enter position $(Q_E(0),P_E(0))$ and exit position $(Q_E(T),P_E(T))$ of $\gamma_E$ effectively. Here we take $c_{12}$ properly small comparing to $c_{11}$ and assume $T$ as the time of $\gamma_E$ staying in $\mathcal{B}(0,c_{12})$. On the other side, we have
\[
E=\overline{H}_{1,k}(\frac{P_1(0)^2-Q_1(0)^2}{2})+\overline{H}_{2,k}(\frac{P_2(0)^2-Q_2(0)^2}{2})
\]
and
\[
E=\overline{H}_{1,k}(\frac{P_1(T)^2-Q_1(T)^2}{2})+\overline{H}_{2,k}(\frac{P_2(T)^2-Q_2(T)^2}{2}).
\]
So we can get an estimate of $E$ and $T$, since the 1-order term of $\overline{H}_{i,k}$ occupies the main value of the right side of above formulas, $i=1,2$.
\begin{equation}\label{local time}
T=\frac{1}{\lambda_2}\ln\frac{c_{12}^2}{E}+\tau_{c_{12}},
\end{equation}
where $\tau_{c_{12}}\sim\mathcal{O}(1)$ is uniformly bounded for $E\in[0,c_{12}]$. So we get the period estimation of a form $T_E=-\frac{C_E}{\lambda_2}\ln\frac{E}{c_{12}}$ and Lemma is proved.
\end{proof}
\begin{figure}
\begin{center}
\includegraphics[width=10cm]{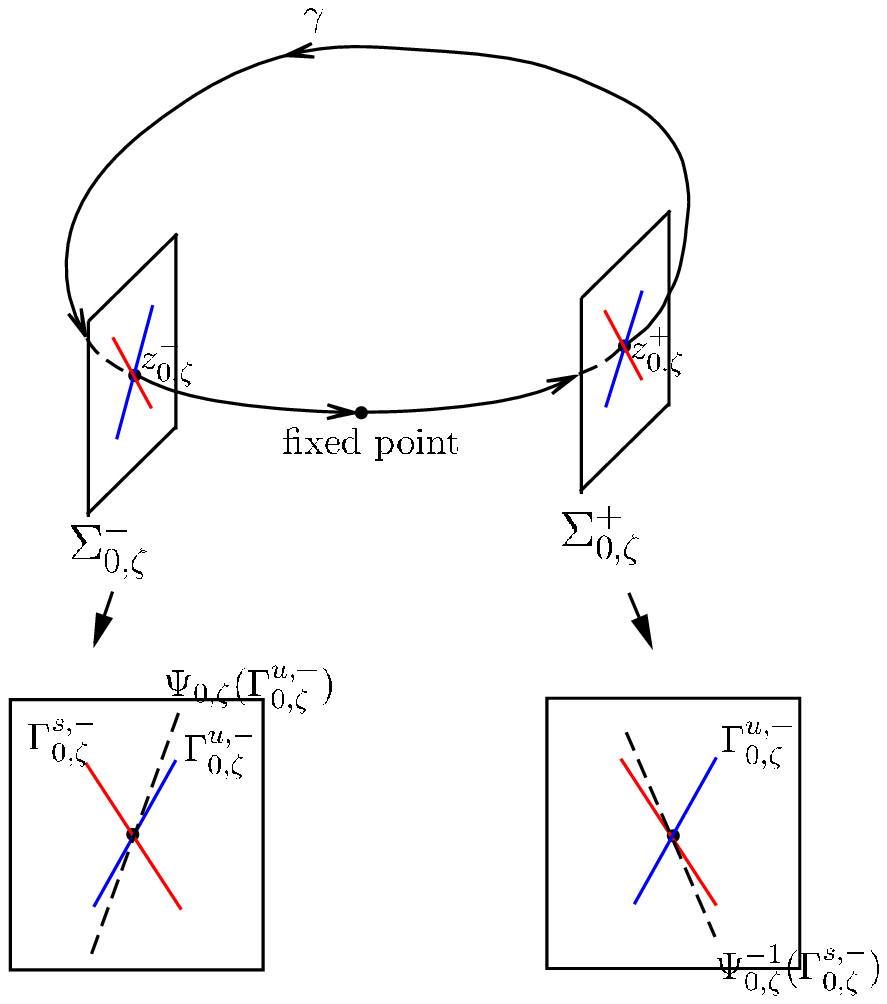}
\caption{ }
\label{fig17}
\end{center}
\end{figure}
Recall that $\frac{Q^{\gamma}}{\|Q^{\gamma}\|}\rightarrow(0,1)$ as $t\rightarrow\pm\infty$ from {\bf U6,7}. Then for sufficiently small energy $E$, we can find a couple of 2-dimensional sections
\[
\Sigma_{E,\zeta}^{\pm}\doteq\{(Q,P)\in\mathbb{R}^4\big{|}\|(Q,P)\|\leq c_{12}, \overline{H}(Q,P)=E,Q_2=\pm\zeta\},\quad E\in[0,E_0].
\]
Then on the energy surface $\overline{H}^{-1}(0)$, $\overline{W}^{u}$ will intersect $\Sigma_{0,\zeta}^{+}$ transversally of a 1-dimensional submanifold, which can be written by $\Gamma_{0,\zeta}^{u,+}$. Similarly $\overline{W}^s$ intersects $\Sigma_{0,\zeta}^{-}$ of a 1-dimensional submanifold $\Gamma_{0,\zeta}^{s,-}$. We also know that $\gamma$ passes across $\Sigma_{0,\zeta}^{\pm}$ and denote the intersectional points by $z_{0,\zeta}^{\pm}\in T^*\mathbb{T}^2$. It's obvious that $z_{0,\zeta}^{+}\in\Gamma_{0,\zeta}^{u,+}$ and $z_{0,\zeta}^{-}\in\Gamma_{0,\zeta}^{s,-}$. So we could define a global mapping
\begin{equation}
\Psi_{0,\zeta}:\Sigma_{0,\zeta}^{+}\rightarrow\Sigma_{0,\zeta}^{-}
\end{equation}
with $\Psi_{0,\zeta}(z_{0,\zeta}^{+})=z_{0,\zeta}^{-}$ (see Figure \ref{fig17}). From the $\lambda-$Lemma \cite{Pa} we know that $\Psi_{0,\zeta}(\Gamma_{0,\zeta}^{u,+})$ is $C^1-$close to $\Gamma_{0,\zeta}^{u,-}$ at the point $z_{0,\zeta}^{-}$, where $\Gamma_{0,\zeta}^{u,-}\doteq\overline{W}^u\cap\Sigma_{0,\zeta}^{-}$. This is because $\Psi_{0,\zeta}(\Gamma_{0,\zeta}^{u,+})\subset\Sigma_{0,\zeta}^{-}$ and $\Gamma_{0,\zeta}^{u,+}\subset\overline{W}^u$, so we have:
\begin{equation}
\|\langle \frac{D\Psi_{0,\zeta}(z_{0,\zeta}^{+})\cdot v}{\|D\Psi_{0,\zeta}(z_{0,\zeta}^{+})\cdot v\|},v'\rangle\|\geq\frac{4}{5}\|v'\|,\quad \forall \;v\in T_{z_{0,\zeta}^{+}}\Gamma_{0,\zeta}^{u,+},\quad v'\in T_{z_{0,\zeta}^{-}}\Gamma_{0,\zeta}^{u,-}
\end{equation}
and
\begin{equation}
\|\langle \frac{D\Psi^{-1}_{0,\zeta}(z_{0,\zeta}^{-})\cdot w}{\|D\Psi^{-1}_{0,\zeta}(z_{0,\zeta}^{-})\cdot w\|},w'\rangle\|\geq\frac{4}{5}\|w'\|,\quad \forall\; w\in T_{z_{0,\zeta}^{-}}\Gamma_{0,\zeta}^{s,-},\quad w'\in T_{z_{0,\zeta}^{+}}\Gamma_{0,\zeta}^{s,+},
\end{equation}
provided $\zeta>0$ is sufficiently small. The following is a sketch-map of $\lambda-$Lemma for 2-dimensional mappings, which can be helpful to reader's understanding (see Figure \ref{fig18}).\\
\begin{figure}
\begin{center}
\includegraphics[width=8cm]{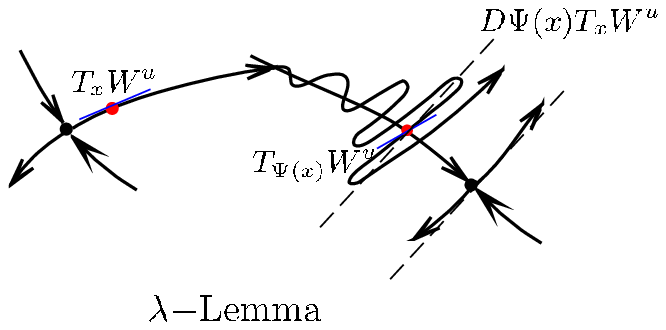}
\caption{ }
\label{fig18}
\end{center}
\end{figure}

For sufficiently small $E>0$, $\Sigma_{E,\zeta}^{\pm}$ is $C^{r-1}$-close to $\Sigma_{0,\zeta}^{\pm}$. Let $z_{E,\zeta}^{\pm}$ be intersectional points of $\gamma_E$ with $\Sigma_{E,\zeta}^{\pm}$, we can similarly define
\begin{equation}
\Psi_{E,\zeta}:\Sigma_{E,\zeta}^{+}\rightarrow\Sigma_{E,\zeta}^{-}
\end{equation}
with $\Psi_{E,\zeta}(z_{E,\zeta}^{+})=z_{E,\zeta}^{-}$. Because of the smooth dependence of ODE solutions on initial data, there exists a sufficiently small $\nu>0$ such that $\forall$ vector $v\in T_{z_{E,\zeta}^{+}}\Sigma_{E,\zeta}^{+}$ which is $\nu-$parallel to $T_{z_{0,\zeta}^{+}}\Gamma_{0,\zeta}^{u,+}$ in the sense that $\|\langle v,v_0\rangle\|\geq(1-\nu)\|v\|\|v_0\|$ holds for any $ v_0\in T_{z_{0,\zeta}^{+}}\Gamma_{0,\zeta}^{u,+}$, we have
\begin{equation}
\|\langle \frac{D\Psi_{E,\zeta}(z_{E,\zeta}^{+})\cdot v}{\|D\Psi_{E,\zeta}(z_{E,\zeta}^{+})\cdot v\|},v'\rangle\|\geq\frac{2}{3}\|v'\|,\quad \forall v'\in T_{z_{0,\zeta}^{-}}\Gamma_{0,\zeta}^{u,-}.
\end{equation}
Similarly $\forall w\in T_{z_{E,\zeta}^{-}}\Sigma_{E,\zeta}^{-}$ which is $\nu-$parallel to $T_{z_{0,\zeta}^{-}}\Gamma_{0,\zeta}^{s,-}$, we have
\begin{equation}
\|\langle \frac{D\Psi^{-1}_{E,\zeta}(z_{E,\zeta}^{-})\cdot w}{\|D\Psi^{-1}_{E,\zeta}(z_{E,\zeta}^{-})\cdot w\|},v'\rangle\|\geq\frac{2}{3}\|v'\|,\quad \forall v'\in T_{z_{0,\zeta}^{+}}\Gamma_{0,\zeta}^{s,+}.
\end{equation}
Once $\zeta$ is fixed, we have $c_{13}(\zeta)>1$ such that
\begin{equation}
c_{13}^{-1}\leq\|D\Psi_{0,\zeta}(z_{0,\zeta}^{+})\big{|}_{T_{z_{0,\zeta}^{+}}\Gamma_{0,\zeta}^{u,+}}\|,\;\|D\Psi_{0,\zeta}^{-1}(z_{0,\zeta}^{-})\big{|}_{T_{z_{0,\zeta}^{-}}\Gamma_{0,\zeta}^{s,-}}\|\leq c_{13}.
\end{equation}
Clearly $c_{13}\rightarrow\infty$ as $\zeta\rightarrow 0$. For sufficiently small $E>0$, we have
\begin{equation}
(2c_{13})^{-1}\leq\frac{\|D\Psi_{E,\zeta}(z_{E,\zeta}^{+})v\|}{\|v\|},\;\frac{\|D\Psi_{E,\zeta}^{-1}(z_{E,\zeta}^{-})w\|}{\|w\|}\leq 2c_{13},
\end{equation}
where $v$ is $\nu-$parallel to $T_{z_{0,\zeta}^{+}}\Gamma_{0,\zeta}^{u,+}$ and $w$ is $\nu-$parallel to $T_{z_{0,\zeta}^{-}}\Gamma_{0,\zeta}^{s,-}$.\\

Besides, for $E>0$ we have a local mapping:
\[
\Phi_{E,\zeta}:\Sigma_{E,\zeta}^{-}\rightarrow\Sigma_{E,\zeta}^{+},
\]
with $\Phi_{E,\zeta}(z_{e,\zeta}^{-})=z_{E,\zeta}^{+}$. Since $\zeta\ll r$, normal form (\ref{linear tonelli}) is available and the time from $z_{E,\zeta}^{-}$ to $z_{E,\zeta}^{+}$ is about $T=\frac{1}{\lambda_2}\ln\frac{\zeta^2}{E}+\tau_{\zeta}$ (see formula (\ref{local time})), where $\tau_{\zeta}$ is uniformly bounded as $\zeta\rightarrow0$. On the other side, as (\ref{linear tonelli}) is uncoupled, for an arbitrary vector $v$ $\nu-$parallel to $T_{z_{0,\zeta}^{-}}\Gamma_{0,\zeta}^{u,-}$, $D\Phi_{E,\zeta}(z_{E,\zeta}^{-})v$ is also $\nu-$parallel to $T_{z_{0,\zeta}^{+}}\Gamma_{0,\zeta}^{u,+}$. Analogously, for all $w$ which is $\nu-$parallel to $T_{z_{0,\zeta}^{+}}\Gamma_{0,\zeta}^{s,+}$, $D\Phi_{E,\zeta}^{-1}(z_{E,\zeta}^{+})w$ is also $\nu-$parallel to $T_{z_{0,\zeta}^{-}}\Gamma_{0,\zeta}^{s,-}$. Furthermore, we have
\begin{equation}
c_{14}^{-1}(\frac{\zeta^2}{E})^{\frac{\lambda_1}{\lambda_2}-\mu}\leq\frac{\|D\Phi_{E,\zeta}(z_{E,\zeta}^{-})v\|}{\|v\|}\leq c_{14}(\frac{\zeta^2}{E})^{\frac{\lambda_1}{\lambda_2}-\mu}
\end{equation}
and
\begin{equation}
c_{14}^{-1}(\frac{\zeta^2}{E})^{\frac{\lambda_1}{\lambda_2}-\mu}\leq\frac{\|D\Phi_{E,\zeta}^{-1}(z_{E,\zeta}^{+})w\|}{\|w\|}\leq c_{14}(\frac{\zeta^2}{E})^{\frac{\lambda_1}{\lambda_2}-\mu},
\end{equation}
where $c_{14}>1$ is a $\mathcal{O}(1)$ constant and $\mu\rightarrow0$ as $\zeta\rightarrow 0$.

The composition of the two mappings above constitutes a Poincar\'e recurrent mapping
\begin{equation}
\Psi_{E,\zeta}\circ\Phi_{E,\zeta}:\Sigma_{E,\zeta}^{-}\rightarrow\Sigma_{E,\zeta}^{-},
\end{equation}
with $\Psi_{E,\zeta}\circ\Phi_{E,\zeta}(z_{E,\zeta}^{-})=z_{E,\zeta}^{-}$. Then we have
\begin{equation}
(2c_{14}c_{13})^{-1}(\frac{\zeta^2}{E})^{\frac{\lambda_1}{\lambda_2}-\mu}
\leq\frac{\|D(\Psi_{E,\zeta}\circ\Phi_{E,\zeta})(z_{E,\zeta}^{-})v\|}{\|v\|}\leq 2c_{14}c_{13}(\frac{\zeta^2}{E})^{\frac{\lambda_1}{\lambda_2}-\mu}
\end{equation}
for $v$ $\nu-$parallel to $T_{z_{0,\zeta}^{-}}\Gamma_{0,\zeta}^{u,-}$ and 
\begin{equation}
(2c_{14}c_{13})^{-1}(\frac{\zeta^2}{E})^{\frac{\lambda_1}{\lambda_2}-\mu}
\leq\frac{\|D(\Psi_{E,\zeta}\circ\Phi_{E,\zeta})^{-1}(z_{E,\zeta}^{-})w\|}{\|w\|}\leq 2c_{14}c_{13}(\frac{\zeta^2}{E})^{\frac{\lambda_1}{\lambda_2}-\mu}
\end{equation}
for $w$ $\nu-$parallel to $T_{z_{0,\zeta}^{-}}\Gamma_{0,\zeta}^{s,-}$.\\

From these two inequalities above, we can see that $z_{E,\zeta}^{-}$ is a hyperbolic fixed point of $\Psi\circ\Phi_{E,\zeta}$. We denote by $\overline{W}_{E}^{s}$ $(\overline{W}_E^u)$ the stable (unstable) manifold corresponding to $\gamma_E$. Besides, we also have 
\[
\Gamma_{E,\zeta}^{u,\pm}\doteq\overline{W}_{E}^u\cap\Sigma_{E,\zeta}^{\pm},\quad \Gamma_{E,\zeta}^{s,\pm}\doteq\overline{W}_E^s\cap\Sigma_{E,\zeta}^{\pm}.
\]
From {\bf U5,6}, we can take $E_0\leq(2c_{14}c_{13})^{-\frac{2}{\iota}}\zeta^{\frac{4\lambda_1}{\iota\lambda_2}}$ then get the following
\begin{Lem}\label{recurrent}
$\forall E\in(0,E_0]$, the recurrent mapping $\Psi_{E,\zeta}\circ\Phi_{E,\zeta}$ satisfying
\begin{itemize}
\item $\|D(\Psi_{E,\zeta}\circ\Phi_{E,\zeta})(z_{E,\zeta}^{-})v_E^u\|\geq E^{-(\frac{\lambda_1}{\lambda_2}-\iota)}\|v_E^u\|$, $\forall v_E^u\in T_{z_{E,\zeta}^{-}}\Gamma_{E,\zeta}^{u,-}$,\\
\item $\|D(\Psi_{E,\zeta}\circ\Phi_{E,\zeta})^{-1}(z_{E,\zeta}^{-})v_E^s\|\geq E^{-(\frac{\lambda_1}{\lambda_2}-\iota)}\|v_E^s\|$, $\forall v_E^s\in T_{z_{E,\zeta}^{-}}\Gamma_{E,\zeta}^{s,-}$,
\end{itemize}
where $1+10^3\iota<\frac{\lambda_1}{\lambda_2}$ and $\iota\sim\mathcal{O}(1)$, as long as $\zeta$ is chosen sufficiently small.
\end{Lem}

For $E_1<E_0$, The segment of NHIC $\overline{N}_{1,E_1,E_0}^{+}$ is a 2-dimensional symplectic sub manifold. We can restrict symplectic 2-form $\omega$ to $\overline{N}_{1,E_1,E_0}^{+}$, which is equivalent to the area form $\Omega$. Recall that $\overline{N}_{1,E_1,E_0}^{+}$ is an invariant manifold under the flow mapping $\phi_{\overline{H}}^t$, then $\det|D\phi_{\overline{H}}(z)|\equiv1$, $\forall z\in\overline{N}_{1,E_1,E_0}^{+}$ and $t\in\mathbb{R}$. So the eigenvalues of $D\phi_{\overline{H}}^t(z)$ must appear in pairs of the form $\lambda(z)$ and $\lambda(z)^{-1}$.\\

On the other side, the normal form (\ref{linear tonelli}) is available in the domain $\mathcal{B}(0,c_{11})$. Then we have
\[
E=\frac{\lambda_1}{2}(P_1^2-Q_1^2)+\frac{\lambda_2}{2}(P_2^2-Q_2^2)+\mathcal{O}(Q,P,4)
\]
and the ODE
\begin{equation}\label{linear tonelli ODE}
\left\{
\begin{array}{cccccc}
\dot{Q}_i=\frac{\partial\overline{H}}{\partial P_i}=\lambda_1 P_i+\mathcal{O}(Q_i,P_i,3),\vspace{5pt}\\
\dot{P}_i=-\frac{\partial\overline{H}}{\partial Q_i}=\lambda_1 Q_i+\mathcal{O}(Q_i,P_i,3),
\end{array}
\right.
\end{equation}
where $i=1,2$. Then we have 
\[
\min\|\dot{z}_E\|\geq c_{15}(\lambda_2)\sqrt{E}, 
\]
where $z_{E}(t)=(Q_E(t),P_E(t))$ is the trajectory of $\gamma_E$ and $c_{15}\sim\mathcal{O}(1)$ is a constant depending on $\lambda_2$. Besides, $\dot{z}_E(t)$ is just an eigenvector of $D\phi_{\overline{H}}^{t}(z_E(t))$. Therefore, $\exists c_{16}>1$ such that  
\begin{equation}
\inf_{t\in\mathbb{R}} \|D\phi_{\overline{H}}^t(z_E(t))v\|\geq\frac{\sqrt{E}}{c_{16}}\|v\|,\quad\max_{t\in\mathbb{R}}\|D\phi_{\overline{H}}^t(z_E(t))v\|\leq\frac{c_{16}}{\sqrt{E}}\|v\|
\end{equation}
holds for each vector $v$ tangent to the periodic orbit $z_E(t)$, $\forall t\in\mathbb{R}$. In fact, above formulae give us a control of $\|D\phi_{\overline{H}}^t\|$ on $\overline{N}_{1,E_1,E_0}^{+}$. We can make a comparison between $|D\phi_{\overline{H}}^t\|_{\overline{N}_{1,E_1,E_0}^{+}}$ and $\|D\Psi_{E,\zeta}\circ\Phi_{E,\zeta}\|$.
\begin{Lem}\label{NHIC time}
$\overline{N}_{1,E_1,E_0}^{+}$ is NHIC under $\phi_{\overline{H}}^t$, where $t=\frac{2}{\lambda_2}\ln\frac{1}{E_1}$.
\end{Lem}
\begin{proof}
As is known to us from (\ref{local time}), $\tau_E$ is uniformly bounded as $E\rightarrow0$, so $T_E\leq\frac{2}{\lambda_2}\ln\frac{1}{E}$ when $E_0$ is sufficiently small. 

On the other side, the tangent bundle of $T^*\mathbb{T}^2$ over $\overline{N}_{1,E_1,E_0}^{+}$ admits such a $D\phi_{\overline{H}}^t\big{|}_{t=T_{E_1}}-$invariant splitting
\[
T_{z}T^*\mathbb{T}^2=T_zN_{\bot}^{+}\oplus T_z\overline{N}_{1,E_1,E_0}^{+}\oplus T_zN_{\bot}^{-},\quad z\in\overline{N}_{1,E_1,E_0}^{+}.
\]
Besides,
\begin{eqnarray}
\frac{\sqrt{E}}{c_{16}}\leq\frac{\|D\phi_{\overline{H}}^t(z)v\|}{\|v\|}\leq\frac{c_{16}}{\sqrt{E}},\quad\forall v\in T_z\overline{N}_{1,E_1,E_0}^{+},\\
\frac{\|D\phi_{\overline{H}}^t(z)v\|}{\|v\|}\geq E^{-(\frac{\lambda_1}{\lambda_2}-\iota)}, \quad\forall v\in T_zN_{\bot}^{-},\\
\frac{\|D\phi_{\overline{H}}^t(z)v\|}{\|v\|}\leq E^{(\frac{\lambda_1}{\lambda_2}-\iota)}, \quad\forall v\in T_zN_{\bot}^{+},
\end{eqnarray}
holds for $t\geq\frac{2}{\lambda_2}\ln\frac{1}{E_1}$. Since $\frac{\lambda_1}{\lambda_2}>\iota+1>\frac{1}{2}$ from {\bf U3'} and Lemma \ref{recurrent}. Then we finished the proof.
\end{proof}
For $E\in[E_0,E_2]$ with $E_2\sim\mathcal{O}(1)$, we can see that $\overline{N}_{1,E_0,E_2}^{+}$ is also NHIC under $\phi_{\overline{H}}^t$ with $t\geq\frac{2}{\lambda_2}\ln\frac{1}{E_1}$. Now we explain this. Recall that we can find a hyperbolic fixed point $z_{E,\pi}\in\Sigma_{E,\pi}$ of the recurrent mapping $P_E:\Sigma_{E,\pi}\rightarrow\Sigma_{E,\pi}$ from Proposition \ref{autonomous}, where $\Sigma_{E,\pi}=\{(X,Y)\in T^*\mathbb{T}^2\big{|}X_1=\pi,\|X_2\|\leq\sqrt{\varepsilon}\}$ is a local section. Besides, we have
\begin{eqnarray}
\frac{\|DP_E(z_{E,\pi})v\|}{\|v\|}\geq\exp^{(\lambda_2-c_{17}(\varepsilon))T_E},\quad\forall v\in T_{z_{E,\pi}}(\overline{W}_E^u\cap\Sigma_{E,\pi}),\\
\frac{\|DP_E(z_{E,\pi})v\|}{\|v\|}\leq\exp^{-(\lambda_2-c_{17}(\varepsilon))T_E},\quad\forall v\in T_{z_{E,\pi}}(\overline{W}_E^s\cap\Sigma_{E,\pi}),
\end{eqnarray}
where $c_{17}(\varepsilon)\rightarrow0$ as $\varepsilon\rightarrow0$. On the other side, $\forall v\in T_{z_{E,\pi}}\overline{N}_{1,E_0,E_2}^{+}$, 
\begin{equation}
\frac{\sqrt{E}_0}{c_{16}}\leq\frac{\|D\phi_{\overline{H}}^t(z)v\|}{\|v\|}\leq\frac{c_{16}}{\sqrt{E}_0},\quad\forall t\in\mathbb{R}.
\end{equation}
If we take $t\geq\frac{2}{\lambda_2}\ln\frac{1}{E_1}$, then $T_E<t$ for all $E\in[E_0,E_2]$. So we have
\begin{eqnarray}
\frac{\sqrt{E}_0}{c_{16}}\leq\frac{\|D\phi_{\overline{H}}^t(z_{E,\pi})v\|}{\|v\|}\leq\frac{c_{16}}{\sqrt{E}_0},\quad\forall v\in T_{z_{E,\pi}}\overline{N}_{1,E_0,E_2}^{+},\\
\frac{\|D\phi_{\overline{H}}^t(z_{E,\pi})v\|}{\|v\|}\geq E_1^{-2(1-\frac{c_{17}}{\lambda_2})}, \quad\forall v\in T_{z_{E,\pi}}(\overline{W}_E^u\cap\Sigma_{E,\pi}),\\
\frac{\|D\phi_{\overline{H}}^t(z_{E,\pi})v\|}{\|v\|}\leq E_1^{2(1-\frac{c_{17}}{\lambda_2})}, \quad\forall v\in T_{z_{E,\pi}}(\overline{W}_E^s\cap\Sigma_{E,\pi}).
\end{eqnarray}
As $E_1<E_0$, so we get the normal hyperbolicity of $\overline{N}_{1,E_0,E_2}^{+}$ accordingly.
\begin{Rem}
From the analysis of normal hyperbolicity above, we can see that $\overline{N}_{1,E_1,E_2}^{+}$ is NHIC under $\phi_{\overline{H}}^t$ with $t\geq\frac{2}{\lambda_2}\ln\frac{1}{E_1}$. Notice that $t\rightarrow+\infty$ as $E_1\rightarrow0$. This leads to a technical flaw and prevent us from proving the persistence of $\overline{N}_{1,E_1,E_2}^{+}$ with $E_1=0$ for system $H$, using the classical invariant manifold theory \cite{HPS}. In the following we will give a precise estimation of $E_1>0$ and prove the existence of wNHIC $N_{1,E_1,E_2}^{+}$ by sacrificing a small segment near the bottom.
\end{Rem}
\vspace{10pt}

Recall that system $H$ is of a form
\begin{equation}
\overline{H}=\frac{1}{2}\langle Y^t, Y\rangle+Z_1(X_1)+Z_2(X_2)+\varepsilon{Z}_3(X_1,X_2)+\epsilon R(X,Y,S),
\end{equation}
where $\varepsilon=\frac{1}{\mathbb{L}}$ and $\epsilon=\mathbb{K}d_m^{*\sigma/2}l^{m(r+6+2\xi)/2}$ and we can take $m\geq M\gg1$ sufficiently large and $\epsilon\ll\varepsilon$. Once again we use the self-similar structure. The target of this part is to get the exact value of $E_1$. We can assume it by $E_1=\epsilon^d$ with $d>0$. Later the analysis will give $d$ a proper value. Similar as above or \cite{Ch}, we will modify $H$ into 
\begin{equation}\label{modified H}
H'=\frac{1}{2}\langle Y^t, Y\rangle+Z_1(X_1)+Z_2(X_2)+\varepsilon{Z}_3(X_1,X_2)+\epsilon \rho(\frac{\overline{H}-\epsilon^d}{\epsilon^d})R(X,Y,S),
\end{equation}
where $\rho(\cdot)$ is the same with the one of Theorem \ref{wiggins} (see figure \ref{fig11}). We can see that $H'=H$ in the domain $\{\overline{H}\geq2\epsilon^d\}$ and $H'=\overline{H}$ in the domain $\{\overline{H}\leq\epsilon^d\}$. If we can prove that there exists a NHIC under $\phi_{H'}^t$ in the domain $\{\overline{H}\geq\frac{\epsilon^d}{2}\}$, then $\{\gamma_{E=\frac{\epsilon^d}{2}}(S),S\}$ must be the bottom of this NHIC. This invariant cylinder verifies the existence of wNHIC for $H$ system, as $H'=H$ in the domain $\{\overline{H}\geq2\epsilon^d\}$. This is the main idea to prove the persistence of $g_1-$wNHIC for $H$, which was firstly used in \cite{Ch} to prove a similar result.
\begin{Lem}
Let the equation $\dot{z}=F_{\epsilon}(z,t)$ be a small perturbation of $\dot{z}=F_0(z,t)$, with $\phi_{\epsilon}^t$ and $\phi_0^t$ the flow determined respectively. Then we have:
\[
\|\phi_{\epsilon}^t(z)-\phi_0^t(z)\|_{C^1}\leq\frac{B}{A}(1-e^{-At})e^{2At},
\]
where 
\[
A=\max_{0\leq s\leq t}\{\|F_0(\cdot,s)\|_{C^2},\|F_{\epsilon}(\cdot,s)\|_{C^2}\}
\] 
and 
\[
B=\max_{0\leq s\leq t}\|F_{\epsilon}(\cdot,s)-F_0(\cdot,s)\|_{C^1}.
\]
Here $\|\cdot\|_{C^2}$ and $\|\cdot\|_{C^1}$ only depend on $z-$variables, and $supp_{z}(F_{\varepsilon}-F_0)\subset\mathbb{R}^n$.
\end{Lem}
\begin{proof}
Let $\Delta z(t)=z_{\varepsilon}(t)-z(t)$ and 
\[
\Delta\dot{z}=\nabla F_{\epsilon}(\nu(t)z+(1-\nu(t))z_{\epsilon})\Delta z+(F_{\epsilon}-F_0)(z),
\]
where $\nu(t)\in[0,1]$. Therefore,
\[
\|\Delta\dot{z}\|\leq\max\|\nabla F_{\epsilon}\|\|\Delta z\|+\max\|F_{\epsilon}-F_0\|.
\]
It follows from Gronwall's inequality that
\[
\|\Delta z(t)\|\leq\frac{B}{A}(e^{At}-1).
\]

On the other side, the variational equation of $D\phi_{\lambda}^t$ $(\lambda=\epsilon,0)$ is the following:
\[
\frac{d}{dt}D\phi_{\lambda}^t(z_{\lambda(t)})=\nabla F_{\lambda}(z_{\lambda}(t))D\phi_{\lambda}^t(z_{\lambda}(t)),\quad\lambda=\epsilon,0.
\]
Therefore, for each tangent vector $v\in T_{z_{\lambda}(0)}\mathbb{R}^n$ one has
\[
\|D\phi_{\lambda}^t(z_{\lambda}(0))v\|\leq\|v\|e^{At}.
\] 
As for $D(\phi_{\epsilon}^t-\phi_0^t)$, we have
\begin{eqnarray*}
\frac{d}{dt}(D\phi_{\epsilon}^t(z_{\epsilon})-D\phi_0^t(z))&=&\nabla F_{\epsilon}(z_{\epsilon})D\phi_{\epsilon}^t-\nabla F_{0}(z)D\phi_{0}^t\\
&=& DF_{\epsilon}(z(t))D\phi_{\epsilon}^t(z_{\epsilon})+D^2F_{\epsilon}(\mu z_{\epsilon}+(1-\mu)z)\Delta zD\phi_{\epsilon}^t\\
&\quad-& DF_0(z(t))D\phi_{0}^t(z)\\
&=& DF_{\epsilon}(z)(D\phi_{\epsilon}^t-D\phi_0^t)+(DF_{\epsilon}(z)-DF_0(z))D\phi_0^t\\
&\quad+& D^2F_{\epsilon}(\mu z_{\epsilon}+(1-\mu)z)\Delta z D\phi_{\epsilon}^t.
\end{eqnarray*}
Then we have
\begin{eqnarray*}
\|\frac{d}{dt}(D\phi_{\epsilon}^t(z_{\epsilon})-D\phi_0^t(z))\|&\leq& A\|D\phi_{\epsilon}^t(z_{\epsilon})-D\phi_0^t(z)\|+Be^{At}+B(e^{At}-1)e^{At}\\
&\leq&A\|D\phi_{\epsilon}^t(z_{\epsilon})-D\phi_0^t(z)\|+Be^{2At},
\end{eqnarray*}
and
\[
\|D\phi_{\epsilon}^t-D\phi_0^t\|\leq\frac{B}{A}(e^{2At}-e^{At})
\]
from Gronwall inequality again. Then we complete the proof.
\end{proof}

Now we use this Lemma for our systems $H'$ and $\overline{H}$. We have the estimation 
\begin{equation}\label{peturbation estimation}
\|D\phi_{H'}^t-D\phi_{\overline{H}}^t\|_{C^1}\lessdot\epsilon^{1-2d}(e^{2c_{18}t}-e^{c_{18}t}),
\end{equation}
where $c_{18}\sim\mathcal{O}(1)$ is a constant depending on $E_2$. For $t\lessdot\ln\epsilon^{2d-1}$, we can see that $\|D\phi_{H'}^t-D\phi_{\overline{H}}^t\|_{C^1}\rightarrow0$ as $\epsilon\rightarrow0$. On the other side, Lemma \ref{NHIC time} gives a time restriction for the existence of NHIC $\overline{N}_{1,E_1,E_2}^{+}$. So we need
\[
\frac{2}{\lambda_2}\ln\epsilon^{-d}\lessdot\ln\epsilon^{2d-1} 
\]
if we take $E_1=2\epsilon^d$. Then $d<\frac{1}{3}$ is sufficient.\\

Therefore, from the invariant manifold theory of \cite{HPS} we can see that the separated property of spectrums in Lemma \ref{NHIC time} can not be destroyed as a result of  (\ref{peturbation estimation}), as long as $\epsilon$ is chosen sufficiently small. So we get the persistence of wNHIC $N_{1,\frac{\epsilon^d}{2},E_2}$ for system $H$.

\subsection{Location of Aubry sets}
\vspace{10pt}
$\newline$

From subsection \ref{Mather theory} we can see that there exists a conjugated Lagrangian system $L_H$ for any Tonelli Hamiltonian $H$. The autonomous system $\overline{H}=h+Z$ has the corresponding $\overline{L}$ which satisfies
\[
\overline{L}(x,v)\doteq\max_{p\in T^*_x\mathbb{T}^2}\{\langle v,p\rangle-\overline{H}(x,p)\},\quad v\in T_x\mathbb{T}^2.
\]
Besides, it's a diffeomorphism that $\mathcal{L}:T\mathbb{T}^2\rightarrow T^*\mathbb{T}^2$ via $(x,v)\rightarrow(x,p)$, where $v=\overline{H}_p(x,p)$. Then we can list all the $\overline{L}$ as follows:
\begin{equation}
\overline{L}(x,v)=l(v,p_i^*)-d_i^{*\sigma}Z(x_1),\quad\text{(1-resonance),}
\end{equation}
\begin{equation}
\overline{\tilde{L}}(X,V)=\frac{1}{2}\langle V^t-(0,\frac{\omega_{i,2}}{\mu\iota_m\delta}),D^2h^{-1}(p_i)(V-\begin{pmatrix}
0\\ \frac{\omega_{i,2}}{\mu\iota_m\delta}
\end{pmatrix})\rangle-\tilde{Z}_i(X_1),
\end{equation}
\hspace*{\fill}(homogenized system of transitional segment),
\begin{equation}
\overline{L}(X,V)=\frac{1}{2}V^2-Z_1(X_1)-Z_2(X_2)-\varepsilon Z_3(X),
\end{equation}
\hspace*{\fill}(homogenized system of 2-resonance),
\begin{equation}
\overline{\overline{L}}(X,V)=\frac{1}{2}V^2-Z_1(X_1)-Z_2(X_2),\quad\text{(uncoupled system of 2-resonance).}
\end{equation}
Recall that these autonomous Lagrangians are all Tonelli and defined in the domains where canonical coordinations are valid respectively. For a given $\vec{h}=(0,h_2)\in H_1(\mathbb{T}^2,\mathbb{R})$, the minimizing measure $\overline{\mu}_h$ exists and supp$(\overline{\mu}_h)$ lays on the cylinder $\overline{N}_2$ ($\overline{\overline{N}}_2$ for $\overline{\overline{L}}$). Actually, $\overline{\mu}_h$ is uniquely ergodic with the periodic trajectory $\gamma_h$ as its support. This is because the strictly positive definiteness of $h(p)$ w.r.t. $p$.\\

Besides, we have
\[
\beta_{\overline{L}}(h)=\int \overline{L}d\overline{\mu}_h.
\]
Since $\alpha_{\overline{L}}(c)$ is the convex conjugation of $\beta_{\overline{L}}(h)$, we can find $c_h=(0,c_{h,2})\in D^{+}\beta_{\overline{L}}(h)\in H^1(\mathbb{T}^2,\mathbb{R})$ and
\[
\alpha_{\overline{L}}(c_h)=-\int \overline{L}-c_hd\overline{\mu}_h=\overline{H}\big{|}_{supp(\overline{\mu}_h)},
\] 
i.e. $\overline{\mu}_h$ is the $c_h-$minimizing measure.\\

Besides, we can see that $\tilde{\mathcal{A}}_{\overline{L}}(c_h)=\widetilde{\mathcal{M}}_{\overline{L}}(c_h)=$supp$(\overline{\mu}_h)$. Then $\widetilde{N}_{\overline{L}}(c_h)=\tilde{A}_{\overline{L}}(c_h)$ from \cite{B}. Since $\tilde{\mathcal{N}}_{L}(c_h)$ is upper semicontinuous as a set-valued function of $L$ (see subsection \ref{Mather theory}), then $\forall \epsilon_0>0$ sufficiently small, there exists $\delta(\epsilon_0)$ such that for $\|L-\overline{L}\|_{C^2}\leq\delta$, $\widetilde{\mathcal{N}}_{L}(c_h)\subset\mathcal{B}(\widetilde{\mathcal{N}}_{\overline{L}}(c_h),\epsilon_0)$. However, in the previous paragraph we have proved the persistence of wNHIC for $L$ conjugated to $H=\overline{H}+R$. As the wNHIC is the unique invariant set in its small neighborhood (the normal hyperbolic property), this implies that $\tilde{\mathcal{A}}_{L}(c_h)$ is still contained in the weak invariant cylinder $N_1$.\\

For the 1-resonant case, we can see that 
\begin{equation}\label{1-resonant Lagrangian}
L=\overline{L}+R(x,v,t),
\end{equation}
where $\|R(x,v,t)\|_{C^2,\mathcal{B}}\lessdot d_m^{\sigma-r-6}$. Owing to our self-similar structure, we can take $m\geq M\gg1$ sufficiently large such that $\widetilde{\mathcal{N}}_{L}(c_h)$ is in the wNHIC $N_1$. Actually, this is a typical {\it a priori} unstable case. the readers can find an alternative proof from \cite{BKZ} and \cite{Ch}.\\

For the transitional segment from 1-resonance to 2-resonance, we can see that
\begin{equation}\label{transitional Lagrangian}
L=\overline{\tilde{L}}+R(X,V,S),
\end{equation}
where $\|R\|_{C^2,\mathcal{B}}\lessdot\frac{1}{\mathbb{K}}$. Recall that $\tilde{H}$ is homogenized and the norm $\|\cdot\|_{C^2,\mathcal{B}}$ depends just on $(X,V)$ variables, but not on $S$. However, we needn't control the value of $\partial_{S}R$ in the Legendre transformation from $H$ to $L$. So we can take $\mathbb{K}$ {\it a priori} large such that $\widetilde{\mathcal{N}}_{L}(c_h)$ is in the wNHIC $N_1$ due to above analysis.
\begin{The}
Given $g_2=(0,1)\in H_1(\mathbb{T}^2,\mathbb{Z})$, there exists a wNHIC $N_2\subset T^*\mathbb{T}^2\times\mathbb{S}^1$ corresponding to it for the two cases above: 1-resonance and transitional segment. We can find a wedge region $\mathbb{W}_{g_2}^c\subset H^1(\mathbb{T}^2,\mathbb{R})$ corresponding to $[h_1,h_2]\subset H_1(\mathbb{T}^2,\mathbb{R})$ with $h_i=(0, h_{i,2})=\nu_i g_2$, $i=1,2$. $\forall c_h=(0,c_{h,2})\in\mathbb{W}_{g_2}$ for $h\in[h_1,h_2]$, we have $\tilde{\mathcal{A}}_{L}(c_h)\subset\widetilde{\mathcal{N}}_{L}(c_h)\subset N_2$. Besides, from the upper semicontinuous of Ma\~{n}\'e set, we can see that $\widetilde{\mathcal{N}}_{L}(c'_h)\subset N_2$ for $c'_h=(c'_{h,1},c_{h,2})\in\mathring{\mathbb{W}}_{g_2}$
\end{The}

Now we consider the 2-resonant case of $\omega_{m,1/2}$. We also need to consider the location of Aubry sets for $N_1$ cylinder in this case, as there exists a transformation of resonant lines. First, the uncoupled Lagrangian $\overline{\overline{L}}$ has two routes $\overline{\overline{\Upsilon}}_{g_2}^h\doteq\{(0,h_2)\in H_1(\mathbb{T}^2,\mathbb{R})\big{|}h_2\in[0,h_2^{\max}]\}$ and $\overline{\overline{\Upsilon}}_{g_1}^h\doteq\{(h_1,0)\in H_1(\mathbb{T}^2,\mathbb{R})\big{|}h_1\in[0,h_1^{\max}]\}$, of which we can find minimizing measures $\overline{\overline{\mu}}_{(0,h_2)}$ and $\overline{\overline{\mu}}_{(h_1,0)}$ respectively. Besides, they are uniquely ergodic with the supports lay on cylinders $\overline{\overline{N}}_2$ and $\overline{\overline{N}}_1$. Accordingly, we can find routes $\overline{\overline{\Upsilon}}_{g_2}^c\doteq\{(0,c_2)\in H^1(\mathbb{T}^2,\mathbb{R})\big{|}c_2\in[0,c_2^{\max}]\}$ and $\overline{\overline{\Upsilon}}_{g_1}^c\doteq\{(c_1,0)\in H^1(\mathbb{T}^2,\mathbb{R})\big{|}c_1\in[0,c_1^{\max}]\}$, of which $\overline{\overline{\mu}}_{(0,h_2)}$ is $(0,c_2)-$minimizing and $\overline{\overline{\mu}}_{(h_1,0)}$ is $(c_1,0)-$minimizing. Similarly as above, $\widetilde{\mathcal{N}}_{\overline{\overline{L}}}(c)=\widetilde{\mathcal{A}}_{\overline{\overline{L}}}(c)=\widetilde{\mathcal{M}}_{\overline{\overline{L}}}(c)\subset\overline{\overline{N}}_i$ for $c\in\overline{\overline{\Upsilon}}_{g_i}^c$, $i=1,2$.\\

Second, $\overline{L}$ can be considered as an autonomous perturbation of $\overline{\overline{L}}$ by taking $\varepsilon\ll1$ {\it a priori} small. From {\bf C4} we can see that $\min_{c}\alpha_{\overline{L}}(c)=0=\alpha_{\overline{L}}(0)$. Still by the upper semicontinuous of Ma\~{n}\'e set, $\widetilde{\mathcal{N}}_{\overline{L}}(c)$ is contained in a small neighborhood of $\widetilde{\mathcal{N}}_{\overline{\overline{L}}}(c)$, of which the neighborhood radius depends only on $\varepsilon$. Then we have $\widetilde{\mathcal{N}}_{\overline{L}}(c)\subset\overline{N}_i$ for $c\in\overline{\overline{\Upsilon}}_{g_i}^c$, $i=1,2$.\\

Recall that $\widetilde{\mathcal{N}}_{\overline{L}}(0)=\widetilde{\mathcal{A}}_{\overline{L}}(0)=\widetilde{\mathcal{M}}_{\overline{L}}(0)=\{0\}\subset T^*\mathbb{T}^2$. So we have a 2-dimensional flat $\overline{\mathbb{F}}$ containing $0\in H^1(\mathbb{T}^2,\mathbb{R})$ in it and $\widetilde{\mathcal{N}}(c)=\{0\}$, $\forall c\in\mathring{\overline{\mathbb{F}}}$. So we can find $(0,c_2^{\min})$ and $(c_1^{\min},0)$ in $\partial\mathbb{F}$ such that 
\[
\widetilde{\mathcal{N}}((0,c_2^{\min}))=\overline{\{0\}\cup\{\gamma_{g_2}\}},
\]
and
\[
\widetilde{\mathcal{N}}((c_1^{\min},0))=\overline{\{0\}\cup\{\gamma_{g_1}\}}.
\]
These are the end points of routes $\overline{\Upsilon}_{g_2}^c$ and $\overline{\Upsilon}_{g_1}^c$ of which Ma\~{n}\'e set is contained in the cylinders $\overline{N}_1$ or $\overline{N}_2$ respectively (see figure \ref{fig19}).
\begin{figure}
\begin{center}
\includegraphics[width=6cm]{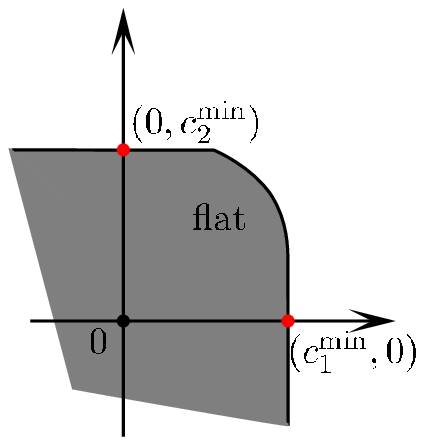}
\caption{ }
\label{fig19}
\end{center}
\end{figure}
As a further perturbation of $\overline{L}$, we have:
\begin{equation}\label{homogenized Lagrangian}
L=\overline{L}-R(X,V,S),
\end{equation}
where $\|R\|_{C^2,\mathcal{B}}\sim\mathcal{O}(\epsilon)$. Since we have proved that $N_2$ can be expanded to minus energy surfaces from Corollary \ref{wiggins time-dependent}, we can still find a route $\Upsilon_{g_2}^c=\{(0,c_2)\big{|}c_2\in[c_2^{*\min},c_2^{*\max}]\}$ of which the Ma\~{n}\'e set is contained in $N_2$. Besides, there is still a flat $\mathbb{F}^*$ containing $0\in H^1(\mathbb{T}^2,\mathbb{R})$ in it and $(0,c_2^{*\min})\in\partial\mathbb{F}^*$.\\

On the other side, we only proved the persistence of wNHIC $N_{1, E_1,E_2}$ with $E_1=2\epsilon^d$ in the previous paragraph. So we can find a lower bound $c_1^{**\min}$ and $\Upsilon_{g_1}^c=\{(c_1,0)\big{|}c_1\in[c_1^{**\min},c_1^{*\max}]\}$, of which $\widetilde{\mathcal{N}}(c)\subset N_{1, E_1,E_2}$. Notice that $(c_1^{**\min},0)$ doesn't lie on $\mathbb{F}^*$ (see Figure \ref{fig20}). \\

\begin{figure}
\begin{center}
\includegraphics[width=8cm]{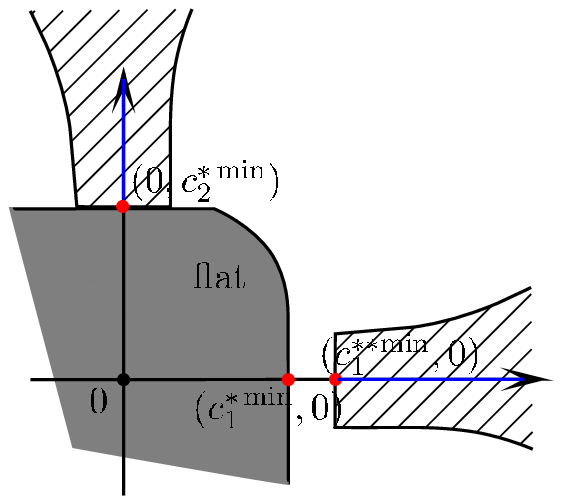}
\caption{ }
\label{fig20}
\end{center}
\end{figure}

We can find a $\overline{c}_1^{\min}$ corresponding to $c_1^{**\min}$ and $\alpha_{\overline{L}}((\overline{c}_1^{\min},0))=3\epsilon^d$. Without loss of generality, we assume ${\mu}_{(\overline{c}_1^{\min},0)}$ is the corresponding minimizing measure of $L$. Then we have $\alpha_{L}(\overline{c}_1^{\min},0)\geq 3\epsilon^d-\mathcal{O}(\epsilon)\gg \frac{5}{2}\epsilon^d\geq E_1$. So we can see that the minimizing measure $\mu_{(\overline{c}_1^{\min},0)}$ lays on the wNHIC $N_{1,E_1,E_2}$. In contrary, this supplies a upper bound of $\overline{c}_1^{\min}>c_1^{**\min}$.\\

Similarly as above, from the upper semicontinuous of $\widetilde{\mathcal{N}}_{L}(c)$ as a set-valued function of $c$, there exist two wedge set $\mathbb{W}_{g_1}^c$ and $\mathbb{W}_{g_2}^c$ satisfying
\[
\mathbb{W}_{g_1}^c\doteq\{(c_1,c_2^{w})\big{|}c_1\in[c_1^{**\min},c_1^{*\max}],\|c_2^w\|\leq\delta(c_1)\}
\]
and
\[
\mathbb{W}_{g_2}^c\doteq\{(c_1^{w},c_2)\big{|}c_2\in[0,c_2^{*\max}],\|c_1^w\|\leq\delta(c_2)\}.
\]
Here $\delta>0$ depending on $c_1$ and $c_2$ can be properly chosen, but we actually don't care the exact value of $\delta$.

\begin{The}\label{wedge}
For the case of 2-resonance, we can find 3-dimensional cylinder $N_i$ corresponding to $g_i\in H_1(\mathbb{T}^2,\mathbb{Z})$, $i=1,2$. There exists a wedge set $\mathbb{W}_{g_i}^c$ as is given above such that
\begin{itemize}
\item $\forall c\in\mathring{\mathbb{W}}_{g_i}$, $\tilde{\mathcal{A}}(c)\subset\widetilde{\mathcal{N}}(c)\subset N_{i}$.
\item the wedge set reaches to certain small neighborhood of flat $\mathbb{F}^*$ in the sense that
\[
\min_{c\in\mathbb{W}_{g_1}^c}\alpha(c)-\min_{c\in H^1(\mathbb{T}^2,\mathbb{R})}\alpha(c)\leq 3\epsilon^d,
\]
and
\[
\min_{c\in\mathbb{W}_{g_2}^c}\alpha(c)-\min_{c\in H^1(\mathbb{T}^2,\mathbb{R})}\alpha(c)=0.
\]
\end{itemize}
\end{The}

Therefore, we can connect different $\mathbb{W}_{g_2}^c$ sets into a long `channel' according to $\Gamma_{m,1}^{\omega}$. Recall that here a homogenization method is involved in the cases of 2-resonance and transition part from 1-resonance to 2-resonance. But this approach is just a special kind symplectic transformation. The following subsection ensures the validity of all these properties for system $H$ in the original coordinations.

\subsection{Symplectic invariance of Aubry sets}
$\newline$

Let $\omega=dp\wedge dq$ be the symplectic 2-form of $T^*M$. The diffeomorphism 
\[
\Psi:T^*M\rightarrow T^*M,\quad via (q,p)\rightarrow(Q,P)
\]
is call {\bf exact}, if $\Psi^*\omega-\omega$ is exact 2-form of $T^*M$.
\begin{The}{\bf (Bernard \cite{B2})}
For the exact symplectic diffeomorphism $\Psi$ and $H:T^*M\rightarrow\mathbb{R}$ Tonelli Hamiltonian, 
\[
\alpha_{H}(c)=\alpha_{\Psi^*H}(\Psi^*c),\quad \widetilde{\mathcal{M}}_{H}(c)=\Psi(\widetilde{\mathcal{M}}_{\Psi^*H}(\Psi^*c)),
\] 
\[
\widetilde{\mathcal{A}}_{H}(c)=\Psi(\widetilde{\mathcal{A}}_{\Psi^*H}(\Psi^*c)),\quad\widetilde{\mathcal{N}}_{H}(c)=\Psi(\widetilde{\mathcal{N}}_{\Psi^*H}(\Psi^*c)),
\]
where $\forall c\in H^1(M,\mathbb{R})$.
\end{The}
Since the time-1 mapping of Hamiltonian flow $\phi_H^t\big{|}_{t=1}$ can be isotopic to identity and $\phi_{H}^{*t}\big{|}_{t=1}\omega=\omega$, then $\phi_H^t\big{|}_{t=1}$ is exact symplectic. So we can use this theorem and get the symplectic invariance of these sets via different KAM iterations and homogenization.

\section{Annulus of incomplete intersection}\label{Annulus}
\vspace{10pt}
From the previous section we can see that $\mathbb{W}_{g_1}^c$ can reach the place $\epsilon^d-$approaching to $\arg\min_c\alpha_L(c)$. But it's unclear whether it can reach the margin of $\mathbb{F}^*$. So we have to find a route in $H^1(\mathbb{T}^2,\mathbb{R})$ to connect $\mathbb{W}_{g_1}^c$ and $\mathbb{W}_{g_2}^c$, along which the diffusion orbits can be constructed connecting the two wNHICs $N_1$ and $N_2$. In contrary, this demands that there must be an annulus region $\mathbb{A}$ around $\mathbb{F}^*$ of the thickness greater than $\epsilon^d$ such that $\forall c,c'\in\mathbb{A}$ are $c-$equivalent (see Figure \ref{fig21}). This is the central topic we'll discuss in this section.\\

We mention that the `incomplete intersection' here means that the stable manifold of the Aubry set `intersects' the unstable manifold non-trivially but possibly incomplete. In other words, for each class in this region, the Ma\~{n}\'e set does not cover the whole configuration space. On the other side, \cite{L} supplies us with a mechanism to construct $c-$equivalent in this region, which we will discuss in the next section.

\subsection{$\alpha-$flat of system $\overline{H}$}
$\newline$

For $\overline{H}$ system, we have $\alpha_{\overline{H}}(0)=0$ and there exists a 2-dimensional flat $\mathbb{F}$ such that $0\in\mathring{\mathbb{F}}$. Besides, we have the following
%From \cite{Ms} we know $\widetilde{\mathcal{N}}(c)=\tilde{\mathcal{A}}(c)=\widetilde{\mathcal{M}}(c)=\{(0,0)\in T^*\mathbb{T}^2\}$, $\forall c\in\mathring{\mathbb{F}}$. Besides, we can find sub-flat $\mathbb{E}_i$ such that $\forall c\in\mathbb{E}_i$, we have
%\[
%\widetilde{\mathcal{M}}(c)=\{(0,0)\},\quad\widetilde{\mathcal{N}}(c)=\tilde{\mathcal{A}}(c)=\{(0,0)\}\cup\{\gamma_{g_i}\},\quad i=1,2.
%\]
%Several useful properties are listed as following

\begin{Pro}\label{flat}
\begin{enumerate}
\item $\forall c\in\mathring{\mathbb{F}}$, $\widetilde{\mathcal{N}}(c)=\tilde{\mathcal{A}}(c)=\widetilde{\mathcal{M}}(c)=\{(0,0)\in T^*\mathbb{T}^2\}$.
\item There exist two sub-flat $\mathbb{E}_i\subset\partial\mathbb{F}$, such that $\forall c\in\mathbb{E}_i$
\[
\widetilde{\mathcal{M}}(c)=\{(0,0)\},\quad\widetilde{\mathcal{N}}(c)=\tilde{\mathcal{A}}(c)=\{(0,0)\}\cup\{\gamma_{g_i}\},\quad i=1,2.
\]
\item $\forall c\in\partial\mathbb{F}$, $\{(0,0)\}\subset\widetilde{\mathcal{M}}(c)$ and $\widetilde{\mathcal{N}}(c)\setminus\{(0,0)\}\neq\emptyset$.
\item $\forall c\in\partial\mathbb{F}$, if we can find more than one ergodic minimizing measure $\mu_c$ and $\mu'_c$, there exists $\gamma_c\subset\widetilde{\mathcal{N}}(c)$ connecting supp$\mu_c$ and supp$\mu'_c$. 
\end{enumerate}
\end{Pro}
\begin{proof}
1. A direct cite of \cite{Ms}.

2. We can see this point from the analysis of previous subsection.

3. See \cite{Z} for the proof.

4. A direct cite of \cite{Con}.
\end{proof}

In the universal covering space $\mathbb{R}^2$, we consider the elementary weak KAM solution $u_{(0,0)}^{\pm}$ corresponding to $0\in\mathbb{R}^2$, which is the projection of hyperbolic fixed point $(0,0)\in T^*\mathbb{T}^2$. Since $\varepsilon\ll1$ is sufficiently small, the graph part of $\overline{W}_{(0,0)}^{s,u}$ can cover a whole basic domain of $\mathbb{R}^2$, i.e. $\overline{W}_{(0,0)}^{s,u}=\{(x,dS_{(0,0)}^{s,u}(x))\big{|}x\in\Omega\}$. Here $\Omega$ is the maximal domain of which $\overline{W}_{(0,0)}^{s,u}$ is a graph. We can see that there exists a small constant $ \delta>0$ depending on $\varepsilon$ and $[-\pi-\delta,\pi+\delta]\times[-\pi-\delta,\pi+\delta]\subset\Omega$.

\begin{Lem}
$\forall x\in\Omega$, $(x,du_{(0,0)}^{\pm})=(x,dS_{(0,0)}^{s,u})$, i.e. $u_{(0,0)}^{\pm}$ is differentiable in $\Omega$. 
\end{Lem}
\begin{proof}
We just prove the Lemma for $u_{(0,0)}^{-}$ and the case $u_{(0,0)}^{+}$ can be proved in the same way. Because $u_{(0,0)}^{-}$ is linear semi-concave (SCL) from \cite{Ca}, so it's differentiable for a.e. $x\in\Omega$. If $x\in\Omega$ is a differentiable point, $(x,du_{(0,0)}^{-}(x))$ will decide a backward semistatic trajectory $\gamma_{(0,0)}^{-}(t)$ which trends to $(0,0)$ as $t\rightarrow-\infty$. Then $(\gamma_{(0,0)}^{-}(t),\dot{\gamma}_{(0,0)}^{-}(t))\subset\overline{W}_{(0,0)}^u$ for $t\in(-\infty,0]$, with $\gamma_{(0,0)}^{-}(0)=x$. Notice that there may be some $t_0\in(-\infty,0]$ existing and $y=\gamma_{(0,0)}^{-}(t_0)\in\partial\Omega$ (see Figure \ref{fig22}). But there exists no focus locus for $t\in(-\infty,0)$ in the set $\{(\gamma_{(0,0)}^{-}(t),\dot{\gamma}_{(0,0)}^{-}(t))\}$ (see Theorem 6.3.6 of \cite{Ca}). So $u_{(0,0)}^{-}$ is $C^r$ differentiable at the set $\{(\gamma_{(0,0)}^{-}(t),\dot{\gamma}_{(0,0)}^{-}(t))\big{|}t\in(-\infty,0)\}$. But on the other side, we can see that $D^2u_{(0,0)}^{-}$ doesn't exist because the invalidity of graph property. This contradiction means that $y$ doesn't exist and $\{(\gamma_{(0,0)}^{-}(t),\dot{\gamma}_{(0,0)}^{-}(t))\}$ lays on the graph part of $\overline{W}_{(0,0)}^u$ for $t\in(-\infty,0]$.\\

If there exists any $x$  of which $u_{(0,0)}^{-}$ is not differentiable, from \cite{Ca}, for arbitrary reachable gradient $p\in D^*u_{(0,0)}^{-}(x)$ we can find a sequence of differentiable points $x_n\rightarrow x$ and $(x_n, du_{(0,0)}^{-}(x_n))\rightarrow(x,p)$, where $\{x_n\}\subset\Omega$ for $n\in\mathbb{N}$. Then $(x,p)$ lies on the graph part of $\overline{W}_{(0,0)}^u$. But from the invariant manifold theorem $\overline{W}_{(0,0)}^u$ is $C^r-$differentiable, so $D^*u_{(0,0)}^{-}(x)$ is a single-point set and $u_{(0,0)}^{-}$ is differentiable at this point. 
\end{proof}
\begin{figure}
\begin{center}
\includegraphics[width=8cm]{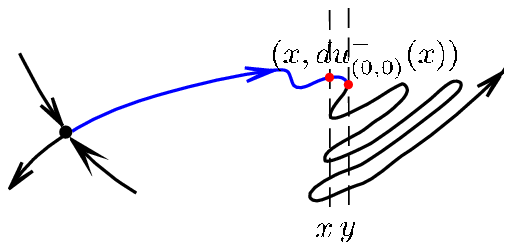}
\caption{ }
\label{fig22}
\end{center}
\end{figure}
Based on this Lemma, we can prove the following
\begin{The}
$\forall c\in\partial\mathbb{F}$, $\widetilde{\mathcal{M}}(c)=\{(0,0)\in T^*\mathbb{T}^2\}$, i.e. there doesn't exist new minimizing measures except the one supported on the hyperbolic fixed point.
\end{The}
\begin{proof}
If $\exists c\in\partial\mathbb{F}$ such that $\widetilde{\mathcal{M}}(c)\setminus\{(0,0)\}\neq\emptyset$, there must be an ergodic measure $\mu_c$ with $\rho(\mu_c)\neq0$ and supp$(\mu_c)\cap\{(0,0)\}\neq\emptyset$. From (4) of Proposition \ref{flat}, we can find a c-semistatic orbit $\gamma_{c}$ connecting supp$(\mu_c)$ and $\{(0,0)\}$. Then $\forall t_0\in(-\infty,+\infty)$, $\gamma_c:(-\infty,t_0]$ is of course a backward semistatic orbits. Then $\{(\gamma_c(t),\dot{\gamma}_c(t))\}$ lies on the graph part of $\overline{W}_{(0,0)}^u$ for all $t\in(-\infty,t_0]$ from the above Lemma. On the other side, the hyperbolic fixed point is the only invariant set of $\overline{W}_{(0,0)}^u\big{|}_{\Omega}$ and $t_0$ is arbitrarily chosen, so $\mu_c$ only can be the fixed point itself.
\end{proof}
Based on this Theorem, we can use the Theorem 3.2 of \cite{Ch} and see that $\mathbb{F}$ is actually a polygon with finite edges. Besides, $\mathbb{F}$ is of central symmetry as $\overline{H}$ is a mechanical system. Then from (3) of Proposition \ref{flat} we can see that $\widetilde{\mathcal{N}}(c)=\tilde{\mathcal{A}}(c)$ and there must be homoclinic orbits come out for $c\in\partial\mathbb{F}$. Without loss of generality, we can assume the edges of $\mathbb{F}$ as $\mathbb{E}_i^{\pm}$ and $\forall c\in\mathbb{E}_i^{\pm}$, there will be homoclinic orbits coming out with homologgy class just being $\pm g_i$, $i=1,2\cdots,m$. \\

Actually, we can prove in the following that the only possible homology classes of homoclinic orbits are $(0,1)$, $(1,0)$, $(1,1)$ and $(-1,1)$. Since the existence of $(0,1)-$ and $(1,0)-$homoclinic orbit has been proved in the previous subsection, we only need to prove the existence of $(1,1)-$type and $(-1,1)-$type can be sealed with in the same way. \\

Suppose $\vec{n}=(n_1,n_2)$ and $\gamma_n$ is a $n-$type homoclinic orbit. Accordingly, we have $\mathbb{E}_n$ as an edge of $\mathbb{F}$ and $\gamma_n\subset\mathcal{N}(c)$, $\forall c\in\mathbb{E}_n$. Then $\{(\gamma_n(t),\dot{\gamma}_n(t))\}$ is contained in the graph part of $\overline{W}_{(0,0)}^u$, for $t\in(-\infty,t_0]$. By taking $\varepsilon\ll1$ sufficiently small, $\Omega$ can be chosen large enough such that $\Omega\cap\mathcal{B}((2\pi,2\pi),\delta)$, $\Omega\cap\mathcal{B}((2\pi,0),\delta)$ and $\Omega\cap\mathcal{B}((0,2\pi),\delta)$ are all nonempty. Here $\Omega$ is the definition domain of $\overline{W}_{(0,0)}^u$ and $\delta=\delta(\lambda)$ is the radius of the  neighborhood of the fixed point, in which the normal form (\ref{linear tonelli}) is valid (see figure \ref{fig23}). There exists $[t_1,t_2]$ during which $\gamma_n(t)$ is contained in $\mathcal{B}((2\pi,2\pi),\delta)$. Since (\ref{linear tonelli}) is uncoupled, the corresponding equation (\ref{linear tonelli ODE}) is $C^1$ conjugated to the linear ODE

\begin{equation}
\dot{X}_i=\lambda_i Y_i,\quad\dot{Y}_i=\lambda_i X_i,\quad i=1,2,
\end{equation}
in the domain $\mathcal{B}((2\pi,2\pi),\delta)$ from the Hartman Theorem of \cite{Har}. Under this coordination we can see that
\[
\overline{H}(X,Y)=\lambda_1\frac{(Y_1^2-X_1^2)}{2}+\lambda_2\frac{(Y_2^2-X_2^2)}{2},
\]
and 
\[
\overline{L}(X,V)=\frac{(V_1^2+\lambda_1^2X_1^2)}{2\lambda_1}+\frac{(V_2^2+\lambda_2^2X_2^2)}{2\lambda_2},
\]
where $V_i=\lambda_i Y_i$, $i=1,2$. This smoothness is enough for us to calculate the action value of $\gamma_n$. As the coordinate transformation doesn't change the energy, $\{\gamma_n(t)\big{|}_{t\in[t_1,t_2]}\}$ lies on the energy surface $\{\overline{H}=0\}$. Besides, $\overline{H}_i=\lambda_i\frac{(Y_i^2-X_i^2)}{2}$ is a first integral in this domain, $i=1,2$. So we can involve a parameter $e$ to simplify our calculation, where $\overline{H}_1\big{|}_{(\gamma_n(t),\dot{\gamma}_n(t))}=e$ and $\overline{H}_2\big{|}_{(\gamma_n(t),\dot{\gamma}_n(t))}=-e$, $t\in[t_1,t_2]$.\\
\begin{figure}
\begin{center}
\includegraphics[width=8cm]{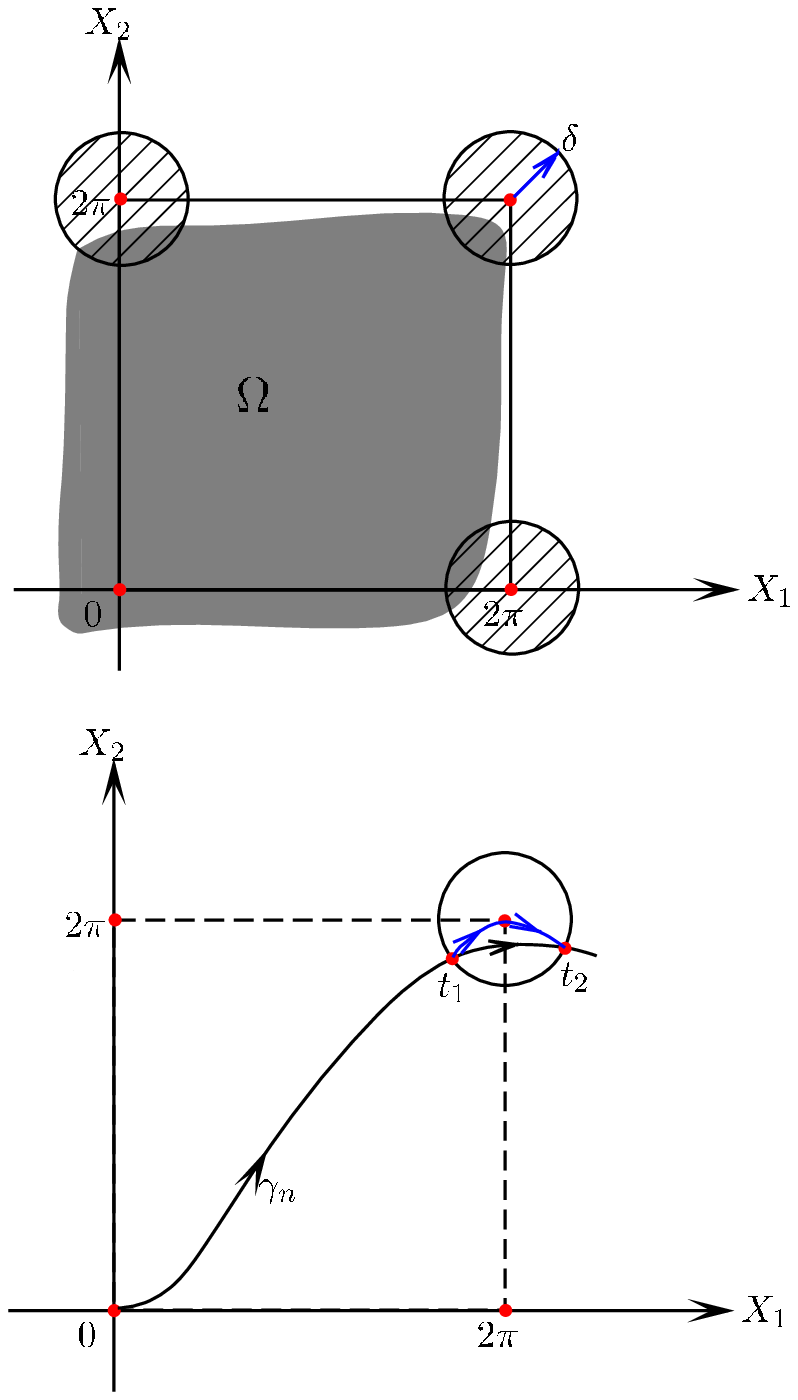}
\caption{ }
\label{fig23}
\end{center}
\end{figure}

Without loss of generality, we can assume that $\gamma_n(t_i)=X^i$, where $X^i=(X_1^i,X_2^i)\in\partial\mathcal{B}((2\pi,2\pi),\delta)$ with $i=1,2$. By a tedious but simple computation, we can get the flow via
\begin{equation}
\left\{
\begin{array}{ccccccc}
X_1=a_1e^{\lambda_1t}+a_2e^{-\lambda_1t},\\
Y_1=a_1e^{\lambda_1t}-a_2e^{-\lambda_1t},
\end{array}
\right.
\end{equation}
\begin{equation}
\left\{
\begin{array}{ccccccc}
X_2=b_1e^{\lambda_1t}+b_2e^{-\lambda_1t},\\
Y_2=b_1e^{\lambda_1t}-b_2e^{-\lambda_1t},
\end{array}
\right.
\end{equation}
where 
\[
a_1=\frac{X^2_1-e^{-\lambda_1T}X^1_1}{e^{\lambda_1T}-e^{-\lambda_1T}},\quad a_2=\frac{e^{\lambda_1T}X^1_1-X^2_1}{e^{\lambda_1T}-e^{-\lambda_1T}},
\]
and
\[
b_1=\frac{X^2_2-e^{-\lambda_2T}X^1_2}{e^{\lambda_2T}-e^{-\lambda_2T}},\quad b_2=\frac{e^{\lambda_2T}X^1_2-X^2_2}{e^{\lambda_2T}-e^{-\lambda_2T}},
\]
with $T=t_2-t_1$. We then get the action as
\begin{equation}
A_{\overline{L}}(\gamma_n)\big{|}_{[t_1,t_2]}=\frac{{X^1_1}^2+{X^1_2}^2}{2}\frac{\sinh2\lambda_1T}{\cosh2\lambda_1T-1}+\frac{{X^2_1}^2+{X^2_2}^2}{2}\frac{\sinh2\lambda_2T}{\cosh2\lambda_2T-1}.
\end{equation}

Now we connect $X^i$ and $(2\pi,2\pi)$ with trajectories lie on $\overline{W}_{(2\pi,2\pi)}^{u,s}$. Also we can calculate the formula as
\begin{equation}
I:\quad\left\{
\begin{array}{cccccccc}
X_i(t)=X^1_ie^{-\lambda_it},\\
Y_i(t)=-X^1_ie^{-\lambda_it},
\end{array}
\right.
\end{equation}
where $t\in[0,+\infty)$ and $i=1,2$ and
\begin{equation}
II:\quad\left\{
\begin{array}{cccccccc}
X_i(t)=X^2_ie^{\lambda_it},\\
Y_i(t)=X^2_ie^{\lambda_it},
\end{array}
\right.
\end{equation}
where $t\in(-\infty,0]$ and $i=1,2$. We then get the action 
\begin{equation}
A_{\overline{L}}(I)+A_{\overline{L}}(II)=\frac{{X^1_1}^2+{X^1_2}^2}{2}+\frac{{X^2_1}^2+{X^2_2}^2}{2}.
\end{equation}
Comparing the actions and we get $A_{\overline{L}}(I)+A_{\overline{L}}(II)<A_{\overline{L}}(\gamma_n)\big{|}_{[t_1,t_2]}$. Therefore $\gamma_n$ will break into two segment when it passes $\mathcal{B}((2\pi,2\pi),\delta)$, as it's semi-static. Then $\gamma_n$ can be decomposed into $\gamma_{(1,1)}$ and $\gamma_{(n_1-1,n_2-1)}$. In the same way we can decompose $\gamma_{(n_1-1,n_2-1)}$. So we get that all the minimizing homoclinic orbits are just of these homology classes:
\[
(1,0),\;(1,1),\;(0,1),\;(-1,1),\;(-1,0),\;(-1,-1),\;(0,-1),\;(1,-1).
\]

\begin{figure}
\begin{center}
\includegraphics[width=8cm]{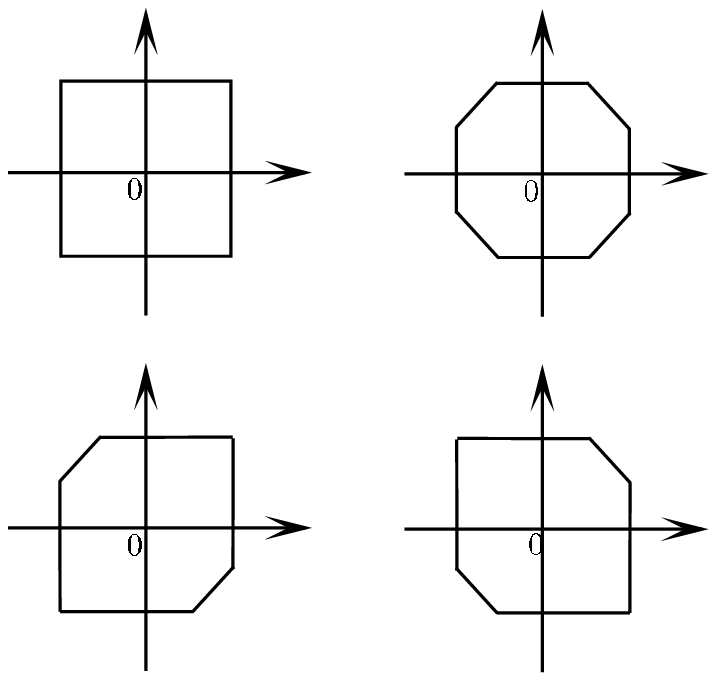}
\caption{ }
\label{fig24}
\end{center}
\end{figure}

\begin{The}
$\partial\mathbb{F}$ only could be rectangle, hexagon and octagon (see figure \ref{fig24}).
\end{The}
\begin{proof}
1. As a convex set, the number of expose points of $\mathbb{F}$ is not more than the number of homology classes of minimizing homoclinic orbits.

2. The inner of sub-flat $\mathbb{E}$ will share the same homoclinic orbits $\gamma_g$, and $\langle g,c-c'\rangle=0$, $\forall c,c'\in\mathring{\mathbb{E}}$.

3. The homology class of minimizing homoclinic orbit is irreducible, i.e. $l\cdot(n_1,n_2)-$type minimizing homoclinic orbit is a conjunction of $(n_1,n_2)-$type minimizing homoclinic orbits.
\end{proof}

\begin{figure}
\begin{center}
\includegraphics[width=7cm]{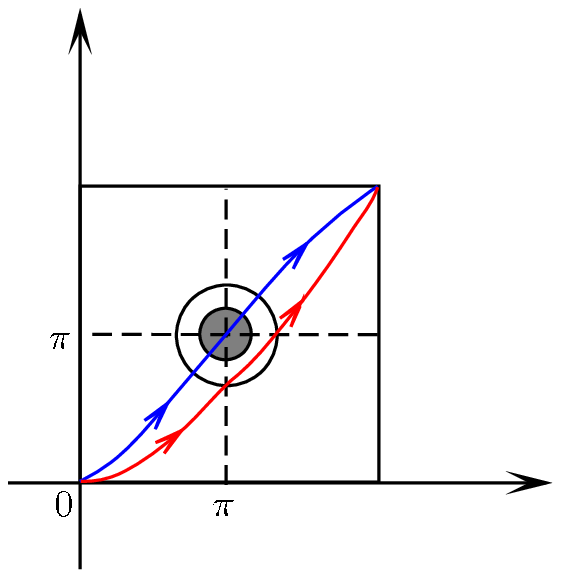}
\caption{ }
\label{fig25}
\end{center}
\end{figure}

Since we have proved the uniqueness of $(\pm1,0)-$ and $(0,\pm1)-$ type minimizing homoclinic orbits, we just need to reduce the number of minimizing homoclinic orbits of other classes. The Melnikov method can be used and we raise a new restriction\\

{\bf U7:}  The Melnikov function of $\overline{H}$ has a unique critical point in $\mathcal{B}((\pi,\pi),\delta)$. Here $\delta=\delta(\epsilon)>0$ is a small constant.\\

 Notice that the Melnikov function of $\overline{H}$ can be of $(1,1)-$type and $(-1,1)-$type, which is written by $M_{\overline{H},(1,1)}$ and $M_{\overline{H},(-1,1)}$. Once {\bf U7} is satisfied, then $(1,1)-$type and $(-1,1)-$type homoclinic orbits is unique in $\mathcal{B}((\pi,\pi),\delta)$ (see figure \ref{fig25}). Since $\overline{H}$ is a mechanical system, the homoclinic orbit of other type is also unique. Actually, we can satisfy {\bf U7} by restricting $Z_3$ of a certain form in $\mathcal{B}((\pi,\pi),\delta)$, since the Melnikov function can be considered as a continuous linear functional of potential $Z_3(X)$:

\begin{eqnarray}
M_{\overline{H},(1,1)}(X)=\lim_{T\rightarrow+\infty}-\int_{-T}^{T}Z_3(\gamma_{(1,1)}(t))dt,\\
M_{\overline{H},(-1,1)}(X)=\lim_{T\rightarrow+\infty}-\int_{-T}^{T}Z_3(\gamma_{(-1,1)}(t))dt,
\end{eqnarray}
where $\gamma_{(1,1)}(0)=X=\gamma_{(-1,1)}$ is the homoclinic orbit of $\overline{\overline{H}}$ of certain homology class, $X\in\mathcal{B}((\pi,\pi),\delta)$.

\begin{Pro}\label{Melnikov}
\begin{itemize}
\item $\nabla_{\overline{\overline{H}}}M_{\overline{H},(1,1)}(X)=0$, $\nabla_{\overline{\overline{H}}}M_{\overline{H},(-1,1)}(X)=0$.
\item If for some $X_0\in\mathcal{B}((\pi,\pi),\delta)$, we have 
\[
\nabla_{\overline{\overline{H}}_1}M_{\overline{H},(1,1)}(X_0)=0=\nabla_{\overline{\overline{H}}_2}M_{\overline{H},(1,1)}(X_0),
\]
and the rank of the matrix $(\nabla_{\overline{\overline{H}}_i}\nabla_{\overline{\overline{H}}_j}M_{\overline{H},(1,1)}(X_0))_{i,j=1,2}$ equals one. Then for sufficiently small $\varepsilon$, there exists a transversal homoclinic orbit of $\overline{H}$ passing from the $\mathcal{O}(\varepsilon)$ neighborhood of $X_0$. For $M_{\overline{H},(-1,1)}(X)$ we have the same conclusion.
\end{itemize}
\end{Pro}
\begin{proof}
See Appendix for the proof. It's a direct cite of \cite{Tr2}.
\end{proof}

Once {\bf U7} is satisfied, we have

\begin{The}
$\forall c\in\partial\mathbb{F}$, we have ${\mathcal{N}}(c)\subsetneqq\mathbb{T}^2$.
\end{The}
\vspace{10pt}

\subsection{Thickness of Annulus}
$\newline$
\begin{figure}
\begin{center}
\includegraphics[width=6cm]{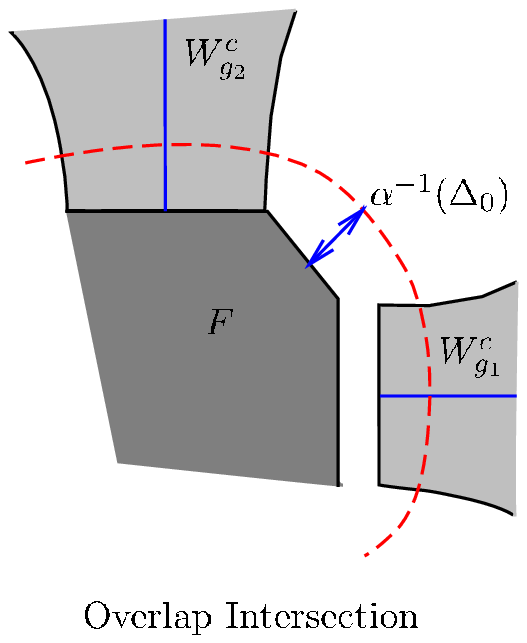}
\caption{ }
\label{fig21}
\end{center}
\end{figure}

From the previous Theorem and the upper semicontinuity of Ma\~{n}\'e set, there exists $\Delta_0>0$ such that $\forall c\in\{\alpha_{\overline{H}}(c\leq\Delta_0)\}$, we have $\widetilde{\mathcal{N}}(c)\subsetneqq\mathbb{T}^2$. Here $\Delta_0$ is a constant depending on $\varepsilon$.\\

On the other side, for $\Delta\in(0,\Delta_0)$ and $c\in\alpha_{\overline{H}}^{-1}(\Delta)$, All the $c-$minimizing measures have the same rotational direction because of the graph property of Mather set \cite{Car}. So we can find a loop section $\Gamma_c\subset\mathbb{T}^2$ such that all the semi-static orbits intersect it transversally. As is known to us that $\mathcal{N}_{\overline{H}}(c)\subsetneqq\mathbb{T}^2$, we can find finitely many open sets $\{U_i\}_{i=1}^n$ covering $\Gamma_c\cap\mathcal{N}_{\overline{H}}(c)$. $\{U_i\}_{i=1}^n$ is diffeomorphic to a list of open internals $\{(a_i,b_i)\subset[0,1]\}_{i=1}^n$ and they are disjoint with each other.\\

Once again we use the upper-semicontinuity of Ma\~{n}\'e set, for $m\geq M\gg1$ sufficiently large, system $H$ satisfies the following 
\begin{The}
$\exists \epsilon_0>0$ and $\Delta_0>0$ depending on $\varepsilon$, such that $\forall \Delta<\Delta_0$, $\epsilon\in(0,\epsilon_0)$ and $c\in\alpha^{-1}(\Delta)$, $\mathcal{N}_H(c)$ intersects $\Gamma_c\times\{S\equiv0\}$ transversally. Besides, we still have $\mathcal{N}_H(c)\cap\Gamma_c\times\{S\equiv0\}\subset\bigcup_{i=1}^nU_i$. Here $\{U_i\}\subset\Gamma_c\times\{S\equiv0\}$ are 1-dimensional open internals disjoint with each other.
\end{The}

On the other side, we can see that the wedge set $\mathbb{W}_{g_1}^c$ can reach the place of $\Delta=3\epsilon^d$ from Theorem \ref{wedge} (see Figure \ref{fig21}). So the following property of $H$ is valid:

\begin{Cor}{\bf (Overlap Property)}\label{overlap}
For $g_1$ and $g_2$, $\exists\epsilon_0>0$ sufficiently small such that $\mathbb{W}_{g_1}^c$ and $\mathbb{W}_{g_2}^c$ could intersect the annulus region $\mathbb{A}\doteq\{c\big{|}\alpha_H(c)\in[0,\Delta_0]\}$, as long as $\epsilon\leq\epsilon_0$.
\end{Cor}
\vspace{10pt}
\begin{Rem}{\bf (Robustness of system $\overline{H}$)}
Recall that our system is actually of a form (\ref{model system}), which can be considered as a $\mathcal{O}(\frac{1}{l^m})$ perturbation of $\overline{H}$. But (\ref{model system}) is also an autonomous system, so all these uniform properties for $\overline{H}$ can be preserved for it, as long as we take $m\geq M\gg1$ {\it a posterior} large.
\end{Rem}

\section{Local connecting orbits and generic diffusion mechanism}
\vspace{10pt}

To construct orbits connecting some Aubry set to another Aubry set nearby, we introduce two types of modified Tonelli Lagrangian, i.e. the time-step and the space-step Lagrangian. The former one was firstly developed in \cite{B}, \cite{CY1} and \cite{CY2}, with the earliest idea given by J. Mather in \cite{Mat3}. The latter one was firstly developed in \cite{L}.  Then C-Q. Cheng made a further elaboration and generalization of it in \cite{Ch}, for dealing with {\it a priori} stable Arnold Diffusion problem. Since our construction in this section have a great similarity with this case of \cite{Ch}, we choose it as the main reference for this section.\\

Actually, we can ascribe the local connecting orbits as two different mechanism: Arnold's and Mather's. The former one is essentially heteroclinic orbit, which is known as $h-$equivalent orbit from the variational viewpoint. The latter one is known as $c-$equivalent orbit, which is constructed with the topological conditions provided by Aubry sets. We mainly use the Arnold mechanism to construct the diffusion orbits along the wNHICs. The Mather mechanism is mainly used to solve the difficulty of incomplete intersection annulus.
\vspace{10pt}

\subsection{Modified Lagrangian: time-step case}

\begin{defn}{\bf (Time-step Lagrangian)}
We call a Tonelli Lagrangian $L:T\mathbb{T}^n\times\mathbb{R}\rightarrow\mathbb{R}$ {\bf time-step Lagrangian}, if we can find $L^{-}$ and $L^{+}:T\mathbb{T}^n\times\mathbb{S}^1\rightarrow\mathbb{R}$ such that 
\[
L(\cdot,t)=L^{-}(\cdot,t),\quad\forall t\in(-\infty, 0]
\]
and
\[
L(\cdot,t)=L^{+}(\cdot,t),\quad\forall t\in[1,+\infty),
\]
i.e. $L(\cdot,t)$ is periodic in $(-\infty, 0]\cup[1,+\infty)$.
\end{defn}
%\begin{Rem}
%Notice that we needn't $L$ being periodic for $t\in\mathbb{R}$. Often by modifying a periodic Largangian in the place $t\in[0,1]$ we can get a Lagrangian of this type, which is crucial in construct locally connecting orbits.
%\end{Rem}
\begin{defn}
A curve $\gamma:\mathbb{R}\rightarrow M$ is called {\bf minimal} if
\[
\int_{\tau}^{\tau'}L(\gamma(t),\dot{\gamma}(t),t)dt\leq\int_{\tau}^{\tau'}L(\zeta(t),\dot{\zeta}(t),t)dt
\]
holds for $\tau<\tau'$ and every absolutely continuous curve $\zeta: [\tau,\tau']\rightarrow M$ with $\zeta(\tau)=\gamma(\tau)$ and $\zeta(\tau')=\gamma(\tau')$. We denote the set of minimal curves for time-step Lagrangian $L$ by $\mathscr{G}(L)$ and $\tilde{\mathcal{G}}(L)\doteq\{(\gamma(t),\dot{\gamma}(t),t)\in TM\times\mathbb{R}\big{|}\gamma\in\mathscr{G}(L)\}$. Then we have $\mathcal{G}(L)=\pi\tilde{\mathcal{G}}(L)$ where $\pi:TM\times\mathbb{R}\rightarrow M\times\mathbb{R}$ is the standard projection.
\end{defn}

\begin{The}\cite{CY1, CY2, Ch}
The set-valued map $L\rightarrow\tilde{\mathscr{G}}(L)$ is upper semicontinuous. Consequently, the map $L\rightarrow\mathscr{G}(L)$ is also upper semicontinuous.
\end{The}
\begin{proof}
Here we only give a sketch of the proof and more details can be found in \cite{Ch}. Let $\{L_i\}\subset C^r(TM\times\mathbb{R},\mathbb{R})$ be a sequence converging to $L$ under the norm $\|\cdot\|_{C^2,TM}$, and $\gamma_i\in\mathscr{G}(L_i)$ be a sequence of minimal curves. These two modified Lagrangian are actually defined in a proper universal covering space of $\mathbb{T}^{n+1}$. We can see that $\|\dot{\gamma}_i(t)\|\leq K$, $\forall t\in\mathbb{R}$ with $K$ is a constant depending on $\|\partial_{vv}L_i\|$. Then the set $\{\gamma_i\}$ is compact in the $C^1(\mathbb{R},M)-$topology. Let $\gamma$ be one accumulated point of $\gamma_i$, we can see that $\gamma$ is $L-$minimal, and this proves the upper-semicontinuity of $\tilde{\mathscr{G}}(L)$.
\end{proof}
In application, the set $\mathscr{G}(L)$ seems too big for the construction of connecting orbits. For time-step Lagrangian, we can introduce the following set of pseudo connecting orbits, which is written by $\mathscr{C}(L)$. Let $\alpha^{\pm}$ be the minimal average action of $L^{\pm}$. For $m_0,m_1\in M$ and $T_0,T_1>0$, we define
\[
h_L^{T_0,T_1}(m_0,m_1)=\inf_{\substack{\gamma(-T_0)=m_0\\
\gamma(T_1)=m_1}}\int_{-T_0}^{T_1}L(d\gamma(t),t)dt+T_0\alpha^{-}+T_1\alpha^{+}
\]
and
\[
h_L^{\infty}(m_0,m_1)=\liminf_{T_0,T_1\rightarrow\infty}h_L^{T_0,T_1}(m_0,m_1).
\]
Let $\{T_0^i\}_{i\in\mathbb{N}}$ and $\{T_1^i\}_{i\in\mathbb{N}}$ be the sequence of positive numbers such that $T_j^i\rightarrow\infty$ $(j=0,1)$ as $i\rightarrow\infty$ and satisfies
\[
\lim_{i\rightarrow\infty}h_L^{T_0^i,T_1^i}(m_0,m_1)=h_L^{\infty}(m_0,m_1).
\]
Accordingly, we can find $\gamma_i(t,m_0,m_1):[-T_0^i,T_1^i]\rightarrow M$ being the minimizer connecting $m_0$ and $m_1$ and
\[
h_L^{T_0^i,T_1^i}(m_0,m_1)=\int_{-T_0^i}^{T_1^i}L(d\gamma_i(t),t)dt+T_0^i\alpha^{-}+T_1^i\alpha^{+}.
\]
Then the following Lemma holds:
\begin{Lem}
The set $\{\gamma_i\}$ is pre-compact in $C^1(\mathbb{R},M)$. Let $\gamma:\mathbb{R}\rightarrow M$ be an accumulation point of $\{\gamma_i\}$, then $\forall s,\tau\geq1$
\begin{eqnarray}\label{pseudo curve}
A_L(\gamma)\big{|}_{[-s,\tau]}=\inf_{\substack{s_1-s\in\mathbb{Z},\tau_1-\tau\in\mathbb{Z}\\
s_1,\tau_1\geq1\\
\gamma^*(-s_1)=\gamma(-s)\\
\gamma^*(\tau_1)=\gamma(\tau)}}\int_{-s_1}^{\tau_1}L(d\gamma^*(t),t)dt&+&(s_1-s)\alpha^{-} \nonumber\\
 &+&(\tau_1-\tau)\alpha^{+}.
\end{eqnarray}
\end{Lem}
\begin{proof}
The pre-compactness of $\{\gamma_i\}$ can be proved in the same way with that of $\mathscr{G}(L)$. As for the formula \ref{pseudo curve}, we can use the proof by contradiction, with an approach of comparing the action. We omit the proof since you can find it in \cite{Ch}.
\end{proof}
We define $\mathscr{C}(L)$ by the set $\{\gamma\in\mathscr{G}(L)\big{|}(\ref{pseudo curve})\; \text{holds for}\; \gamma\}$. Clearly, $\forall \gamma\in\mathscr{C}(L)$, the orbits $(\gamma(t),\dot{\gamma}(t),t)$ approaches to $\tilde{\mathcal{A}}(L^{-})$ as $t\rightarrow-\infty$ and approaches to $\tilde{\mathcal{A}}(L^{+})$ as $t\rightarrow+\infty$.
That's why we call it pseudo connecting curve. We also have
\[
\widetilde{\mathcal{C}}(L)=\bigcup_{\gamma\in\mathscr{C}(L)}(\gamma(t),\dot{\gamma}(t),t),\quad\mathcal{C}(L)=\bigcup_{\gamma\in\mathscr{C}(L)}(\gamma(t),t).
\]
If $L$ is a periodic Tonelli Lagrangian, then $\tilde{\mathcal{C}}(L)=\widetilde{\mathcal{N}}(L)$ and $\mathcal{C}(L)=\mathcal{N}(L)$.
\begin{The}\cite{CY1,CY2,Ch}
The map $L\rightarrow\mathscr{C}(L)$ is upper semicontinuous. As the special case for periodic Lagrangian $L$, $c\rightarrow\mathcal{N}(L)$ as well as the map $c\rightarrow\widetilde{\mathcal{N}}(L)$ is upper semicontinuous.
\end{The}
\begin{Cor}
For periodic Tonelli Lagrangian $L$ and $c\in H^1(M,\mathbb{R})$, the set-valued function $c\rightarrow\mathcal{N}(c)$ is upper semicontinuous.
\end{Cor}
\vspace{10pt}

\subsection{Modified Lagrangian: space-step case}

\begin{defn}{\bf (Space-step Lagrangian)}\label{li xia}
In the covering space $\overline{M}\doteq\mathbb{R}\times\mathbb{T}^{n-1}$ of $\mathbb{T}^n$, we call a Tonelli Lagrangian $L:T\overline{M}\times\mathbb{S}^1\rightarrow\mathbb{R}$ {\bf space-step}, if we can find $L^{-}$ and $L^{+}:T\mathbb{T}^n\times\mathbb{S}^1\rightarrow\mathbb{R}$ such that
\[
L(x_1,\cdot)=L^{-}(x_1,\cdot),\quad\forall x_1\in(-\infty, 0]
\]
and
\[
L(x_1,\cdot)=L^{+}(x_1,\cdot),\quad\forall x_1\in[1,+\infty),
\] 
where $(x_1,x_2,\cdots,x_n,t)\in\overline{M}\times\mathbb{S}^1$. Besides, the following conditions should be satisfied at the same time:
\begin{itemize}
\item Let $\mu^{\pm}$ be the $0-$minimizing measure of $L^{\pm}$ respectively. Then $\pi_1\rho(\mu^{\pm})>0$ for each minimizing measure $\mu^{\pm}$.
\item $\alpha^{-}=\alpha^{+}$ as the minimal average action of $L^{\pm}$. Without loss of generality, we assume it equals $0$.
\item $\|L^{+}-L^{-}\|_{\{(x,v)\in TM\big{|}\|v\|\leq D\}}\leq\frac{1}{2}\min_{h_1=0}\{\beta_{L^{-}}(h),\beta_{L^{+}}(h)\}$.
\end{itemize}
\end{defn}
\begin{Rem}\label{uniform}
A time-periodic Lagrangian $L(x,\dot{x},t)$ can be considered as an autonomous Lagrangian $L(\theta,\dot{\theta})$, where $(x,\dot{x},t)\in T\mathbb{T}^n\times{S}^1$ and $\theta=(t,x)$. As we know, $\mathbb{T}^n\times\mathbb{S}^1$ is diffeomorphic to $\mathbb{T}^{n+1}$. Then a time-step Lagrangian $L$ can be considered as a space-step one with $\theta_1=t$ taken in the universal covering space $\mathbb{R}$. So we just need to consider the autonomous Lagrangian of a form $L(x,\dot{x})$ with $(x,\dot{x})\in T\overline{M}$ in this section.
\end{Rem}
We define
\[
h_L^T(\bar{m}_0,\bar{m}_1)=\inf_{\substack{\bar{\gamma}(-T)=\bar{m}_0\\
\bar{\gamma}(T)=\bar{m}_1}}\int_{-T}^{T}L(\bar{\gamma}(t),\dot{\bar{\gamma}}(t))dt,\quad\forall\bar{m}_0,\bar{m}_1\in\bar{M}.
\]
From the super-linearity of $L$, we can see that once $\bar{m}_0$ and $\bar{m}_1$ are fixed, there exists a finite $T_{\bar{m}_0,\bar{m}_1}$ such that $h_L^T(\bar{m}_0,\bar{m}_1)$ gets its minimum. We can see this from the following Lemma:

\begin{Lem}\cite{Ch,L}
If the rotation vector of each ergodic minimal measure has positive first component $\pi_1\rho(\mu^{\pm})>0$, then $\forall\bar{m}_0\neq\bar{m}_1$ with $\bar{m}_0\leq0<1\leq\bar{m}_1$, we have
\[
\lim_{T\rightarrow0} h_L^T(\bar{m}_0,\bar{m}_1)=\infty,\quad\lim_{T\rightarrow\infty}h_L^T(\bar{m}_0,\bar{m}_1)=\infty.
\]
\end{Lem}
\begin{proof}
The first formula is easy to be proved as $\bar{m}_0\neq\bar{m}_1$ and $L$ is super-linear of the variable $\dot{x}$. On the other side, 
%let $m_0=\pi\bar{m}_0$, $m_1=\pi\bar{m}_1$ be the projection points. 
if there exists a sequence of $T_n\rightarrow\infty$ such that $\lim_{n\rightarrow\infty}h_L^{T_n}(\bar{m}_0,\bar{m}_1)=K<\infty$, we can find $\bar{\gamma}_n\in\bar{M}$ as the minimizer of $h_L^{T_n}(\bar{m}_0,\bar{m}_1)$, with $\bar{\gamma}_n(-T_n)=\bar{m}_0$ and $\bar{\gamma}_n(T_n)=\bar{m}_1$. 
%We can see that there exists $N\in\mathbb{N}$ such that $\forall n\geq N$, $\|\dot{\bar{\gamma}}_n\|\leq D<\infty$ is uniformly bounded, where $D$ is the constant of the third bullet in Definition \ref{li xia}. 
Let $\zeta:[0,1]\rightarrow M$ be a geodesic connecting $m_1$ to $m_0$, then $\xi_n\doteq\zeta\ast\pi\bar{\gamma}_n$ becomes a loop in $M$. $\forall\epsilon>0$ sufficiently small, we can always take $N$ properly large, such that $\forall n\geq N$,
\begin{eqnarray*}
\frac{1}{2T_n}h_L^{T_n}(\bar{m}_0,\bar{m}_1)&=&\frac{1}{2T_n}\int_{-T_n}^{T_n+1}L(\xi(t),\dot{\xi}(t))dt-\frac{1}{2T_n}\int_0^1L^{+}(\zeta(t),\dot{\zeta}(t))dt,\\
 &=&\frac{1}{2T_n}\int_{-T_n}^{T_n+1}L^{+}(\xi(t),\dot{\xi}(t))dt-\frac{1}{2T_n}\int_0^1L^{+}(\zeta(t),\dot{\zeta}(t))dt,\\
 & +&\frac{1}{2T_n}\int_{-T_n}^{T_n+1}(L-L^{+})(\xi(t),\dot{\xi}(t))dt,\\
 & \geq&\min_{h_1=0}\beta_{L^{+}}(h)-2\epsilon-\frac{1}{2}\min_{h_1=0}\{\beta_{L^{+}}(h),\beta_{L^{-}}(h)\},\\
 & \geq&\frac{1}{2}\min_{h_1=0}\beta_{L^{+}}(h)-2\epsilon>0,
\end{eqnarray*}
as $\pi_1\rho(\mu^{\pm})>0$ and $\epsilon$ is sufficiently small. This implies that $\lim_{n\rightarrow\infty}h_L^{T_n}(\bar{m}_0,\bar{m}_1)=\infty$ and lead to a contradiction. Then we proved the second formula and finished the proof.
\end{proof}
From the proof of this Lemma, we also get that $T_{\bar{m}_0,\bar{m}_1}\rightarrow\infty$ as $-\bar{m}_{0,1},\bar{m}_{1,1}\rightarrow\infty$. Since $L$ is of a transitional form between $L^{-}$ and $L^{+}$ in the domain $\{0\leq x_1\leq1\}$, the following claim holds:
\begin{Clm}
there exists a $K''>0$ such that
\begin{equation}\label{help}
-K''\leq\inf_{\substack{0\leq\bar{x}_{0,1}\leq1\\
0\leq\bar{x}_{1,1}\leq1}}\inf_{T\geq0}h_L^T(\bar{x}_0,\bar{x}_1)\leq\max_{\substack{0\leq\bar{x}_{0,1}\leq1\\
0\leq\bar{x}_{1,1}\leq1}}\inf_{T\geq0}h_L^T(\bar{x}_0,\bar{x}_1)\leq K''.
\end{equation}
\end{Clm}
\begin{proof}
The proof of $\max_{\substack{0\leq\bar{x}_{0,1}\leq1\\
0\leq\bar{x}_{1,1}\leq1}}\inf_{T\geq0}h_L^T(\bar{x}_0,\bar{x}_1)\leq K''$ is easy. For the other part of this claim, we can assume that there exist $\bar{x}_0$ and $\bar{x}_1$ such that $\inf_{T\geq0}h_L^T(\bar{x}_0,\bar{x}_1)=-\infty$. If so, there must be $\bar{\gamma}_n$ and $T_n\rightarrow\infty$ such that $h_L^{T_n}(\bar{x}_0,\bar{x}_1)\rightarrow-\infty$, where $\bar{\gamma}_n(-T_n)=\bar{x}_0$ and $\bar{\gamma}_n(T_n)=\bar{x}_1$. But $d(\bar{x}_0,\bar{x}_1)$ is bounded, we can construct a similar loop $\xi_n$ as above Lemma and lead to a contradiction that $h_L^{T_n}(\bar{x}_0,\bar{x}_1)\rightarrow+\infty$.
\end{proof}
Once this claim is available, we can see that 
\begin{equation}
\inf_{T\geq0}h_L^{T}(\bar{m}_0,\bar{m}_1)\geq\inf_{\bar{x}_{0,1}=0}\inf_{T^{-}\geq0}h_L^{T^{-}}(\bar{m}_0,\bar{x}_0)+\inf_{\bar{x}_{1,1}=1}\inf_{T^{+}\geq0}h_L^{T^{+}}(\bar{x}_1,\bar{m}_1)-K'',
\end{equation}
and
\begin{equation}
\inf_{T\geq0}h_L^{T}(\bar{m}_0,\bar{m}_1)\leq\max_{\bar{x}_{0,1}=0}\inf_{T^{-}\geq0}h_L^{T^{-}}(\bar{m}_0,\bar{x}_0)+\max_{\bar{x}_{1,1}=1}\inf_{T^{+}\geq0}h_L^{T^{+}}(\bar{x}_1,\bar{m}_1)+K''.
\end{equation}
Then $\forall\{\bar{m}_0^n\}_{n\in\mathbb{N}}$ and $\{\bar{m}_1^n\}_{n\in\mathbb{N}}$ sequences with $-\bar{m}_{0,1}^n$, $\bar{m}_{1,1}^n\rightarrow\infty$ as $n\rightarrow\infty$, $\|\inf_{T\geq0}h_L^{T}(\bar{m}_0^n,\bar{m}_1^n)\|$ and $\|\dot{\bar{\gamma}}_n\|$ is uniformly bounded for $t\in[-T_{\bar{m}_0^n,\bar{m}_1^n}]$, $\forall n\in\mathbb{N}$. That's because we can find $\bar{x}_0^n$ $(\bar{x}_1^n)$ as the first (last) intersection point of $\bar{\gamma}_n$ with $\{x_1=0\}$ $(\{x_1=1\})$. We denote the segment of $\bar{\gamma}_n$ from $\bar{x}_0^n$ to $\bar{x}_1^n$ by $\bar{\gamma}_n\big{|}_{\bar{x}_0^n}^{\bar{x}_1^n}$. From (\ref{help}), we can see that $\dot{\bar{\gamma}}_n\big{|}_{\bar{x}_0^n}^{\bar{x}_1^n}$ is also uniformly bounded. On the other side, we can see that the segment $\bar{\gamma}_n\big{|}_{\bar{m}_0^n}^{\bar{x}_0^n}$ $(\bar{\gamma}_n\big{|}_{\bar{x}_1^n}^{\bar{m}_1^n})$ satisfies the E-L equation of $L^{-}$ $(L^{+})$. Then as $n\rightarrow\infty$, 
\[
A_{L^{-}}(\bar{\gamma}_n\big{|}_{\bar{m}_0^n}^{\bar{x}_0^n})-h_{L^{-}}^{\infty}(m_0^n,y)-h_{L^{-}}^{\infty}(y,x_0^n)\rightarrow0
\]
and
\[
A_{L^{+}}(\bar{\gamma}_n\big{|}_{\bar{x}_1^n}^{\bar{m}_1^n})-h_{L^{+}}^{\infty}(x_1^n,z)-h_{L^{+}}^{\infty}(z,m_1^n)\rightarrow0
\]
hold. Here $y\in{\mathcal{M}}(L^{-})$, $z\in{\mathcal{M}}(L^{+})$ and $x_i^n\in M$ (${m}_i^n\in M$) is the projection of $\bar{x}_i^n$ $(\bar{m}_i^n)$, $i=0,1$. Then $\bar{\gamma}_n$ $C^1-$uniformly converges to a $C^1-curve$ $\bar{\gamma}:\mathbb{R}\rightarrow\bar{M}$ which satisfies the following:
%Notice that $\|\dot{\bar{\gamma}}_{\bar{m}_0,\bar{m}_1}\|$ is also uniformly bounded by $D$ as $T_{\bar{m}_0,\bar{m}_1}\rightarrow\infty$. So the set of pseudo curves $\mathscr{G}(L)\neq\emptyset$ and can be defined as following:
\begin{defn}
A curve $\bar{\gamma}:\mathbb{R}\rightarrow\bar{M}$ is contained in $\mathscr{G}(L)$ if
\[
A_L(\bar{\gamma})\big{|}_{[-T,T]}=\inf_{T'\in\mathbb{R}_{+}}h_L^{T'}(\bar{\gamma}(-T),\bar{\gamma}(T)),\quad\forall T\in\mathbb{R}_{+}.
\]
\end{defn}
We can see that $\mathscr{G}(L)\neq\emptyset$ based on above analysis. Besides, we have:
\begin{Pro}\cite{Ch}
There exists some $K>0$ such that $\forall \bar{\gamma}\in\mathscr{G}(L)$ and $T>0$, $\|h_L^T(\bar{\gamma}(-T), \bar{\gamma}(T))\|\leq K$ holds.
\end{Pro}
Each $k\in\mathbb{Z}$ defines a Deck transformation ${\bf k}: \bar{M}\rightarrow\bar{M}$ with ${\bf k}x=(x_1+k,x_2,\cdots,x_n)$. Let $\bar{M}^{-}=\{x\in\bar{M}\big{|}x_1\leq0\}$ and $\bar{M}^{+}=\{x\in\bar{M}\big{|}x_1\geq1\}$.
\begin{defn}
A curve $\bar{\gamma}\in\mathscr{G}(L)$ is called {\bf pseudo connecting curve} if the following holds
\[
A_L(\bar{\gamma})\big{|}_{[-T,T]}=\inf_{\substack{T'\in\mathbb{R}_{+}\\
{\bf k^{-}}\bar{\gamma}(-T)\in\bar{M}^{-}\\
{\bf k^{+}}\bar{\gamma}(T)\in\bar{M}^{+}
}}h_L^{T'}({\bf k^{-}}\bar{\gamma}(-T),{\bf k^{+}}\bar{\gamma}(T))
\]
for each $\bar{\gamma}(T)\in\bar{M}^{-}$ and $\bar{\gamma}(T)\in\bar{M}^{+}$. We can denote the set of pseudo connecting curves by $\mathscr{C}(L)$.
\end{defn}
\begin{Lem}
The set $\mathscr{C}(L)$ is not empty.
\end{Lem}
\begin{proof}
Since it's a direct cite of \cite{Ch}, we only give the sketch of proof here. As we know, $\mathscr{G}(L)\neq\emptyset$. So we start with a curve $\bar{\gamma}\in\mathscr{G}(L)$. Given a $\Delta>0$, we claim that there are finitely many intervals $[t_i^{-},t_i^{+}]$ such that ${\bf k_i^{-}}\bar{\gamma}(t_i^{-})$ can be connected to ${\bf k_i^{+}} \bar{\gamma}(t_i^{+})$ by another curve $\zeta_i$ with $\Delta-$smaller action than the original one. This is because the previous Proposition and $\Delta>0$. For a sequence $\Delta_i\rightarrow0$, we can do finitely many surgeries on $\bar{\gamma}$ and get a sequence $\bar{\gamma}_i\in\mathscr{G}(L)$ satisfying
\[
A_L(\bar{\gamma}_i)\big{|}_{[-T,T]}\leq \inf_{\substack{T'\in\mathbb{R}_{+}\\
{\bf k^{-}}\bar{\gamma}_i(-T)\in\bar{M}^{-}\\
{\bf k^{+}}\bar{\gamma}_i(T)\in\bar{M}^{+}}}h_L^{T'}({\bf k^{-}}\bar{\gamma}_i(-T),{\bf k^{+}}\bar{\gamma}_i(T))+\Delta_i,\quad\forall T\in\mathbb{R}_{+}.
\]
On the other side, $\forall T>0$, $\exists i_0$ such that the set $\{\bar{\gamma}_i\big{|}_{[-T,T]}:i\geq i_0\}$ is pre-compact in $C^1([-T,T],\bar{M})$. Let $T\rightarrow\infty$, by diagonal extraction argument $\bar{\gamma}_i$ converges $C^1-$uniformly to a $C^1-$curve $\bar{\gamma}:\mathbb{R}\rightarrow\bar{M}$. Obviously, $\bar{\gamma}\in\mathscr{C}(L)$.
\end{proof}
\begin{The}
The map $L\rightarrow\mathscr{C}(L)$ is upper semicontinuous.
\end{The}
\begin{Cor}
If the space-step Lagrangian $L$ is periodic of $x_1$ variable, then $\bar{\gamma}\in\mathscr{C}(L)$ iff the projection $\gamma=\pi\bar{\gamma}:\mathbb{R}\rightarrow M$ is semi-static.
\end{Cor}
Similar to the definition for time-step Lagrangian, we define
\[
\widetilde{\mathcal{C}}(L)=\bigcup_{\bar{\gamma}\in\mathscr{C}(L)}(\bar{\gamma}(t),\dot{\bar{\gamma}}(t)),\quad\mathcal{C}(L)=\bigcup_{\bar{\gamma}\in\mathscr{C}(L)}\bar{\gamma}(t).
\]
We can see that $\pi\widetilde{\mathcal{C}}(L)=\widetilde{\mathcal{N}}(L)$ and $\pi\mathcal{C}(L)=\mathcal{N}(L)$, where $\pi:\bar{M}\rightarrow M$ is the standard projection.
\vspace{10pt}

\subsection{Local connecting orbits of c-equivalence}\label{c-equivalence}
$\newline$

Recall that a time-step Lagrnagian can be considered as a space-step one from Remark \ref{uniform}, so we only deal with the space-step case in this subsection. This new version of $c-$equivalence is firstly raised in \cite{L}, which is more general than the earlier one raised in \cite{Mat3}. This type of connecting orbits are found in the annulus of incomplete intersection and plays a key role in establishing transition chain crossing double resonance.

Assume $\phi:\mathbb{T}^{n-1}\rightarrow\mathbb{T}^n$ is a smooth injection and $\Sigma_c$ is the image of $\phi$. Let $\mathfrak{C}\subset H^1(\mathbb{T}^n,\mathbb{R})$ be a connected set where we are going to define $c-$equivalence. For each class $c\in\mathfrak{C}$, we assume that there exists a non-degenerate embedded $(n-1)-$dimensional torus $\Sigma_c\subset\mathbb{T}^n$ such that each $c-$semi static curve $\gamma$ transversally intersects $\Sigma_c$. Let
\[
\mathbb{V}_c=\bigcap_U\{i_{*}H_1(U,\mathbb{R})\big{|}U \text{ is a neighborhood of }\mathcal{N}(c)\cap\Sigma_c\},
\]
where $i:U\rightarrow M$ is the inclusion map. $\mathbb{V}_c^{\bot}$ is defined to be the annihilator of $\mathbb{V}_c$, i.e. if $c'\in H^1(\mathbb{T}^n,\mathbb{R})$, $\langle c',h\rangle=0$ for all $h\in\mathbb{V}_c$. Clearly,
\[
\mathbb{V}_c^{\bot}=\bigcup_U\{\ker i^{*}\big{|}U \text{ is a neighborhood of } \mathcal{N}(c)\cap\Sigma_c\}.
\]
Note that there exists a neighborhood $U$ of $\mathcal{N}(c)\cap\Sigma_c$ such that $\mathbb{V}_c=i_{*}H_1(U,\mathbb{R})$ and $\mathbb{V}_c^{\bot}=\ker i^*$ \cite{Mat3}.

We say that $c,c'\in H^1(M,\mathbb{R})$ are {\bf $c-$equivalent} if there exists a continuous curve $\Gamma:[0,1]\rightarrow\mathfrak{C}$ such that $\Gamma(0)=c$, $\Gamma(1)=c'$ and $\alpha(\Gamma(s))$ keeps constant for all $s\in[0,1]$. Besides we need $\forall s_0\in[0,1]$, $\exists\delta>0$ such that $\Gamma(s)-\Gamma(s_0)\in\mathbb{V}_{\Gamma(s_0)}^{\bot}$ whenever $s\in[0,1]$ and $\|s-s_0\|<\delta$.

\begin{The}
Assume the cohomology class $c^*$ is $c-$equivalent to the class $c'$ through the path $\Gamma:[0,1]\rightarrow H^1(\mathbb{T}^n,\mathbb{R})$. For each $s\in[0,1]$, the followings are satisfied:
\begin{itemize}
\item There exists at least one component of rotation vector which is positive, i.e. $\forall \mu_{\Gamma(s)}$ ergodic $\Gamma(s)-$minimizing measure, $\omega_j(\mu_{\Gamma(s)})>0$ for some $j\in\{1,2,\cdots,n\}$.
\item We can find finitely many $\{c_i\}_{i=1}^k\subset\Gamma$ $(c_1=c^*,c_k=c')$ and closed 1-forms $\eta_i$, $\bar{\mu}_i$ on $M$ with $[\eta_i]=c_i$ and $[\bar{\mu}_i]=c_{i+1}-c_i$, and smooth functions $\varrho_i$ on $\overline{M}_{j(i)}$ for $j=1,2,\cdots,k-1$ such that the pseudo connecting curve se $\mathscr{c}(L_i)$ for the space-step Lagrangian 
\[
L_i=L-\eta_i-\varrho_i\bar{\mu}_i
\]
possesses the properties:
\begin{itemize}
\item each curve $\bar{\gamma}\in\mathscr{C}(L_i)$ determines an E-L orbit $(\gamma,\dot{\gamma})$ of $\phi_L^t$;
\item the orbit $(\gamma,\dot{\gamma})$ connects $\tilde{\mathcal{A}}(c_i)$ to $\tilde{\mathcal{A}}(c_{i+1})$, i.e. the $\alpha-$limit set is contained in $\tilde{\mathcal{A}}(c_i)$ and the $\omega-$limit set is contained in $\tilde{\mathcal{A}}(c_{i+1})$.
\end{itemize}
Here $\overline{M}_{j(i)}=\mathbb{T}^{j(i)-1}\times\mathbb{R}\times\mathbb{T}^{n-j(i)}$ is the covering space of $M=\mathbb{T}^n$, and $\varrho_i$ is a smooth map of a form
\[
\varrho_i:\overline{M}_{j(i)}\rightarrow\mathbb{R}\quad via\quad\varrho_i(x)=\varrho_i(x_j),
\]
with $\varrho_i(x_j)=0$ for $x_j\leq0$ and $\varrho_i(x_j)=1$ for $x_j\geq1$ (see figure \ref{fig11}).
\end{itemize}
\end{The}
\begin{proof}
By the definition of $c-$equivalence, for each $c=\Gamma(s)$ $(s\in[0,1])$ on the path, there exists $\epsilon>0$ such that $\Gamma(s')-c\in\mathbb{V}_{\Gamma(s)}^{\bot}$ whenever $s'\in[0,1]$ and $\|s-s'\|<\epsilon$. Thus there exist a non-degenerately embedded $(n-1)-$dimensional torus $\Sigma_c$, a closed form $\bar{\mu}_c$ and a neighborhood $U$ of $\mathcal{N}(c)\cap\Sigma_c$ such that $[\bar{\mu}_c]=\Gamma(s')-c$ and supp$\bar{\mu}_c\cap U=\emptyset$. We can restrict $\Sigma_c$ into the elementary domain $\mathbb{T}^{j-1}\times[0,1]\times\mathbb{T}^{n-j}$ of $\overline{M}_j$ and let $\mathcal{B}(\Sigma_c,\delta)$ be the $\delta-$neighborhood of $\Sigma_c$ in it. Then if $\eta$ and $\bar{\mu}_c$ are closed $1-$forms such that $[\eta]=c$ and $[\eta+\bar{\mu}_c]=\Gamma(s')=c'$, we have
\[
\mathcal{B}(\Sigma_c,\delta)\cap\mathcal{B}(\mathcal{C}(L+\eta),\delta)\subset U,
\]
as long as $\delta>0$ is chosen sufficiently small. That's because $\mathcal{N}(c)=\mathcal{C}(L+\eta)$. Then from the upper semicontinuity of $\mathcal{C}(L)$ w.r.t $L$, we have
\[
\mathcal{B}(\Sigma_c,\delta)\cap\mathcal{B}(\mathcal{C}(L+\eta+\varrho_i\bar{\mu}_c),\delta)\subset U,
\]
as long as $\varrho_i\bar{\mu}_c$ is sufficiently small. This is available because of the definition of $c-$equivalence.

Besides, we can see that $\bar{\mu}_c$ can be chosen with its support disjoint from $U$. Then $\forall\bar{\gamma}\in\mathscr{C}(L+\eta+\varrho_i\bar{\mu}_c)$ is a solution of E-L equation determined by $L$, i.e. the term $\varrho_i\bar{\mu}_c$ has no contribution to the equation along $\bar{\gamma}$.

From the definition of $\mathscr{C}$ we get that the projection of $\bar{\gamma}$, which is denoted by $\gamma\in M$, satisfies that $\gamma\big{|}_{(-\infty,t_0]}$ is backward $\Gamma(s)-$semi static once $\bar{\gamma}\big{|}_{(-\infty,t_0]}$ falls entirely into $\bar{M}_j^{-}$. Similarly, we have $\gamma\big{|}_{[t_1,+\infty)}$ is forward $\Gamma(s')-$semi static once $\bar{\gamma}\big{|}_{[t_1,\infty)}$ falls entirely into $\bar{M}_j^{+}$. Therefore, $(\gamma(t),\dot{\gamma}(t))\rightarrow\tilde{\mathcal{A}}(\Gamma(s))$ as $t\rightarrow-\infty$ and $(\gamma(t),\dot{\gamma}(t))\rightarrow\tilde{\mathcal{A}}(\Gamma(s'))$ as $t\rightarrow+\infty$.

Because of the compactness of $[0,1]$, there are finitely many numbers $s_1,s_2,\cdots,s_k\in[0,1]$ such that above argument applies if $s$ and $s'$ are replaced respectively by $s_i$ and $s_{i+1}$. We just take $c_i=\Gamma(s_i)$.
\end{proof}

\begin{Cor}
Let $c_i$, $\eta_i$, $\bar{\mu}_i$ and $\varrho_i$ be evaluated as in above Theorem, and $U_i$ be an open neighborhood of $\mathcal{N}(c_i)\cap\Sigma_{c_i}$ such that $U_i\cap supp\bar{\mu}_i=\emptyset$. Then for large $K_i>0$, $T_i>0$, small $\delta>0$ and $\forall \bar{m},\bar{m}'\in\overline{M}_{j(i)}$ with $-K_i\leq\bar{m}_1\leq-K_i+1$, $K_i-1\leq\bar{m}'_1\leq K_i$, the quantity $h_{\eta_i,\mu_i}^T(\bar{m},\bar{m}')$ reaches its minimum at some $T<T_i$ and the corresponding minimizer $\bar{\gamma}_i(t,\bar{m},\bar{m}')$ satisfies the condition Image$(\bar{\gamma}_i)\cap\mathcal{B}(\Sigma_{c_i},\delta)$.
\end{Cor}
\begin{Rem}
If we take $L(x,\dot{x},t)$ as periodic Tonelli Lagrangian of $T\mathbb{T}^n\times\mathbb{S}^1$ and $\overline{M}=\mathbb{R}\times\mathbb{T}^n$ with $\theta=(t,x)\in\mathbb{S}^1\times\mathbb{T}^n$. Then the previous Theorem of $c-$equivalence is just the special case discovered by Bernard in \cite{B} and Cheng in \cite{CY1,CY2}.
\end{Rem}
\vspace{10pt}

With this approach, we can prove that $c-$equivalence of 2-resonance. 
\begin{The}{\bf ($c-$equivalence of Annulus of Incomplete Intersection)}
Let $\Gamma\subset\mathbb{A}\subset\alpha_{H}^{-1}(E)$ is the curve skirting around the flat $\mathbb{F}$, where $E\in(0,\Delta_0]$ and $H$ is the homogenized system of 2-resonance. $\forall c,c'\in\Gamma$ is $c-$equivalent with each other (see figure \ref{fig21}).
%$L$ is the according Lagrangian of a form (\ref{homogenized Lagrangian}).
\end{The}
\begin{proof}
Recall that we can consider the Lagrangian \ref{homogenized Lagrangian} as an autonomous one with $\Theta=(X_1,X_2,S)\in\mathbb{T}^3$. As $E>0$, $\forall c\in\Gamma$, we have $\omega(\mu_c)\neq0$. Without loss of generality, we can assume $\omega_1(\mu_c)>0$. Then $\Sigma_c=\{X_1=0\}$ is a $2-$dimensional section of $\mathbb{T}^3$ such that each $c-$semi static orbit intersects it transversally and 
\[
\mathcal{N}(c)\cap\Sigma_c\subset\{(X_1,X_2,S)\big{|}X_1=0,X_2\in\cup I_{c,i}\},
\]
where $\{I_{c,i}\}$ are finitely many intervals of $\mathbb{T}^1$ disjoint from each other. We just need to prove the equivalence for $c'$ sufficiently close to $c$.

Clearly, there exists $U\supset\mathcal{N}(c)\cap\Sigma_c$ such that $\mathbb{V}_c=i_{*}H_1(U,\mathbb{R})=span_{\mathbb{R}}\{(0,0,1)\}$. Then we have $\mathbb{V}_c^{\bot}=span_{\mathbb{R}}\{(1,0,0),(0,1,0)\}$. For each $c'\in\Gamma$ sufficiently close to $c$, one has $c'-c=(\Delta c_1,\Delta c_2,0)\in\mathbb{V}_c^{\bot}$. Thus, there exists  closed 1-form $\bar{\mu}$ with $[\bar{\mu}]=c'-c$ and 
\[
supp\bar{\mu}\cap\mathcal{N}(c)\cap\Sigma_c=\emptyset.
\]
Therefore, all classes along the curve $\Gamma$ are equivalent in this case.
\end{proof}
\vspace{10pt}
\subsection{Local connecting orbits of h-equivalence}
$\newline$

Another type of locally connecting orbits look like heteroclinic orbits. That's the reason we call them {\bf type-h}. This type orbits are mainly used to deal with the Diffusion problem of Arnold mechanism, which was firstly raised in \cite{CY1,CY2}. As the time-periodic Lagrangian is more convenient for our application, we won't consider it as an autonomous one in this subsection.

%This type of orbits is used to handle with a typical case when the Aubry set falls into a neighborhood $N$ of some lower-dimensional torus such that $H_1(M,N,\mathbb{Z})\neq\emptyset$. 
For a Tonelli Lagrangian $L(x,\dot{x},t):TM\times\mathbb{S}^1\rightarrow\mathbb{R}$ with $M=\mathbb{T}^2$ in our case of a form (\ref{1-resonant Lagrangian}) or (\ref{transitional Lagrangian}), we can assume $\vec{e}_1$ is a base vector of $H^1(M,N,\mathbb{Z})$ without loss of generality. Then we can take $\tilde{M}=2\mathbb{T}\times\mathbb{T}^{n-1}$ as a finitely covering manifold of $M$. Restricted to the uniform section of $\tilde{M}\times\{t=0\}$, $\mathcal{A}(L)$ will become two different connected components $\mathcal{A}_i$, $i=0,1$. With the approach of \cite{Con} we can connect $\mathcal{A}_0$ to $\mathcal{A}_1$ with a semi-static heteroclinic orbits $\gamma_{0,1}$ of $\tilde{M}$. Besides, we can find $N_i$ as the open neighborhood of $\mathcal{A}_i$, $i=0,1$ such that dist$(N_0,N_1)>0$. 
%Note that $H_1(M,N,\mathbb{Z})\setminus span_{\mathbb{Z}}\{\vec{e}_1\}\neq\emptyset$ is possible and $\mathcal{A}(L)$ will become more than two connected components of $\tilde{M}$. But as long as $\mathcal{A}_i$ is finite, we can find $\gamma_{i,i+1}$ semi-static heteroclinic orbits connecting them with $\omega(\gamma_{i,i+1})\subset\mathcal{A}_{i+1}$ and $\alpha(\gamma_{i,i+1})\subset\mathcal{A}_i$ from \cite{Con}. In our case, we can see that $n=2$ and $H^1(M,N,\mathbb{Z})=span_{\mathbb{Z}}\{\vec{e}_1\}$. So we just need to consider the case of two connected components $\mathcal{A}_i$ with $i=0,1$.
\begin{Lem}{\bf (Connecting Lemma)}
For $c\in\mathbb{W}_{g_j}$ with $j=1,2$ as in section 5, the Aubry set contains two different connected components $\mathcal{A}_i$ in $\tilde{M}$ and we can find $N_i$ open neighborhoods disjoint with each other containing them separately, $i=0,1$. If there exists one semi-static heteroclinic orbit connecting $\mathcal{A}_0(c)$ to $\mathcal{A}_1(c)$ which is disconnected to the others, then there exists some orbit $d\gamma'$ of $\phi_L^t$ connecting $\tilde{\mathcal{A}}_0(c)$ to $\tilde{\mathcal{A}}_1(c')$ for class $c'\in\mathbb{W}_{g_j}$ close to $c$. Here $\mathcal{A}_i(c')\subset N_0\cup N_1$, $i=0,1$.
\end{Lem}
\begin{proof}
This Lemma is also a direct cite of results of \cite{CY1,CY2,Ch}. So we just give the sketch for the proof. Assume $\gamma$ is the isolated heteroclinic orbit connecting $\mathcal{A}_0(c)$ to $\mathcal{A}_1(c)$. We consider the modified Lagrangian
\[
L_{\eta,\mu,\psi}=L-\eta-\mu-\psi,
\]
where $\eta$ is a closed 1-form on $\mathbb{T}^2$ with $[\eta]=c$, $\mu$ is a 1-form depending on $t$ variable in the way that $\mu\equiv0$ for $\{t\leq0\}$ and $\mu=\bar{\mu}$ for $\{t\geq1\}$, where $\bar{\mu}$ is a closed 1-form on $\mathbb{T}^2$ with $[\bar{\mu}]=c'-c$. $\psi(x,v):\mathbb{T}^2\times\mathbb{R}\rightarrow\mathbb{R}$ is a smooth function with $\psi(x,t)\equiv0$ for $t\in(-\infty,0]\cup[1,\infty)$. Let $m,m'$ be two points in $\tilde{M}=2\mathbb{T}^1\times\mathbb{T}^1$, we define 
\[
h^T_{\eta,\mu,\psi}(m,m')=\inf_{\substack{\gamma(-T)=m\\
\gamma(T)=m'}}\int_{-T}^{T}L_{\eta,\mu\psi}(d\gamma(t),t)dt+\alpha(c)dt,
\]
where $T\in\mathbb{Z}_{+}$. Note that $L_{\eta,\mu,\psi}$ can be considered as a time-step Lagrangian of $\tilde{M}$. Then from previous subsection of time-step Lagrangian we can define the sets $\mathscr{C}_{\eta,\mu,\psi}=\mathscr{C}(L_{\eta,\mu,\psi})$, $\mathcal{C}_{\eta,\mu,\psi}$ and $\tilde{\mathcal{C}}_{\eta,\mu,\psi}$. Recall that $\mathcal{C}_{\eta,\mu,\psi}$ is upper semicontinuous of $L$ and $\mathcal{C}_{\eta,0,0}=\mathcal{N}(c)$. Then $\gamma\in\mathcal{N}(c)$ and for $c'$ sufficiently close to $c$ and $\psi$ sufficiently small, there must be an orbit $\gamma'\in\mathscr{C}_{\eta,\mu,\psi}$ close to $\gamma$. We just need to show that $\gamma'$ is an orbit of $\phi_L^t$ by properly chosen $\psi$ and $\bar{\mu}$.

As $\gamma$ is an isolated semi-static heteroclinic orbit, we can find a open and homologically-trivial ball $O$ of it on the section $\{t=0\}$. Besides, no other semi-static heteroclinic orbits passing through $O$ (see figure \ref{fig26}). Then we can define a non-negative function $f\in C^r(\mathbb{T}^2,\mathbb{R})$ with
\[
f(x)=\left\{
\begin{array}{ccccccccc}
0&\quad& \text{ in }\mathcal{B}(O,\delta_1)^c,\\
1&\quad & \text{ in }O,\\
<1&\quad& \text{ elsewhere, }
\end{array}
\right.
\]
and $\psi(x,t)=\lambda\rho(t)f(x)$ with $\rho(t)$ is showed as figure \ref{fig11}. Here $\lambda$ is just a coefficient which can be chosen sufficiently small. On the other side, since $O$ is homologically trivial, we can choose $\bar{\mu}$ with support disjoint from $O$. Once again we use the upper semicontinuous of $\mathcal{C}_{\eta,\mu,\psi}$ of $L$ we can verified $\gamma'$ is a real orbit of $\phi_L^t$.
\end{proof}
\vspace{5pt}

The orbit $d\gamma'$ obtained in above Lemma is locally minimal in the sense we define in the following, which is crucial for the construction of globally connecting orbits.
\begin{defn}
$d\gamma:\mathbb{R}\rightarrow TM$ is called {\bf locally minimal orbit of type-h} connecting $\tilde{\mathcal{A}}(c)$ to $\tilde{\mathcal{A}}(c')$ if
\begin{itemize}
\item $d\gamma_{0,1}$ is an orbit of $\phi_L^t$, with the $\alpha-$limit and $\omega-$limit sets of it contained in $\tilde{\mathcal{A}}(c)$ and $\tilde{\mathcal{A}}(c')$ respectively, i.e. restricted on the section $\{t=0\}$, $\alpha(d\gamma_{0,1})\big{|}_{t=0}\subset TN_0$ and $\omega(d\gamma_{0,1})\big{|}_{t=0}\subset TN_1$;
\item there exist two open balls $V_0$, $V_1$ of $\tilde{M}$ and two positive integers $k^{-}, k^{+}$ such that $\bar{V}_0\subset N_0\setminus\mathcal{A}_0(c)$, $\bar{V}_1\subset N_1\setminus\mathcal{A}_1(c')$, $\gamma(-k^{-})\in V_0$, $\gamma(k^{+})\in V_1$ and
\begin{eqnarray}\label{locally minimal}
h_c^{\infty}(x^{-},m_0)&+&h^{k^{-},k^{+}}_{\eta,\mu,\psi}(m_0,m_1)+h_{c'}^{\infty}(m_1,x^{+})\nonumber\\
&-&\liminf_{\substack{k_i^{-}\rightarrow\infty\\
k_i^{+}\rightarrow\infty}}\int_{-k_i^{-}}^{k_i^{+}}L_{\eta,\mu,\psi}(d\gamma(t),t)dt-k_i^{-}\alpha(c)-k_i^{+}\alpha(c')>0
\end{eqnarray}
holds for all $(m_0,m_1)\in\partial(V_0\times V_1)$, $x^{-}\in N_0\cap\mathcal{A}_0(c)\big{|}_{t=0}$, $x^{+}\in N_1\cap\mathcal{A}_1(c')\big{|}_{t=0}$, where $k_i^{\pm}\in\mathbb{Z}^{+}$ are sequences such that $\gamma(-k_i^{-})\rightarrow x^{-}$ and $\gamma(k_i^{+})\rightarrow x^{+}$ as $i\rightarrow\infty$ (see figure \ref{fig26}).
\end{itemize}
\end{defn}
\begin{figure}
\begin{center}
\includegraphics[width=12cm]{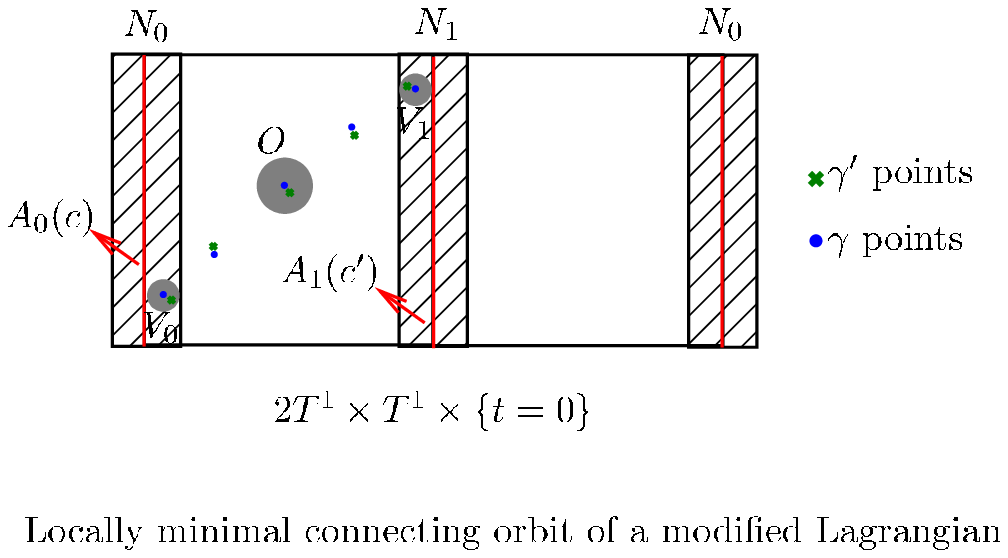}
\caption{ }
\label{fig26}
\end{center}
\end{figure}
\begin{Rem}
Inequality (\ref{locally minimal}) tells us that once a curve $\bar{\gamma}$ touches the boundary of $V_i$, the action of $L_{\eta,\mu,\psi}$ along it will be larger than the action along $\gamma$, $i=0,1$. As $V_i$ can be chosen arbitrarily small, it's reasonable to call it {\bf locally minimal}.
\end{Rem}
\vspace{10pt}

\subsection{Generalized transition chain}
$\newline$

Based on the discussion of locally connecting orbits, of c-type and h-type, now we can find a {\bf generalized transition chain} (GTC) to verify the existence of global connecting orbits, i.e. the diffusion orbits with large change of momentum variables.

Recall that the earliest definition of GTC was given by J. Mather in \cite{Mat3} for autonomous systems, then \cite{B, CY1, CY2} generalized it to the time-periodic case. From \cite{B} we know that if $c$, $c'\in H^1(M,\mathbb{R})$ is equivalent with each other and there exists a GTC connecting them, then both $c$ and $c'$ lie in a flat $\mathbb{F}$ of $\alpha_L$, where $L$ is an autonomous Lagrangian. This case is of no interest for us since $\tilde{\mathcal{A}}_L(c)\cap\tilde{\mathcal{A}}_L(c')\neq\emptyset$, this point was revealed by Massart in \cite{Ms}. 

In \cite{L}, they gave us a new way to get locally connecting orbits with a local surgery method in a proper covering space. This skill we have illustrated in the subsection \ref{c-equivalence}.

With these two branches of equivalent skills, globally connecting orbits then can be constructed shadowing these locally connecting orbits.

\begin{defn}
Let $c$, $c'$ be two classes in $H^1(M,\mathbb{R})$. We say that $c$ is joined with $c'$ by a GTC if a continuous curve $\Gamma:[0,1]\rightarrow H^1(M,\mathbb{R})$ exists such that $\Gamma(0)=c$, $\Gamma(1)=c'$ and for each $s\in[0,1]$ at least one of the following cases takes place:
\begin{itemize}
\item {\bf (h-type)} the Aubry set is contained in a domain $N\subset M$ with nonzero topological codimension.
%$H^1(M, N, \mathbb{R})\neq\emptyset$.
There exist a certain finitely covering space $\tilde{M}$, two open domains $N_1,N_2$ with dist$(N_1,N_2)>0$ and small numbers $\delta_s, \delta'_s>0$ such that\\
\begin{itemize}
\item the Aubry set $\mathcal{A}_0(\Gamma(s))\cap N_1\neq\emptyset$, $\mathcal{A}_1(\Gamma(s))\cap N_2\neq\emptyset$ and $\mathcal{A}_i(\Gamma(s'))\cap(N_1\cup N_2)\neq\emptyset$ for each $s,s'$ with $\|s-s'\|\leq\delta_s$, $i=0,1$;\\
\item $\pi\mathcal{N}(\Gamma(s),\tilde{M})\setminus\mathcal{B}(\mathcal{A}_i,\delta'_s)\neq\emptyset$ and there exists at least one isolated orbit in it;\\
\end{itemize}
\item {\bf (c-type)} for each $s'\in(s-\delta_s,s+\delta_s)$, $\Gamma(s')$ is equivalent to $\Gamma(s)$. Namely, there exists a neighborhood of $\mathcal{N}(\Gamma(s))$, which is denoted by $U$, such that $\Gamma(s')-\Gamma(s)\in\ker i^*$.
\end{itemize}
\end{defn}

If $\Gamma:[0,1]\rightarrow H^1(M,\mathbb{R})$ is a GTC connecting $c$ and $c'$, then we can find a partition for it with $c_j=\{\Gamma(s_j)\}_{j=1}^k$. Here $s_j\in[0,1]$, $c=\Gamma(s_1)$, $c'=\Gamma(s_k)$ and 
\[
\{0,1,\cdots k\}=\underbrace{\{1,2,\cdots,i_1\}}_{\Gamma_{i_1}}\bigcup\underbrace{\{i_1+1,\cdots,i_2\}}_{\Gamma_{i_2-i_1}}\bigcup\cdots\bigcup\underbrace{\{i_{m-1}+1,\cdots,i_m\}}_{\Gamma_{i_m-i_{m-1}}}
\]
with $i_m=k$. $\tilde{\mathcal{A}}(c_l)$ and $\tilde{\mathcal{A}}(c_{l+1})$ can be connected by local connecting orbits of same type (h- or c-), as long as $c_i,c_{i+1}\in \Gamma_{i_j-i_{j-1}}$, $j\in\{1,2,\cdots,m\}$. This partition always can be found because $k$ can be chosen sufficiently large and $\widetilde{\mathcal{N}}(c)$ is upper semicontinuous of $c$.\\

For our case, the GTC should be chosen in the set $\bigcup_{i=M}^{\infty}(\mathbb{W}_{g_1}^i\cup\mathbb{A}^i\cup\mathbb{W}_{g_2}^i)$. Since our construction is self-similar, and uniform restrictions ensure that this set is toppologically connected, we can take $\Gamma=\bigcup_{i=M}^{\infty}\Gamma_i$ as a GTC and $\Gamma_i=\Gamma_{i_1}\cup\Gamma_{i_2}\cup\Gamma_{i_3}$ as a partition, where $\Gamma_{i_1}\subset\mathbb{W}_{g_2}^i$, $\Gamma_{i_2}\subset\mathbb{A}^i$ and $\Gamma_{i_3}\subset\mathbb{W}_{g_1}^i$. Recall that $M\gg1$ can be chosen {\it a posterior} large enough. $\Gamma_{i_1}$ and $\Gamma_{i_3}$ is of $h-$type and $\Gamma_{i_2}$ is of $c-$type, this point has been showed in previous subsections. But for the validity of {\bf Connecting Lemma}, we also need the existence of isolated semi-static orbits for all $\Gamma_{i_1}$ and $\Gamma_{i_3}$, $i=M,M+1,\cdots,\infty$. So we need the following regularity and genericity conditions of wNHICs of \cite{CY1,CY2}.

\begin{Lem}{\bf (Regularity\cite{CY1,CY2,Ch})}
For a fixed $i\in\{M,M+1,\cdots,\infty\}$, we can take 
\[
\Gamma_{i_1}=\{(0,c_2(s))\in\mathbb{W}_{g_2}^i\big{|} c_2(s)\neq c_2(s')\text{ for }s\neq s'\in[0,1]\}
\]
and 
\[
\Gamma_{i_3}=\{(c_1(s),0)\in\mathbb{W}_{g_1}^i\big{|} c_2(s)\neq c_2(s')\text{ for }s\neq s'\in[0,1]\}.
\]
Besides, we can introduce two area parameters $\sigma_j(s)$ which is one-to-one with $c_j(s)$, $s\in[0,1]$ and $j=1,2$. Then in the covering space $\tilde{M}$, we have 
\begin{equation}
\|u_{0,\sigma_j(s)}^{\pm}(x)-u_{1,\sigma_j(s')}^{\pm}(x)\|\leq c_{27}(\sqrt{\|\sigma_j(s)-\sigma_j(s')\|}+\|c_j(s)-c_j(s')\|),
\end{equation}
where $i=1,2$ and $s,s'\in[0,1]$.
\end{Lem}
\begin{Rem}
Note that here $\Gamma_{i_1}$ and $\Gamma_{i_3}$ are not yet GTCs but only candidate ones. To avoid too many symbols involved, we still denote them by these without ambiguity.
\end{Rem}
With the help of this {\bf Regularity Lemma}, we can get the following genericity of isolated semi-static orbits, which is a skillful application of box-dimension.
\begin{Lem}{\bf (Genericity\cite{CY1,CY2,Ch})}\label{genericity}
For a fixed $i\in\{M,M+1,\cdots,\infty\}$, the system corresponding to $\mathbb{W}_{g_j}^i$ with $j=1,2$ is of a normal form $H=h+Z+R$. There exists an open and dense set $\mathcal{G}_i(R)$ contained in the domain $\mathcal{B}(0,c_{28})\subset C^r(TM\times\mathbb{S}^1,\mathbb{R})$ such that the system with $R$ in it satisfies the following:\\

{\bf For all $c_l(s)\in\Gamma_{i_l}$, there exists at least one heteroclinic orbit in $\mathcal{N}(c_j(s),\tilde{M})$ which is isolated.}\\

Here $l=1,3$, $s\in[0,1]$ and $c_{28}=c_{28}(i)$ is a proper constant depending on the underlying resonant line.
\end{Lem}

Based on all these preparations, we can finish our construction and get our main conclusion now.
\vspace{10pt}
\section{Proof of the Main Conclusion}
$\newline$

First, we can take proper $f(p,q,t)\in C^r(T^*\mathbb{T}^2\times\mathbb{S}^1,\mathbb{R})$ for system (\ref{2}), then transform it into (\ref{4}) with an exact symplectic transformation $\mathfrak{R}_f^{\infty}$ in a small neighborhood $\mathbb{D}_0$ of $\{0\}\times\mathbb{T}^2\times\mathbb{S}^1$. Here $f(q,p,t)$ can be chosen satisfying all the {\bf U*} and {\bf C*} conditions and denoted by $f_0=f(q,p,t)$. Then we get a connected cohomology set $\bigcup_{i=M}^{\infty}(\mathbb{W}_{g_1,f_0}^i\cup\mathbb{A}_{f_0}^i\cup\mathbb{W}_{g_2,f_0}^i)$, of which the wNHIC $N_{g_j,f_0}^i$ corresponding to $\mathbb{W}_{g_i,f_0}^i$ persists with certain length and the thickness of $\mathbb{A}_{f_0}^i$ can be uniformly estimated, $j=1,2$. This point is based on our analysis in section \ref{NHIC and Aubry} and \ref{Annulus}.\\

Second, we choose the candidate GTCs satisfying $\Gamma_{i_1,f_0}\subset\mathbb{W}_{g_2,f_0}^i$, $\Gamma_{i_2,f_0}\subset\mathbb{A}_{f_0}^i$ and $\Gamma_{i_3,f_0}\subset\mathbb{W}_{g_1,f_0}^i$. Then we can add a small perturbation $\epsilon_1\Delta f_1(p,q,t)$ to $f_0$ and $f_1=f_0+\epsilon_1\Delta f_1$. Here $\Delta f_1(q,p,t)\in\mathcal{B}^{r+r'}(0,1)\subset C^{r+r'}(T^*\mathbb{T}^2\times\mathbb{S}^1,\mathbb{R})$, $\epsilon_1 (0<\epsilon_1\ll1)$ is a small constant and $r'>0$ can be chosen arbitrarily large. From Lemma \ref{Fourier} we can see that the change of resonant term $Z_{\Delta f_1}$ of $H_1\doteq H+\epsilon\Delta f_1$ along the resonant plan $\mathfrak{P}_{f_1}$ will have much smaller Fourier coefficients, as long as $i\geq M\gg1$ and $r'$ is properly large. This point is very important for our construction: $H_1$ still satisfies the uniform restrictions {\bf U*} and {\bf C*}. Additionally, $\Gamma_{i_1,f_1}$ is now a GTC of $H_1$ from Lemma \ref{genericity}. Besides, the thickness of $\mathbb{A}_{f_1}^i$ can persist with just a $\epsilon_1\Delta_0$ decrease from Corollary \ref{overlap}. So $\Gamma_{i_2,f_1}\subset\mathbb{A}_{f_1}^i$ is also a GTC.\\

After above once perturbation with $\epsilon_1\Delta f_1$, we can then add another perturbation $\epsilon_2\Delta f_2$ to $H_1$ and get $H_2\doteq H_1+\epsilon_2\Delta f_2$. Here $0<\epsilon_2\ll\epsilon_1$ and $\Delta f_2\in\mathcal{B}^{r+r'}(0,1)$. Similarly, all the uniform restrictions are valid for system $H_2$ and $\Gamma_{i_1,f_2}\cup\Gamma_{i_2,f_2}\cup\Gamma_{i_3,f_2}$ is now a GTC contained in $\mathbb{W}_{g_1,f_2}^i\cup\mathbb{A}_{f_2}^i\cup\mathbb{W}_{g_2,f_2}^i$. So we have found a GTC for a whole transport process $\Gamma_i^{\omega}$.\\

Repeat above process and we get $H_{\infty}=H+\sum_{k=1}^{\infty}\epsilon_k\Delta f_k$. For this system $H_{\infty}$, $\bigcup_{i=M}^{\infty}(\Gamma_{i_1,f_{\infty}}\cup\Gamma_{i_2,f_{\infty}}\cup\Gamma_{i_3,f_{\infty}})$ then become a whole GTC and we finally find the asymptotic trajectories for $\mathcal{T}_{\omega}$ torus.

\section{Appendix}
\vspace{10pt}

\subsection{Introduction of Melnikov Method}
$\newline$

This part serves as a supplement of Proposition \ref{Melnikov}, which can be found of another version in \cite{Tr2}. But for the completeness of this paper, we list it in the following.\\

For a uncoupled system $H_0(x,q,y,p)=H_{0,1}(x,y)+H_{0,2}(q,p)$ with $(x,q,y,p)\in T^*\mathbb{T}^2$, we assume $(0,0,0,0)$ is the unique hyperbolic fixed point. Without loss of generality, we let $H_0(0,0,0,0)=0$. Let $H_{\epsilon}(x,q,y,p)=H_0+\epsilon H_1(x,q,y,p)$ be a perturbed system where $\epsilon\ll1$. Then we can still get a unique hyperbolic fixed point for $H_{\epsilon}$, which can be denoted by $(x^*(\epsilon),q^*(\epsilon),y^*(\epsilon),p^*(\epsilon))$. We can see that $(x^*(0),q^*(0),y^*(0),p^*(0))=(0,0,0,0)$. Also we can assume that $H_{\epsilon}(x^*(\epsilon),q^*(\epsilon),y^*(\epsilon),p^*(\epsilon))=0$.\\

As $H_0$ is a uncoupled system with two first integrals $H_{0,1}$ and $H_{0,2}$, we can find the generating functions $S^{u,s}_{0}(x,q)$ defined in a proper neighborhood $\Omega$ of $(0,0)\in\mathbb{T}^2$. Besides, we can assume that suspended in the universal covering space, the union of all the copies of $\Omega$ can cover $\mathbb{R}^2$. The trajectory with initial position $(x,q,\frac{\partial S^{u}_0(x,q)}{\partial x},\frac{S^u_0(x,q)}{\partial q})$ (or $(x,q,\frac{\partial S^{s}_0(x,q)}{\partial x},\frac{S^s_0(x,q)}{\partial q})$) will tend to $(0,0,0,0)$ as $t\rightarrow-\infty$ (or $t\rightarrow\infty$). Here $(q,x)\in\Omega$ is a fixed point. We can also see that the whole trajectory will lay on the graph $\{(q,x,\frac{\partial S^{u}_0(x,q)}{\partial x},\frac{S^u_0(x,q)}{\partial q})\big{|}(q,x)\in\Omega\}$ (or $\{(x,q,\frac{\partial S^{s}_0(x,q)}{\partial x},\frac{S^s_0(x,q)}{\partial q})\big{|}(q,x)\in\Omega\}$). \\

As $\epsilon\ll1$, we can still get the persistent generating functions $S_{\epsilon}^{u,s}$ for $H_{\epsilon}$ in the domain $\Omega$. Formally we take $S^{u,s}_{\epsilon}=S_0^{u,s}+\epsilon S_1^{u,s}+\mathcal{O}(\epsilon^2)$, then we have
\begin{eqnarray*}
0=&H_{\epsilon}(x,q,\nabla S^{u,s}_{\epsilon}(x,q))& \\
=&H_0(x,q,\nabla S_0^{u,s}+\epsilon\nabla S_1^{u,s})+\epsilon H_1(x,q,\nabla S_0^{u,s})+\mathcal{O}(\epsilon^2)\\
=&\epsilon(\frac{\partial H_0}{\partial y},\frac{\partial H_0}{\partial p})\big{|}_{(x,q,\nabla S_0^{u,s})}\cdot
\begin{pmatrix}
\frac{\partial S_1^{u,s}}{\partial x}\\\frac{\partial S_1^{u,s}}{\partial q}
\end{pmatrix}+\epsilon H_1(x,q,\nabla S_0^{u,s})+\mathcal{O}(\epsilon^2).\\
\end{eqnarray*}
We denote by $(x_0^{u,s}(t),q_0^{u,s}(t),y_0^{u,s}(t),p_0^{u,s}(t))$ the trajectory of $H_0$ with initial position $(x^{u,s}_0(0),q^{u,s}_0(0),y_0^{u,s}(0),p_0^{u,s}(0))=(x,q,\frac{\partial S_0^{u,s}}{\partial x},\frac{\partial S_0^{u,s}}{\partial q})$. As $H_0$ is uncoupled, these two trajectories can actually be joint into a whole $\gamma_0(t)=(x_0^{u,s}(t),q_0^{u,s}(t))$ with $t\in\mathbb{R}$. We can omit the superscript `u,s' for short.\\

The $\mathcal{O}(\epsilon)$ term of above formula is of a form:
\[
\frac{d}{dt}S_1^{u,s}(\gamma_0(t))+H_1(\gamma_0(t),\nabla S_0^{u,s}(\gamma_0(t)))=0.
\]
Then we take a path integral by
\[
\left\{
\begin{array}{cccccccccc}
S_1^u(x_0(t),q_0(t))\big{|}_{-\infty}^0=-\int_{-\infty}^0H_1(\gamma_0(t),\nabla S_0^{u,s}(\gamma_0(t)))dt,\vspace{5pt}\\
S_1^s(x_0(t),q_0(t))\big{|}^{\infty}_0=-\int_0^{\infty}H_1(\gamma_0(t),\nabla S_0^{u,s}(\gamma_0(t)))dt,\\
\end{array}
\right.
\]
where $(x_0(0),q_0(0))=(x,q)$. On the other side, we have $S_1^{u,s}(0,0)=0$. This is because we can make $S_{\epsilon}^{u,s}(x^*(\epsilon),q^*(\epsilon))\equiv0$ by adding a constant. Then we get 
\[
S_0^{u,s}(x^*(\epsilon),q^*(\epsilon))\sim\mathcal{O}(\epsilon^2),\quad\forall \epsilon\ll1,
\]
as $(x^*(\epsilon),q^*(\epsilon))=(0,0)+\epsilon(x_1^*,x_2^*)+\mathcal{O}(\epsilon^2)$ is formally valid. So we get the $\mathcal{O}(\epsilon)$ term $S_1^{u,s}(0,0)=0$.\\
 
With the help of above relationships, we get the Melnikov function by 
\[
M(x,q)\doteq S_1^u(x,q)=S_1^s(x,q)=-\int_{-\infty}^{\infty}H_1(\gamma_0(t),\nabla S_0^{u,s}(\gamma_0(t)))dt.
\]
We can easily see that this function is invariant with respect to the flow of $H_0$, i.e. $\nabla_{H_0}M\equiv0$. Since$H_0=H_{0,1}+H_{0,2}$, we then get $\nabla_{H_{0,1}}M\equiv-\nabla_{H_{0,2}}M$, $\forall (x,q)\in\mathbb{T}^2$.\\

For a fixed $(x,q)\in\mathbb{T}^2$ satisfying $\nabla_{H_{0,1}}M(x,q)=\nabla_{H_{0,2}}M(x,q)=0$, we can get a $2\times2$ matrix by $(\nabla_{H_{0,i}}\nabla_{H_{0,j}}M)_{i,j=1,2}$, whose rank is at most 1. If so, there must be a unique homoclinic point $(x(\epsilon),q(\epsilon))$ in a $\mathcal{O}(\epsilon)-$neighborhhod of $(x,q)$.

\subsection{A generalization of {\bf U2} condition}\label{U2}
$\newline$

In this subsection, we can generalize our condition {\bf U2} to a loose one. 
%The following conclusion is a direct cite of \cite{Ch}:
\begin{Lem}\cite{Ch}
Let $\{F_{\lambda}(x):\mathbb{T}\rightarrow\mathbb{R}\}$ be a family of $C^r-$functions ($r\geq4$) with $\lambda$ contained in $[\lambda_0,\lambda_1]$.  If $F_{\lambda}$ is Lipschitz continuous w.r.t. $\lambda$, then we can find an open dense set $\mathfrak{G}\subset C^r(\mathbb{T},\mathbb{R})$ such that $\forall\; V\in\mathfrak{G}$, the followings hold:
\begin{itemize}\vspace{5pt}
\item {\bf(ND)} $\forall\lambda\in[\lambda_0,\lambda_1]$, all the maximizers of $F_{\lambda}(x)$ are non-degenerate.\vspace{5pt}
\item $\forall\lambda\in[\lambda_0,\lambda_1]$, there exist at most two maximizers for $F_{\lambda}(x)$.\vspace{5pt}
\item there exist finitely many $\lambda_i\in(\lambda_0,\lambda_1)$ such that $F_{\lambda_i}(x)$ has two different maximizers.\vspace{5pt}
\end{itemize} 
\end{Lem} 
\begin{Rem}
We call such a $\lambda_i$ a {\bf bifurcation point}. If we modify the first bullet by the following:\\

{\bf (UND)} $\forall\lambda\in[\lambda_0,\lambda_1]$, all the maximizers of $F_{\lambda}(x)$ is uniformly non-degenerate with the eigenvalues not bigger than $-c^*<0$.\\

For our case, $p^{\sigma}[f]_1$ takes place of $F_{\lambda}$ with $p\in\mathscr{S}$ as a parameter. We can just take $c^*=\frac{c_5d_m^{\sigma}}{l^{m(r+2)}}$. Here `m'  reminds us of which resonant line we are considering. As we have showed that the set of functions satisfying {\bf U2} is non-empty, so the open dense property of above Lemma ensures that the functions satisfying these three bullets and {\bf (UND)} do exist.
\end{Rem}

\vspace{10pt}
$\newline$
{\bf Acknowledgement} This work is supported by NNSF of China (Granted 11171146, Granted 11201222), National Basic Research Program of China (973 Program, 2013CB834100) and Basic Research Program of Jiangsu Province (BK2008013).\\

The first author is grateful to C-Q Cheng for teaching his work in {\it a priori} stable Arnold Diffusion. Besides, he also thanks Ji Li for checking the details of this paper, and thanks Jinxin Xue for proposing the result of Wiggins to me, which does help me a lot in solving the 2-resonance difficulties. \\

Both the authors thank J. Cheng, W. Cheng, J. Yan, X. Cui, M. Zhou and all the other colleagues in Dynamical Systems seminar of Nanjing University. They really gave us several inspiring discussions in the process of this paper.


\begin{thebibliography}{9999}
\bibitem{Ar2} Arnold V. {\it Instability of dynamical systems with several degrees of freedom}, (Russian, English) Sov. Math. Dokl {\bf5}1964, 581-585; translation from Dokl. Akad. Nauk SSSR, {\bf156}(1964), 9-12.
\bibitem{Ar1} Arnold V.,Kozlov V. and Neishtadt A. {\it Mathematics Aspects of Classical and Celestial mechanics}, Dynamics Systems III, Encyclopedia of Mathematical Sciences, {\bf3} Springer-Verlag Berlin Heidelberg, (1988).
\bibitem{Be} Belitskii G. R. {\it Equivalence and normal forms of germs of smooth mappings}, Russ. Math. Surv., {\bf 33}, (1978) 107-177.
\bibitem{B} Bernard P. {\it Connecting orbits of time dependent Lagrangian systems}, Annales de l'institut Fourier, (2002) 1533-1568.
\bibitem{B2} Bernard P. {\it Symplectic aspects of Mather theory}, Duke Math. J. {\bf 136} (2007) 401-420.
\bibitem{B3} Bernard P. {\it The dynamics of pseudo graphs in convex Hamiltonian systems}. J. Amer. Math. Soc., {\bf 21}(3), 615-669, 2008.
\bibitem{Ber} Bernard P. {\it Large normally hyperbolic cylinder in a priori stable Hamiltonian systems}, Annales H. Poincare 11(2010) 929-942.
\bibitem{BKZ} Bernard P., Kaloshin V. and Zhang K. {\it Arnold diffusion in arbitrary degrees of freedom and crumpled 3-dimensional normally hyperbolic invariant cylinder}, arXiv: 1112.2773v1 (2011).
\bibitem{Bi} Birkhoff G. Collected Math Papers, vol.2, 462-465.
\bibitem{CD} Calvez P.\& Douady R. {\it Exemple de point fixe elliptique non topologiquement stable en dimension 4}, C.R.Acad.Sci.Paris, t.296,1983, 895-898.
\bibitem{Ca} Cannarsa P.\& Sinestrari C. {\it Semiconcave Functions, Hamilton-Jacobi Equations, and Optimal Control}, Birkh\"auser Boston, 2004.
\bibitem{Car} Carneiro M. J. D. {\it On minimizing measures of the action of autonomous Lagrangians}, Nonlinearity {\bf 8}(1995), 1077-1085.
\bibitem{Ch2} Cheng C-Q. {\em Minimal invariant tori in the resonant regions for nearly integrable Hamiltonian systems}, Trans Amer Math Soc, vol. 357, {\bf 12}, (2005) 5067-5095.
\bibitem{CY1} Cheng C-Q. and Yan J. {\it Existence of diffusion orbits in a priori unstable Hamiltonian systems}. J. Differential Geom. {\bf 67} (2004), 457-517.
\bibitem{CY2} Cheng C-Q. and Yan J. {\it Arnold Diffusion in Hamiltonian systems: apriori unstable case}. J. Differential. Geom. {\bf 82} (2009) 229-277.
\bibitem{Ch} Cheng C-Q. {\em Arnold Diffusion in nearly integrable Hamiltonian systems}, preprint (2012).
\bibitem{Con} Contreras G.\& Paternain P. {\it Connecting orbits between static classes for generic Lagrangian systems}, Topology, {\bf 41} (2002), 645-666.
\bibitem{Del} Delshams A., Guti\'errez P.\& Pacha J. {\it Transversality of homoclinic orbits to hyperbolic equilibria in a Hamiltonian system, via the Hamilton-Jacobi equation}, Phys. D, {\bf 243}(1), 2013, 64-85. 
\bibitem{D L S} Delshams A., de la Llave R. and Seara T. {\it Geometric mechanism for diffusion in Hamiltonian systems overcoming the large gap problem: heuristic and rigorous verification of a model}, Memoirs Amer. Math. Soc. {\bf 179} (844), (2006).
\bibitem{Du} Duady R. {\it Stabilit\'{e} ou instabilit\'{e} des points fixes elliptique}. Ann. Sci. Scole Norm. Sup. {\bf 21}(4) (1988), 1-46.
\bibitem{E} Ehrenfest T. {\it The Conceptual Foundations of the Statistical Approach in Mechanics}. Cornell University Press, Ithaca, (1959).
\bibitem{Fa} Fathi A. {\it Weak KAM Theorem in Lagrangian Dynamics}, Cambridge Studies in Advaced Mathematics, Cambridge University Press, (2009).
\bibitem{F} Fenichel N. {\it Geometric singular perturbation theory for ordinary differential equations}, J. Diff. Eqns., {\bf 31}, 53-98, (1979).
\bibitem{Fi} Figalli A.\& Rifford L. {\it Closing Aubry Sets I}, Comm. Pure Appl. Math., accepted, 2013.
\bibitem{GK} Guardia M. \& Kaloshin V. {\it Growth of Sobolev norms in the cubic defocusing nonlinear
Schr\"odinger equation}, arXiv:1205.5188, 2012.
\bibitem{Har} Hartman P. {\it On local homeomorphisms of Euclidean spaces}, Boletin de la Sociedad mathem\'atica, Mexicana, {\bf 2}5, (1960) 220-241.
\bibitem{HPS} Hirsch M. W., Pugh C.C. and Shub M. {\it Invariant manifolds}, Lecture notes in Math. Springer Berlin, New York, (1977).
\bibitem{KMV} Kaloshin V., Mather J.\& Valdinoci E. {\it Instability of resonant totally elliptic points of symplectic maps in dimension 4}, Analyse complex, systems dynamiques, sommabilite des series divergent et theories galoisiennes. II. Asterisque No. 297(2004), 79-116.
\bibitem{KS} Kaloshin V.\& Saprykina M. {\it An example of a nearly integrable Hamiltonian system with a trajectory dense in a set of maximal Hausdorff dimension}, CMP, (315)2012, {\bf (3)}643-697.
\bibitem{KZZ} Kaloshin V., Zhang K.\&Zheng Y. {\it Almost dense orbit on energy surface}, Proceedings of the XVIth ICMP, Prague, World Scientific, 2010,314-397.
\bibitem{K Z} Kaloshin V. \& Zhang K. {\it A strong form of Arnold Diffusion for two and a half degrees of freedom}, preprint (2013).
\bibitem{L} Li X., \& Cheng C-Q. {\it Connecting orbits of autonomous Lagrangian systems}, Nonlinearity {\bf 23}(2009) 119-141.
\bibitem{Lo} Lochak P. {\it Canonical perturbation theory via simultaneous approximation}, Russian Math. Surveys {\bf 47} (1992) 57-133.
\bibitem{Mar} Marco J. P. {\it Generic hyperbolic properties of classical systems on the torus $\mathbb{T}^2$}, preprint (2013).
\bibitem{Ms} Massart D. {\it On Aubry sets and Mather's action function}, Israel J. Math. {\bf 134}(2003) 157-171.
\bibitem{Mat2} Mather J. {\it Variational construction of orbits of twist diffeomorfisms}. J. Amer. Math. Soc. {\bf 4} (1991), 207-263.
\bibitem{Mat} Mather J. {\it Action minimizing invariant measures for positive definite Lagrangian systems}, Math. Z., {\bf 207}(2), (1991), 169-207. 
\bibitem{Mat3} Mather J. {\it Variational construction of connecting orbits}, Ann. Inst. Fourier (Grenoble). {\bf 43}(5), (1993), 1349-1386. 
\bibitem{Mat1} Mather J. {\it Arnold Diffusion I: Announcement of results}, J. Mathematical Sciences, {\bf 124}(5) (2004) 5275-5289. (Russian translation in Sovrem. Mat. Fundam. Napravl, {\bf 2} (2003) 116-130)
\bibitem{Mo} Moser J. {\it Selected Chapters in the Calculus of Variations}, 2003 Birkh\"auser Verlag.
\bibitem{Pa} Palis J.\& de Melo W. {\it Geometric Theory of Dynamical Systems: A Introduction}, Springer Verlag, 1982.
\bibitem{Po} Pollicott M. {\it Lectures on ergodic theory and Pesin theory on compact manifolds}, Cambridge University Press, 1993.
\bibitem{P} P\"{o}schel J. {\it A Lecture on the Classical KAM Theorem}, Proc. Symp. Pure Math. {\bf 69} (2001) 707-732.
\bibitem{P2} P\"{o}schel J. {\it Nekhoroshev Estimates for Quasi-convex Hamitonian Systems}, Math.Z. {\bf 213} (1993) 187-216.
\bibitem{Ru} R\"{u}ssmann H., Kleine N I. {\it \"{U}ber invariante Kurven differenzierbarer Abbildungen eines Kreispringes}. Nachr. Akad. Wiss. G\"{o}tingen. Math. Phys. Kl. (1970), 67-105.
\bibitem{Sa} Samovol V. S. {\it Equivalence of systems of differential equations in a neighborhood of a singular point}, Trans. Moscow Math. Soc., {\bf 44} (1982), 217-237.
\bibitem{Sh} Shilnikov L., Shilnikov A., Turaev D.\& Chua L. {\it Method of qualitative theory in nonlinear dynamics}, World Scientific Series on Nonlinear Science, Series a, Part 1, 1998.
\bibitem{Tao} Tao T. {\it Tranfer of energy to high frequencies in the cubic defocusing nonlinear Sch\"{o}dinger equation}, Invent Math(2010) 181:39-113.
\bibitem{Tr} Treschev D. {\it Evolution of slow variables in a priori unstable Hamiltonian systems}, Nonlinearity, {\bf 17} (2004) 1803-1841.
\bibitem{Tr2} Treschev D. Zubelevich O. {\it Introduction to the Perturbation Theory of Hamiltonian Systems}. Springer-Verlag Berlin heidelberg 2010.
\bibitem{W} Wiggins S. {\it Normally Hyperbolic Invariant Manifolds in Dynamical Systems}, Springer Verlag, Vol.105 App. Math. Sci., 1994.
\bibitem{Z} Zheng Y.\& Cheng C.-Q. {\it Homoclinic orbits of positive definite Lagrangian systems}, J. Diff. Eqns. {\bf 229} (2006) 297-316.
\end{thebibliography}
\end{document}